\documentclass{amsart}
\usepackage[fulloldstylenums]{kpfonts}

\usepackage{fullpage}

\usepackage{amsmath}
\usepackage{amsxtra}
\usepackage{amsthm}

\usepackage{mathrsfs}
\usepackage{tikz-cd}
\usepackage{mathtools}
\usepackage{mathpartir}
\usepackage{microtype}

\usepackage[bbgreekl]{mathbbol}
\usepackage{stmaryrd}
\usepackage{enumitem} 
\usepackage{xcolor}
\definecolor{darkgreen}{rgb}{0,0.45,0}
\definecolor{darkred}{rgb}{0.75,0,0}
\definecolor{darkblue}{rgb}{0,0,0.6}
\usepackage[pdfborder=0,colorlinks,citecolor=darkgreen,linkcolor=darkgreen,urlcolor=darkblue]{hyperref}
\usepackage{cleveref}
\usepackage{comment}

\usepackage[style=alphabetic,giveninits,natbib=true,maxalphanames=7,maxbibnames=99,backend=biber,backref=true]{biblatex}
\DeclareFieldFormat{postnote}{#1}
\DeclareFieldFormat{multipostnote}{#1}

\usetikzlibrary{decorations.markings}
\usetikzlibrary{decorations.pathmorphing}
\tikzset{
  module/.style={
    postaction={decorate},
    decoration={
      markings,
      mark=at position #1 with {\arrow{|}}}},
  module/.default=0.5,
  we/.style=
  { postaction={%
      decorate,
      decoration={
        markings,
        mark=at position #1 with {%
          \node[transform shape, yshift=.2em]{%
            \resizebox{0.5em}{!}{$\sim$}};}}}},
  we/.default=0.5,
  we'/.style=
  { postaction={%
      decorate,
      decoration={
        markings,
        mark=at position #1 with {%
          \node[transform shape, yshift=-.2em, rotate=180]{%
            \resizebox{0.5em}{!}{$\sim$}};}}}},
  we'/.default=0.5,
  iso/.style=
  { postaction={%
      decorate,
      decoration={
        markings,
        mark=at position #1 with {%
          \node[transform shape, yshift=.2em]{%
            \resizebox{0.5em}{!}{$\simeq$}};}}}},
  iso/.default=0.5,
  iso'/.style=
  { postaction={
      decorate,
      decoration={
        markings,
        mark=at position #1 with {%
          \node[transform shape, yshift=-.2em, rotate=180]{%
            \resizebox{0.5em}{!}{$\simeq$}};}}}},
  iso'/.default=0.5,
}

\tikzset{
  proarrow/.style={->, module},
  proequal/.style={-, double, module},
  prodotted/.style={->,dotted, module},
  prodashed/.style={->,dashed, module},
  wearrow/.style={->, we},
  wedashed/.style={->, dashed, we},
  wedotted/.style={->, dotted, we},
  tfibarrow/.style={->>, we=0.45},
  tfibdotted/.style={->>,dotted, we=0.45},
  tfibdashed/.style={->>,dashed, we=0.45},
  tcofarrow/.style={>->, we},
  tcofdashed/.style={>->, dashed, we},
  uwearrow/.style={->, we'},
  uwedashed/.style={->, dashed, we'},
  uwedotted/.style={->, dotted, we'},
  utfibarrow/.style={->>, we'=0.45},
  utfibdotted/.style={->>,dotted, we'=0.45},
  utfibdashed/.style={->>,dashed, we'=0.45},
  utcofarrow/.style={>->, we'},
  isoarrow/.style={->, iso},
  isodashed/.style={->, dashed, iso},
  uisoarrow/.style={->, iso'},
  uisodashed/.style={->, dashed, iso'},
  isocell/.style={=>, iso},
  isocelldashed/.style={=>, dashed, iso},
  uisocell/.style={=>, iso'},
  uisocelldashed/.style={=>, dashed, iso'}
}

\tikzcdset{
  center/.style={
    start anchor=center,
    end anchor=center,
  },
  phantomcenter/.style={
    phantom,
    center,
  },
}

\tikzcdset{
  adj/.style = {phantom, start anchor=center, end anchor=center, "\scriptstyle\bot"{#1}},
}

\setlist{}
\setenumerate{leftmargin=*,labelindent=0\parindent}
\setitemize{leftmargin=\parindent}

\newlist{parts}{enumerate}{2}
\setlist[parts]{label={(\roman*)},ref={(\roman*)}}
\crefname{partsi}{part}{parts}

\newlist{conditions}{enumerate}{2}
\setlist[conditions]{label={(\roman*)},ref={(\roman*)}}
\crefname{conditionsi}{condition}{conditions}

\numberwithin{equation}{subsection}

\crefformat{section}{\S#2#1#3}
\crefmultiformat{section}{\S\S#2#1#3}{ and~#2#1#3}{, #2#1#3}{, and~#2#1#3}

\usepackage{thmtools}

\newcommand{\defthm}[3]{
  \declaretheorem[
    sibling=equation,
    title={#2},
    refname={{#2},{#3}},
  ]{#1}
}

\defthm{thm}{Theorem}{Theorems}
\defthm{lem}{Lemma}{Lemmas}
\defthm{prop}{Proposition}{Propositions}
\defthm{cor}{Corollary}{Corollaries}
\defthm{conj}{Conjecture}{Conjectures}

\theoremstyle{definition}
\defthm{defn}{Definition}{Definitions}
\defthm{ex}{Example}{Examples}
\defthm{cons}{Construction}{Constructions}
\defthm{nex}{Non-Example}{Non-Examples}
\defthm{ntn}{Notation}{Notations}

\theoremstyle{remark}
\defthm{rmk}{Remark}{Remarks}
\defthm{exc}{Exercise}{Exercises}
\defthm{dig}{Digression}{Digressions}
\defthm{war}{Warning}{Warnings}
\defthm{rec}{Recollection}{Recollections}
\defthm{apo}{Apology}{Apologies}
\defthm{cav}{Caveat}{Caveats}

\newcommand{\op}{\mathup{op}}
\newcommand{\co}{\mathup{co}}
\newcommand{\id}{\mathup{id}}
\newcommand{\ob}{\mathup{ob}}

\newcommand{\const}{\mathup{const}}
\newcommand{\res}{\mathup{res}}
\newcommand{\ev}{\mathup{ev}}

\newcommand{\last}{\mathup{last}}

\newcommand{\lan}{\mathup{lan}}

\newcommand{\obj}{\mathup{obj}}
\newcommand{\dom}{\mathup{dom}}
\newcommand{\cod}{\mathup{cod}}
\newcommand{\colim}{\mathup{colim}}
\newcommand{\hocolim}{\mathup{hocolim}}

\newcommand{\cat}[1]{\mathup{\mathsf{#1}}}

\newcommand{\alg}[1]{{#1}\text{-}\cat{Alg}}

\newcommand{\CC}{\mathfrak{C}}
\newcommand{\EE}{\mathfrak{E}}
\newcommand{\FF}{\mathfrak{F}}
\newcommand{\GG}{\mathfrak{G}}

\newcommand{\LL}{\mathfrak{L}}

\newcommand{\NN}{\mathbb{N}}

\newcommand{\TF}{\mathfrak{TF}}

\DeclareMathAlphabet{\mathbbe}{U}{bbold}{m}{n}

\newcommand{\cA}{\mathcal{A}}
\newcommand{\cB}{\mathcal{B}}
\newcommand{\cC}{\mathcal{C}}
\newcommand{\cD}{\mathcal{D}}
\newcommand{\cE}{\mathcal{E}}

\newcommand{\cK}{\mathcal{K}}

\newcommand{\cM}{\mathcal{M}}
\newcommand{\cN}{\mathcal{N}}

\newcommand{\cS}{\mathcal{S}}

\newcommand{\DDelta}{\mathbbe{\Delta}}

\newcommand{\leftclass}{\mathscr{L}}
\newcommand{\rightclass}{\mathscr{R}}

\newcommand{\leibEHom}[2]{\widehat{\EHom}({#1},{#2})}
\newcommand{\leibEHomC}[2]{\widehat{\EHom}({#1},{#2})}

\newcommand{\leibEHomNC}[2]{\widehat{\EHom_{N\cC}}({#1},{#2})}
\newcommand{\leibEHomND}[2]{\widehat{\EHom_{N\cD}}({#1},{#2})}

\newcommand{\Funpseudo}{\mathcal{F}\!\cat{un}_\pseudo}
\newcommand{\Funlax}{\mathcal{F}\!\cat{un}_\lax}
\newcommand{\Funoplax}{\mathcal{F}\!\cat{un}_\oplax}

\newcommand{\slice}[2]{{{#1}_{/#2}}}
\newcommand{\Ws}{W^{\mathup{s}}}

\newcommand{\coslice}[2]{{{}^{#2/}\!{#1}}}
\newcommand{\Wc}{W^{\mathup{c}}}

\newcommand{\cart}[1]{{#1}^{[1]}_{\mathup{cart}}}

\font\maljapanese=dmjhira at 2ex 
\def\yo{\textrm{\maljapanese\char"48}}

\newsavebox{\mybox}
\newcommand{\scaledreflect}[1]{%
  \ThisStyle{\ifmmode%
    \savebox{\mybox}{$\SavedStyle#1$}%
    \reflectbox{\usebox{\mybox}}%
  \else%
    \savebox{\mybox}{#1}%
    \reflectbox{\usebox{\mybox}}%
  \fi%
}}

\makeatletter
\def\makeslashed#1#2#3#4#5{#1{\mathpalette{\sla@{#2}{#3}{#4}}{#5}}}

\def\@mathlower#1#2#3{\setbox0=\hbox{$\m@th#2#3$}\lower#1\ht0\box0}
\def\mathlower#1#2{\mathpalette{\@mathlower{#1}}{#2}}
\makeatother

\newcommand{\fto}{\twoheadrightarrow}
\newcommand{\wto}{\xrightarrow{{\smash{\mathlower{0.8}{\sim}}}}}

\newcommand{\type}[1]{{\mathsf{#1}}}
\newcommand{\term}[1]{{\mathsf{#1}}}

\newcommand{\Id}{\mathsf{Id}}
\newcommand{\Eq}{\mathsf{Eq}}

\newcommand{\isContr}{\mathsf{isContr}}

\renewcommand{\hom}{\type{hom}}
\newcommand{\Fun}{\mathsf{F}\!\mathsf{un}}
\newcommand{\Arr}{\mathsf{A}\!\mathsf{rr}}

\newcommand{\Ucov}{{U_{\mathsf{cov}}}}
\newcommand{\Ucocart}{{U_{\mathsf{cocart}}}}
\newcommand{\Ufib}{{U_{\mathsf{fib}}}}
\newcommand{\EUcov}{{\widetilde{U}_{\mathsf{cov}}}}
\newcommand{\EUcocart}{{\widetilde{U}_{\mathsf{cocart}}}}
\newcommand{\EUfib}{{\widetilde{U}_{\mathsf{fib}}}}
\newcommand{\picov}{\pi_{\mathsf{cov}}}
\newcommand{\pifib}{\pi_{\mathsf{fib}}}
\newcommand{\Magma}{\mathsf{Magma}}

\DeclareMathOperator{\arrtofun}{\term{arr-fun}}
\DeclareMathOperator{\funtoarr}{\term{fun-arr}}
\DeclareMathOperator{\arrtofunwgt}{\term{arr-fun-wgt}}
\DeclareMathOperator{\funtoarrwgt}{\term{fun-arr-wgt}}

\newcommand{\mB}{\mathscr{B}}
\newcommand{\mC}{\mathscr{C}}
\newcommand{\mD}{\mathscr{D}}

\newcommand{\iE}{\mathcal{E}}
\newcommand{\mE}{\mathscr{E}}
\newcommand{\sE}{\textit{s}\iE}
\newcommand{\smE}{\cat{s}\mE}
\newcommand{\iS}{\mathcal{S}}

\newcommand{\sS}{\textit{s}\iS}

\newcommand{\Cat}{\mathcal{C}\cat{at}}
\newcommand{\Set}{\mathcal{S}\cat{et}}

\newcommand{\sSet}{\cat{s}\mathcal{S}\cat{et}}

\newcommand{\Map}{\mathup{Map}}
\newcommand{\Hom}{\mathup{Hom}}
\newcommand{\intHom}{\mathup{exp}}
\newcommand{\EHom}{{\mE\Hom}}
\newcommand{\pseudo}{\mathup{pseudo}}
\newcommand{\lax}{\mathup{lax}}
\newcommand{\oplax}{\mathup{oplax}}

\newcommand{\el}[1]{{\smallint\!{#1}}}

\newcommand{\complete}{\mathsf{compl.}}
\newcommand{\cocomplete}{\mathsf{cocompl.}}
\newcommand{\cont}{\mathsf{cont.}}
\newcommand{\cocont}{\mathsf{cocont.}}
\newcommand{\Psh}[2][]{\mathsf{Psh}_{#1}({#2})}

\newcommand{\Prof}{\mathsf{Psh}_{\mathsf{ladj}}}

\begin{document}

\title{Directed univalence for simplicial objects in an \texorpdfstring{$\infty$}{infinity}-topos}
\author{Evan Cavallo}
\author{Emily Riehl}
\author{Christian Sattler}
\date{\today}

\thanks{This project was first instantiated at the 2017 Mathematics Research Community in Homotopy Type Theory coordinated by the AMS and supported by the National Science Foundation under grant number DMS-1321794.
  Over the long gestation period of this project, the first author has been supported by the US Air Force Office of Scientific Research under award numbers FA9550-15-1-0053 and FA9550-19-1-0216 and by the Knut and Alice Wallenberg Foundation (KAW) under grant numbers 2020.0266 and 2019.0116.
  The second author has been supported by the National Science Foundation via the grants DMS-2204304 and DMS-2507077, the Simons Foundation (920415, Riehl), by the US Air Force Office of Scientific Research under award number FA9550-21-1-0009, by the US Army Research Office under MURI Grant W911NF-20-1-0082, and by the Johns Hopkins President's Frontier Award.
  The third author has been supposed by the Swedish Research Council under grant number 2019-03765 and US Air Force Office of Scientific Research under award number FA9550-24-1-0302.}

\email{evan.cavallo@gu.se}

\email{eriehl@jhu.edu}

\email{sattler@chalmers.se}

\begin{abstract}
  A fundamental component of homotopy type theory, a synthetic theory of $\infty$-groupoids, is Voevodsky's univalence axiom.
  Univalence characterizes the identity types in the universal fibration, a classifier for small type families: identity types in the universe are equivalent to types of equivalences.
  The directed univalence axiom plays a similar foundational role in simplicial type theory, a synthetic theory of $\infty$-categories.
  In its original form, which does not include universes or directed univalence, the simplicial type theory has semantics in categories of simplicial objects in an $\infty$-topos, with synthetic $\infty$-categories corresponding to internal $\infty$-categories.
  We verify that directed univalence holds in this semantic setting, constructing an equivalence between hom types in the universal left fibration and function types.
  In fact, we verify a higher version of this result, constructing an equivalence between homotopy coherent composites in the universal left fibration and composable sequences of functions between types.
  Using the technique of weighted limits, we reduce this theorem for simplicial objects in an arbitrary $\infty$-topos to calculations ``on the left'' with simplicial sets.
\end{abstract}

\maketitle

\setcounter{tocdepth}{2}
\tableofcontents

\section{Introduction}

\subsection{The univalence axiom in homotopy type theory}
In \cite{HS}, Hofmann and Streicher confirmed that identities between terms in Martin-L\"{o}f's intensional dependent type theory \cite{PML} can be interpreted as structure rather than property, constructing a model in which types are interpreted as groupoids and identities within a type as isomorphisms within a groupoid.
Hofmann and Streicher's discovery motivated a search for higher-dimensional homotopical models of type theory.
Awodey and Warren \cite{AW} and subsequently Van den Berg and Garner \cite{van-den-berg-garner:12} interpreted identity types using weak factorization systems, following an intuition that the identity type family is a path object in something like a Quillen model category.
In parallel, Voevodsky was developing a homotopical interpretation of intensional dependent type theory motivated by the idea that types should be thought of as homotopy types, now commonly referred to as $\infty$-\emph{groupoids} (following Grothendieck \cite{grothendieck2021pursuingstacks}) or \emph{anima} (following \v{C}esnavi\v{c}ius and Scholze \cite{CesnaviciusScholze}).
Voevodsky reified the structural interpretation of identity types with his \emph{univalence axiom} for a universe $U$ \cite{PelayoWarren}, which asserts that a canonically-defined map from identities $A =_U B$ between types in $U$ to equivalences $A \simeq B$ is an equivalence.

Voevodsky demonstrated the consistency of his axiom with a homotopical model of intensional dependent type theory in Quillen's model category of simplicial sets, interpreting type families as Kan fibrations.
This work was ultimately published as \cite{KapulkinLumsdaine2021} with both expository and technical contributions made by Kapulkin and Lumsdaine.
Following Voevodsky, they interpret a universe of small types by a universal Kan fibration $\pi \colon \EUfib \to \Ufib$ classifying small Kan fibrations.
The univalence axiom is verified by showing that a canonically defined map from the path object of $\Ufib$ to a universal family of equivalences between small types is a weak equivalence in Quillen's model structure.
In 2019, Shulman, building upon several years of developments, showed that type theory with univalent universes has semantics in any $\infty$-topos \cite{shulman}.
More precisely, he introduces the axiomatic language of a \emph{type-theoretic model topos} to describe structure on a model-categorical presentation of an $\infty$-topos sufficient to induce an interpretation of type theory with strict univalent universes, then shows that every $\infty$-topos admits such a presentation.

The univalence axiom has a number of useful consequences for \emph{homotopy type theory} or \emph{univalent foundations} --- Martin-L\"{o}f dependent type theory augmented with the univalence axiom.
One important consequence is a family of theorems that goes by the name of the \emph{structure identity principle}.
These theorems characterize identities in types of mathematical structures built on top of the universe.
For instance, the type of groups is the type of types that are 0-truncated (aka sets) and equipped with a binary operation that is associative, unital, and has inverses.
It follows from univalence that identity types in the type of groups are equivalent to types of group isomorphisms, exactly as would be desired observationally.

\subsection{The directed univalence axiom in simplicial type theory}

Hofmann and Streicher's groupoid model suggests a second direction of generalization: some form of ``directed'' type theory in which types are to be interpreted as categories of some sort and the identity type family is complemented by a second family of hom types.
Such theories have been proposed by Licata and Harper \cite{licata-harper:11}, Nuyts \cite{nuyts:15}, Riehl and Shulman \cite{RS}, North \cite{north:19}, Weaver and Licata \cite{weaver-licata}, Ahrens, North, and Van der Weide \cite{ahrens-north-van-der-weide:23}, and Neumann and Altenkirch \cite{neumann-altenkirch:25}.
Riehl and Shulman's directed type theory, now known as \emph{simplicial type theory}, aims to be a synthetic language for $\infty$-categories, extending homotopy type theory's synthetic language for $\infty$-groupoids.
In an $\infty$-category, composition is only well-defined up to a contractible space of choices.
When chosen composites are needed in a construction, one also needs higher dimensional data to encode homotopy coherence with respect to the choice.
This bookkeeping disappears when $\infty$-category theory is developed internally to homotopy type theory, where a type has a unique inhabitant just when it is contractible.

The intended semantics of the simplicial type theory is in the category of bisimplicial sets, or more precisely in the Reedy model structure on simplicial objects, defined relative to Quillen's model structure on simplicial sets.
This is a model of homotopy type theory \cite[6.4]{shulman-reedy} and interacts with Rezk's complete Segal space model of $\infty$-categories \cite{rezk-CSS}.
A closed type corresponds to a fibrant object in the model structure, that is, a Reedy fibrant simplicial space.
A type defines a pre-$\infty$-category or \emph{Segal type} if binary composition exists uniquely, a condition that captures the Segal spaces in the bisimplicial sets model.
A Segal type defines an $\infty$-category or \emph{Rezk type} if its identity types are equivalent to its types of isomorphisms, capturing Rezk's complete Segal spaces.

While not addressed in the original paper \cite{RS}, it is natural to wonder about a directed analogue of the univalence axiom, characterizing the hom types in a universe.
Nuyts \cite[3.8.22]{nuyts:15} proposes one such axiom, characterizing homs between types $A$ and $B$ in a universe as functions $A \to B$.
For the simplicial type theory, we must first decide: which universe?
The universe of all (small) types, interpreted by the universal Reedy fibration, is not directed univalent in this sense; its homs are more like spans than functions.\footnote{It is unclear to us exactly what structure these spans carry, and whether the homs in this universe can be characterized internally at all.}

Here, we consider the \emph{covariant universe} interpreted by the universal \emph{left fibration}.
A left fibration has \emph{discrete} fibers ($\infty$-groupoids, rather than $\infty$-categories) and depends covariantly functorially on the arrows in the base.\footnote{In \cite{RS}, such maps are called \emph{covariant fibrations}.}
As an example, the family of hom types with a fixed source object over a Segal type defines a left fibration over that Segal type.
Lurie famously proved a ``straightening--unstraightening'' theorem, expressed as a Quillen equivalence between model categories, demonstrating an equivalence between left fibrations with a fixed base $\Gamma$ and functors from $\Gamma$ into the $\infty$-category $\iS$ of spaces, presented by the simplicially enriched category of Kan complexes and all maps between them \cite[2.2.1.2]{lurie-topos}.
While Lurie does not construct a universal left fibration in the strict sense required to model a type-theoretic universe, he does show --- via the case $\Gamma = \Delta^1$ --- that left fibrations over $\Delta^1$ are equivalent to functions between spaces (that is, \emph{functors} between $\infty$-groupoids).

Around 2017--2018, we reported on work in progress towards a verification of the directed univalence axiom for the covariant universe in the bisimplicial sets model of the simplicial type theory; we circulated our notes among this to a few colleagues who have cited this ``private correspondence.''
One motivation for the present paper was a sense of obligation to make this private correspondence publicly available.
Further motivation was provided by developments by our colleagues in the interim that enable us to prove a much stronger theorem than we had originally.

The original \cite{RS} describes the motivating semantics of the simplicial type theory in simplicial spaces and sketches more general semantics in \emph{model categories with shapes} \cite[\S A.2]{RS}, which are shown to have a pseudo-stable coherent tope logic and pseudo-stable extension types.
The authors then claim that the coherence methods of Lumsdaine and Warren \cite{LumsdaineWarren:2015} can promote these to strictly-stable structures.
Weinberger \cite{Weinberger} confirmed this and, with the help of Shulman's models of homotopy type theory \cite{shulman}, extended the semantics of simplicial type theory to categories of simplicial objects in $\infty$-topoi.
These extended semantics pair nicely with further developments of synthetic $\infty$-category theory.
Martini and Wolf develop the theory of \emph{internal $\infty$-categories} defined relative to an arbitrary $\infty$-topos \cite{martini}.
These turn out to be simplicial objects satisfying the Segal and Rezk conditions of the simplicial type theory \cite{RS}.
Thus, the theorems of simplicial type theory are valid for internal $\infty$-categories in any $\infty$-topos.

\subsection{Our results}

Here we show that directed univalence holds for the universal left fibration not just in the bisimplicial sets model but in simplicial objects in any $\infty$-topos.

In \S\ref{sec:topoi} we review the notion of $\infty$-topos as developed by Lurie \cite{lurie-topos} and Rezk \cite{rezk}, with a focus on their model categorical presentations as axiomatized first by Rezk as \emph{model topoi} and later by Shulman as \emph{type-theoretic model topoi} to the end of describing the semantics of homotopy type theory. We refer to a type-theoretic model topos as \emph{fibration-extensive} if the fibrations (and trivial fibrations) are closed under coproduct and show that any $\infty$-topos is presented by a fibration-extensive type-theoretic model topos.

To state our directed univalence theorem, we must first construct in \S\ref{sec:universe} the universal left fibration classifying small left fibrations $\pi \colon \EUcov \to \Ucov$ as a morphism in the $\infty$-topos $\sE$ of simplicial objects in an $\infty$-topos $\iE$.
We take advantage of general machinery developed by Shulman for the purpose of constructing universal fibrations in $\infty$-topoi.
Shulman himself notes the applicability to constructing a universal left fibration in the setting of bisimplicial sets \cite[5.23]{shulman}.

The internal statement of Voevodsky's univalence axiom requires the construction of a universal family of equivalences; for directed univalence, we need an object of arrows.
In \S\ref{sec:universe}, we define an internal category whose base object is $\Ucov$ and whose object of arrows we denote by $\Fun_1$, which is the base of a universal family of functions between types classified by $\Ucov$.
The other objects in this internal category diagram define classifiers $\Fun_n$ for sequences of $n$ composable functions between objects classified by $\Ucov$.

Our main theorem, proven in \S\ref{sec:dua-types}, establishes an equivalence $\Ucov^{\Delta^1} \simeq \Fun_1$ identifying hom types in the universe with function types.
In fact, we prove a higher-dimensional directed univalence result, identifying homotopy coherent diagrams of $n$ composable arrows with sequences of $n$ composable functions.

{
\renewcommand{\thethm}{\ref{thm:dua-types}}
\begin{thm}[directed univalence]
  Let $\iE$ be an $\infty$-topos. The maps
  \[ \begin{tikzcd} \Fun_\bullet \arrow[r, "\funtoarr_\bullet", shift left=.25em] & \Ucov^{\Delta^\bullet} \arrow[l, "\arrtofun_\bullet", shift left=.25em] \end{tikzcd}\]
define a pointwise equivalence between the simplicial objects $\Fun_\bullet$ and $\Ucov^{\Delta^\bullet}$ in $\sE$ that is natural up to homotopy.
\end{thm}
\addtocounter{thm}{-1}
}

The $n$-ary equivalences $\Fun_n \simeq \Ucov^{\Delta^n}$ are needed for applications, also discussed in \S\ref{sec:dua-types}.\footnote{
  More specifically, we need the theorem for $0 \le n \le 2$ (including homotopy naturality and coherence of naturality for such $n$) in order to prove that $\Ucov$ is Segal.
  Once we know that $\Ucov$ is Segal, the equivalences $\Fun_n \simeq \Ucov^{\Delta^n}$ for $n > 2$ could be derived using the equivalence between simplices and spines in $\Ucov$.
  However, because the case $n \le 2$ contains all the complexity of the general case, we instead prove the theorem directly for all $n$.
}
We use \cref{thm:dua-types} to demonstrate that $\Ucov$ is an internal $\infty$-category, satisfying the Segal and Rezk conditions, namely the internal $\infty$-category of internal $\infty$-groupoids and functions between them.
We leverage these results to construct other internal $\infty$-categories from this universe, whose hom-types we characterize as instances of what we call the \emph{structure homomorphism principle}.
We illustrate this general method with the example of the internal $\infty$-category of magmas.

Our proof of Theorem \ref{thm:dua-types} constructs an explicit equivalence as follows.
By the classifying properties of $\Fun_n$ and $\Ucov^{\Delta^n}$, the maps $\funtoarr_n$ and $\arrtofun_n$ express a correspondence between sequences of $n$ composable morphisms between left fibrations over a given base $\Gamma \in \sE$ and left fibrations over $\Delta^n \times \Gamma$.
Indeed, we construct them by exhibiting such a correspondence, in the form of point-set level functors
\[ \begin{tikzcd} \slice{\smE}{\Delta^n \times \Gamma} \arrow[r, bend left, "\arrtofun^\Gamma_n" above] & (\slice{\smE}{\Gamma})^{[n]} \arrow[l, bend left, "\funtoarr^\Gamma_n"] \end{tikzcd}\]
defined at the level of the type-theoretic model topos $\smE$ of simplicial objects in a type-theoretic model topos $\mE$ presenting an $\infty$-topos $\iE$.
Here $\Gamma \in \smE$ is an arbitrary object, allowing us to implement these constructions in any context; $\Delta^n \in \smE$ is the discrete embedding of the simplicial set $\Delta^n$, which parametrizes homotopy coherent diagrams of $n$ composable arrows; and $[n]$ is the ordinal category, which parametrizes diagrams of $n$ composable morphisms.

At this level, we abstract from $[n]$ and $\Delta^n$ to an arbitrary 1-category $\cC$ and its simplicial nerve $N\cC$.
The homotopical properties of these point-set level constructions are summarized in a point-set level directed univalence theorem, which we prove in \S\ref{sec:dua-model} (where a more precise statement of our result may be found).

{
\renewcommand{\thethm}{\ref{thm:dua-model}}
\begin{thm}[directed univalence]
Let $\mE$ be a fibration-extensive type-theoretic model topos. For $\Gamma \in \smE$ and a 1-category $\cC$ the functors
\[ \begin{tikzcd} \slice{\smE}{N\cC \times \Gamma} \arrow[r, bend left, "\arrtofun^\Gamma_\cC" above] & (\slice{\smE}{\Gamma})^\cC \arrow[l, bend left, "\funtoarr^\Gamma_\cC"] \end{tikzcd}\]
define right Quillen functors between the covariant model structure on $\slice{\smE}{N\cC \times \Gamma}$ and the projective covariant model structure on $(\slice{\smE}{\Gamma})^\cC$. Moreover, these functors are weak equivalence inverses on fibrant objects and the mappings $(\cC,\Gamma) \mapsto \arrtofun^\Gamma_\cC$ and $(\cC,\Gamma) \mapsto \funtoarr^\Gamma_\cC$ are respectively lax and pseudonatural in both variables.
\end{thm}
\addtocounter{thm}{-1}
}

In summary, the maps $\arrtofun_n$ and $\funtoarr_n$ of \cref{thm:dua-types} are $\infty$-topos level avatars of the right Quillen functors $\arrtofun^\Gamma_\cC$ and $\funtoarr^\Gamma_\cC$, which we define for an arbitrary indexing context $\Gamma$ and indexing category $\cC$.
These functors are in turn constructed using weighted limits from functors defined on the category of simplicial sets:\footnote{The relative versions of our functors, defined over a context $\Gamma \in \smE$, can be understood as pullbacks of the global versions considered here.}
\[
  \begin{tikzcd}[column sep=large, row sep=tiny] \slice{\sSet}{N\cC} & \sSet^\cC \arrow[l, "{\arrtofunwgt}"'] \rlap{,} & \sSet^\cC & \slice{\sSet}{N\cC} \arrow[l, "\funtoarrwgt"'] \rlap{,} \\
  \slice{\smE}{N\cC} \arrow[r, "\arrtofun"'] & \smE^\cC \rlap{,} & \smE^\cC \arrow[r, "\funtoarr"'] & \slice{\smE}{N\cC} \rlap{.}
  \end{tikzcd}
\]
We think of the simplicial set level functors as the ``left adjoints'' to the type-theoretic model topos level functors, defined relative to $\mE$-valued hom functors that we also describe in \S\ref{sec:dua-model}:
\[
\begin{tikzcd}[column sep=4em]
  (\sSet^\cC)^\op \times \smE^\cC
  \ar[r, "{\EHom}"]
&
  \mE
  \rlap{,}
&
  (\slice{\sSet}{K})^\op \times \slice{\smE}{K} \arrow[r, "\EHom_{K}"]
&
  \mE
  \rlap{.}
\end{tikzcd}
\]
Accordingly, we refer to $\arrtofun$ and $\funtoarr$ as the $\mE$-\emph{relative right adjoints} of $\arrtofunwgt$ and $\funtoarrwgt$.

The homotopical properties of $\arrtofun$ and $\funtoarr$ that we use to prove \cref{thm:dua-model} are consequences of homotopical properties of $\arrtofunwgt$ and $\funtoarrwgt$, which we establish in \S\ref{sec:weights} where these functors are defined.
The main technical work of this paper is contained in this section, which includes a review of some classical results from categorical homotopy theory.
We define the functors $\arrtofunwgt$ and $\funtoarrwgt$ using weighted colimits relative to a very classical pair of weights:\footnote{The weights $\Ws$ and $\Wc$ are the opposites of the weights used by Bousfield and Kan to define the homotopy limit and homotopy colimit of a $\cC$-indexed diagram, respectively \cite[XI.2.2]{bousfield-kan}.}
\[
\begin{tikzcd}[column sep=large, row sep=small]
\arrtofunwgt \colon \sSet^\cC \arrow[r, "{(N\cC)^*}"] & (\slice{\sSet}{N\cC})^\cC \arrow[r, "\Wc\otimes_\cC-"] & \slice{\sSet}{N\cC} \rlap{,}
\\
\funtoarrwgt \colon \slice{\sSet}{N\cC}  \arrow[r, "\Ws\otimes -"] & (\slice{\sSet}{N\cC})^\cC \arrow[r, "(N\cC)_!"] & \sSet^\cC \rlap{.}
\end{tikzcd}\]
where
\[
\begin{tikzcd}[row sep=tiny]
\cC \arrow[r, "\Ws"] & \slice{\sSet}{N\cC} \rlap{,} & \cC^\op \arrow[r, "{\Wc}"] & \slice{\sSet}{N\cC} \rlap{,}
\\
x \arrow[r, maps to] & N(\slice{\cC}{x}) \rlap{,} & x \arrow[r, maps to ] & N(\coslice{\cC}{x}) \rlap{.}
\end{tikzcd}
\]
The $\mE$-valued hom functors are each instances of weighted limit bifunctors, the latter via the equivalence between slice categories and presheaf categories, which extends to the $\mE$-valued setting because type-theoretic model topoi are infinitarily extensive.
Thus, the functors $\arrtofun$ and $\funtoarr$ are the weighted limit $\mE$-relative right adjoints of the weighted colimit functors $\arrtofunwgt$ and $\funtoarrwgt$.
We summarize the fact that the technical work can all be done on the weights side with the slogan ``it's all in the weights.''
In particular, results about left fibrations, which are defined to be maps in $\smE$ that are fibrations at the model category level and internally orthogonal to the inclusion $\iota_0 \colon \Delta^0 \to \Delta^1$ at the $\infty$-category level, reduce to results about left anodyne maps of simplicial sets.\footnote{The left anodyne maps are the left class of a weak factorization system whose right class are left fibrations of simplicial sets. Note these left fibrations are not a special case of the left fibrations referenced elsewhere because $\Set$ is a 1-topos, not an $\infty$-topos.}

\subsection{Future work}

By dualizing, our main theorems also provide directed univalence results for universal right fibrations. In future work, we plan to apply similar techniques to prove a corresponding directed univalence result for \emph{cocartesian fibrations} and \emph{cartesian fibrations}, which describe families of $\infty$-categories (rather than $\infty$-groupoids) varying covariant or contravariantly functorially, respectively, over arrows in the base.
An alternative approach would follow Sattler and W\"arn \cite{sattler-warn:26} and derive directed univalence for cocartesian fibrations in a model-independent fashion from the directed univalence for left fibrations we establish here.

\subsection{Related work}

\subsubsection{Directed type theory}

To our knowledge, the first recorded proposal for a directed univalence axiom is in the master's thesis of Andreas Nuyts \cite[3.8.22]{nuyts:15}.
Nuyts states the axiom within his directed type theory, but does not consider questions of semantics.
We began investigating directed univalence for simplicial type theory at the 2017 Mathematics Research Communities workshop in Snowbird, Utah, within a group focused on directed type theory.
The second author shared our preliminary findings, an interpretation of directed univalence in bisimplicial sets, in a follow-up presentation at the Joint Mathematics Meetings in 2018.
At this time we described an equivalence $\Fun_1 \simeq \Ucov^{\Delta^1}$.
Our $\arrtofun$ was the same as we present here, but our $\funtoarr$ was based on a weighted colimit construction rather than a weighted limit.

Weaver and Licata \cite{weaver-licata,weaver:24} subsequently created a cubical directed type theory with directed univalence and a model in bicubical sets.
By using cubical sets in place of simplicial sets, they define the model in a constructive metatheory (cf.~\cite{CCHM,ABCHFL}).
Their model should present an $\infty$-topos $\iE^{\square^\op}$ of cubical objects where $\iE$ is the $\infty$-topos presented by the underlying model of homotopy type theory in cubical sets.
For the cube category they use, it is an open question whether $\iE$ is the $\infty$-topos of $\infty$-groupoids \cite{streicher-weinberger}, but their construction should adapt to other underlying models that do present $\infty$-groupoids \cite{accrs,cs}.
Weaver and Licata diverged from our proposed weighted colimit construction for $\funtoarr$ and instead adapted the limit-like ``Glue'' \cite{CCHM} and ``V'' \cite{angiuli-favonia-harper:18} types from cubical sets models of homotopy type theory.
As they note, the weighted colimit construction instead corresponds to Nuyts, Vezzosi, and Devriese's dual ``Weld'' types \cite{nuyts:17}.
In the present work, we follow Weaver and Licata and use a weighted limit construction; this avoids issues of fibrant replacement.

Gratzer, Weinberger, and Buchholtz \cite{gratzer-weinberger-buchholtz-univalence} present an extension of simplicial type theory where directed univalence is constructible from simpler primitives.
Their syntactic construction resembles Weaver and Licata's semantic one, as well the one we present here.
The theory has a semantics in cubical spaces; for the moment, it is not clear whether it generalizes to other $\infty$-topoi of cubical objects.

\subsubsection{Synthetic category theory}

Cisinski, Cnossen, Nguyen and Walde, in their book in progress on a synthetic approach to $\infty$-category theory, take a form of directed univalence as an axiom, namely that every cocartesian fibration is classified by some directed univalent cocartesian fibration \cite[Axiom O]{cisinski-cnossen-nguyen-walde}.
Although not presented as a formal type theory, their work shares many ideas with simplicial type theory.

\subsubsection{Straightening--unstraightening}

Lurie's straightening--unstraightening is a correspondence, natural in $\infty$-categories $\mathcal{C}$, between $\mathcal{C}$-indexed diagrams of $\infty$-categories and cocartesian fibrations over $\mathcal{C}$ \cite[3.2.0.1, 3.2.1.4]{lurie-topos}.
This correspondence restricts to one between $\mathcal{C}$-indexed diagrams of $\infty$-groupoids and left fibrations over $\mathcal{C}$ \cite[2.2.1.2]{lurie-topos}, the case of interest here.

Lurie first proves straightening--unstraightening in a model-categorical form.
In the case of left fibrations, it is roughly a Quillen equivalence
\[
  \begin{tikzcd} \slice{\sSet}{\mathrm{N}\mathcal{C}} \arrow[r, bend left] & \sSet^{\mathcal{C}} \arrow[l, bend left]
  \end{tikzcd}
\]
where $\mathcal{C}$ is a simplicial category, $\slice{\sSet}{\mathrm{N}\mathcal{C}}$ is the slice over the homotopy coherent nerve of $\mathcal{C}$ with the covariant model structure, and $\sSet^{\mathcal{C}}$ is the projective model structure on simplicial presheaves.

With our \cref{thm:dua-model}, we prove a form of this theorem for left fibrations into 1-categories $\mathcal{C}$ in an $\infty$-topos of simplicial objects.
The restriction to 1-categories avoids complications of the general case (cf.~\cite[\S3.2.5]{lurie-topos}).
In the case of the $\infty$-topos of simplicial spaces, our construction coincides with that of \textcite{heuts-moerdijk:14}; our $\arrtofunwgt \dashv \arrtofun$ and $\funtoarrwgt \dashv \funtoarr$ adjunctions correspond to their $h_! \dashv h^*$ and $r_! \dashv r^*$ respectively.
\textcite{heuts-moerdijk:16} later generalized their construction from 1-categories to simplicial categories, and \textcite{hebestreit-heuts-ruit:25} extended it from left to cocartesian fibrations, providing an alternative proof of Lurie's straightening--unstraightening.

In working relative to an arbitrary base $\infty$-topos, our work is related to that of Martini \cite{martini,martini-cocartesian}, who proves straightening--unstraightening internal to an $\infty$-topos.
One difference is that Martini takes ordinary straightening--unstraightening --- that is, internal to the $\infty$-topos of $\infty$-groupoids --- as an input; we start from scratch.

\subsubsection{Universal left/cocartesian fibrations}

Lurie uses straightening--unstraightening to build universal left and cocartesian fibrations of quasicategories.
More specifically, he constructs these over pre-existing bases: the homotopy coherent nerves of the simplicial categories of small Kan complexes and quasicategories respectively.
Cisinski and Nguyen develop a different approach to classifiers \cites[\S5.2]{Cisinski}{nguyen:22}{cisinski-nguyen}.
Building on work of Nichols-Barrer \cite[\S2.2]{nichols-barrer} and the Hofmann--Streicher universe construction \cite{hofmann-streicher:97} used in semantics of homotopy type theory \cite{KapulkinLumsdaine2021,shulman}, they define left and cocartesian fibrations of simplicial sets, $\pi \colon \EUcov \to \Ucov$ and $\pi \colon \EUcocart \to \Ucocart$, which are universal essentially by definition.
These can then be shown to be quasicategories \cites[5.2.10]{Cisinski}[3.8]{cisinski-nguyen}.
In this approach, the role of the machinery involved in Lurie's straightening--unstraightening is to analyze the structure of $\Ucov$ and $\Ucocart$.
In particular, the $n$-cells of $\Ucov$ (resp.\ $\Ucocart$) are by definition left (resp.\ cocartesian) fibrations into $\Delta^n$; straightening--unstraightening characterizes these as $[n]$-indexed diagrams of functors.
This is the connection between straightening--unstraightening and directed univalence in the form
\[
  \begin{tikzcd} \Fun_\bullet \arrow[r, "\funtoarr_\bullet", shift left=.25em] & \Ucov^{\Delta^\bullet} \arrow[l, "\arrtofun_\bullet", shift left=.25em] \end{tikzcd}
\]
of \cref{thm:dua-types}.

Relatedly, Nichols-Barrer constructs a universal left fibration \cite[\S2.2]{nichols-barrer} and conjectures it to be a quasicategory \cite[2.3.1]{nichols-barrer} equivalent to the homotopy coherent nerve of the simplicial category of spaces \cite[2.3.10]{nichols-barrer}.
Kazhdan and Varshavsky \cite[\S2.2.3]{KV} construct a universal left fibration in bisimplicial sets and prove that its base is a complete Segal space \cite[2.2.11]{KV}.\footnote{They make a (correctable) mistake: their base object is only a \emph{pseudo}functor from the simplex category (cf.~Hofmann \cite{hofmann:95}).}

Cisinski and Nguyen first prove that their classifiers are quasicategories, then prove a directed univalence result characterizing the $1$-cells in $\Ucov$/$\Ucocart$.
Here our approach differs: we prove a directed univalence result for $n$-cells (\cref{thm:dua-types}), then use this to prove the Segal and Rezk conditions for $\Ucov$ (\cref{cor:Ucov-category}).

\section{\texorpdfstring{$\infty$}{Infinity}-topoi, model topoi, and simplicial topoi}\label{sec:topoi}

Shulman proves that homotopy type theory can be interpreted into any $\infty$-topos \cite{shulman} in the sense of Rezk and Lurie \cite{lurie-topos}. His proof employs a particular model categorical presentation of an $\infty$-topos introduced by Rezk under the name \emph{model topos} \cite{rezk}, a notion which we review in \cref{ssec:model-topoi}.

In \cite{shulman}, Shulman introduces a related definition of a \emph{type-theoretic model topos}, which axiomatizes the structures used to produce categorical models of homotopy type theory. Part of this axiomatization involves a \emph{notion of fibered structure} satisfying certain properties, out of which Shulman constructs a univalent universe. He then shows that model topoi coincide with type theoretic model topoi, up to Quillen equivalence. We review this in \cref{ssec:ttmt} and \cref{ssec:fibered-structure}.

In \cref{ssec:stt}, we then review the semantic setting for the simplicial extension of homotopy type theory due to Riehl--Shulman \cite{RS} and introduce the class of \emph{left maps}. These semantics take place in an $\infty$-topos of simplicial objects in an $\infty$-topos. As this is a specialization of the prior setting, the results of Lurie, Rezk, and Shulman apply.

\subsection{\texorpdfstring{$\infty$}{infinity}-topoi and model topoi}\label{ssec:model-topoi}

In \cite{rezk}, inspired by previous work of Simpson \cite{simpson} and To\"{e}n and Vezzosi \cite{TV}, Rezk introduces the notion of a model topos: a model-categorical presentation of an $\infty$-topos. The basic example is Quillen's model structure on the category $\sSet$ of simplicial sets, which presents the $\infty$-topos $\iS$ of spaces. Rezk proves a Giraud-style theorem characterizing model topoi that we take as the definition: a \textbf{model topos} is a model category $\mE$ such that:
\begin{conditions}
  \item\label{itm:topos-presentation} $\mE$ admits a small presentation in the sense of Dugger \cite{dugger-universal,dugger-combinatorial}.
  \item\label{itm:topos-descent} $\mE$ has descent.
\end{conditions}
We now unpack these conditions.

A model category admits a \textbf{small presentation} just when it is Quillen equivalent to a model category of the form $\sSet^{\cD^\op}_S$, that is, to a \emph{localization} of the injective model structure on simplicial presheaves on a small simplicial category $\cD$ by a set of maps $S$.\footnote{Dugger's small presentations involve localizations of the \emph{projective} model structure rather than the \emph{injective} model structure. Since the projective and injective model structures are Quillen equivalent and the process of localization preserves this Quillen equivalence, up to Quillen equivalence it makes no difference which model structure is used. The cost of using the injective model structure rather than the projective one is that there is no longer a direct Quillen equivalence to the model category $\mE$ being presented but rather a zig-zag of Quillen equivalences. For our purpose here, which is to construct categorical models of homotopy type theory, this is immaterial, and we prefer the injective model structure because its cofibrations are exactly the monomorphisms.}  The category $\cD$ can be thought of as providing ``generators'' for $\mE$ while the set $S$ gives the ``relations.'' Here ``localization'' refers to \emph{left Bousfield localization}, a process which ``formally inverts'' the maps in $S$ by turning them into weak equivalences while restricting the class of fibrations; this process does not change the underlying category or the class of cofibrations.

The model categories $\sSet^{\cD^\op}_S$ are simplicial model categories, so we may unpack the second condition in that setting.

\begin{defn}[{\cite{rezk}}]\label{defn:descent}
  A simplicial model category $\mE$ satisfies \textbf{descent} if:
  \begin{conditions}
    \item Homotopy colimits are stable under homotopy pullback: for any $X = \hocolim_i X_i$ and $f \colon Y \to X$, the map $\hocolim_i (Y \tilde{\times}_X X_i) \to Y$ from the homotopy colimit of the homotopy pullbacks is a weak equivalence.
    \item For a natural transformation between a pair of diagrams whose naturality squares are homotopy pullbacks, the naturality squares of the homotopy colimit cones are also homotopy pullbacks.
  \end{conditions}
\end{defn}

The descent condition \ref{itm:topos-descent} puts a further restriction on the model categories that satisfy \ref{itm:topos-presentation}. Rezk proves that a localized model category of the form $\sSet^{\cD^\op}_S$ satisfies descent just when the left adjoint of the canonical Quillen adjunction
\[
  \begin{tikzcd}
    \sSet^{\cD^\op} \arrow[r, bend left=20, "=" above] \arrow[r, adj] & {\sSet^{\cD^\op}_S} \arrow[l, bend left=20, "=" below]
  \end{tikzcd}
\]
is \textbf{left exact}, meaning that it preserves homotopy pullbacks, and hence all finite homotopy limits.  Thus:

\begin{thm}[{\cite[6.9]{rezk}}] A model category $\mE$ is a \textbf{model topos} if and only if it is Quillen equivalent to a left exact localization of an injective model category of simplicial presheaves indexed by some small simplicial category.
\end{thm}

Dugger proves that all combinatorial model categories have a presentation \cite{dugger-combinatorial}, so model topoi can be thought of as combinatorial simplicial model categories of simplicial sheaves, where the data $(\cD,S)$ together define the \textbf{model site}. Model topoi are simplicial model categories that present $\infty$-\textbf{topoi}, which Lurie defines to be accessible left exact reflective sub-$\infty$-categories of an $\infty$-category $\cS^{\cD^\op}$ of presheaves and then characterizes in many different ways \cite[\S6.1]{lurie-topos}.

\subsection{Type-theoretic model topoi}\label{ssec:ttmt}

Shulman introduces the following axiomatization to consolidate the hypotheses he uses to construct models of homotopy type theory with univalent universes.

\begin{defn}[{\cite[6.1]{shulman}}]\label{defn:ttmt} A \textbf{type-theoretic model topos} is a model category such that:
  \begin{conditions}
    \item The underlying category is a Grothendieck 1-topos.
    \item The model category is a right proper simplicial Cisinski model category: the cofibrations are the monomorphisms and the model structure is proper, simplicial, and combinatorial.
    \item The underlying category is simplicially locally cartesian closed, which amounts to the additional requirement that pullbacks preserve simplicial tensors.
    \item\label{itm:notion-of-fibered} There is a locally representable and relatively acyclic notion of fibered structure $\FF$ such that the fibrations are exactly the $\FF$-algebras.
  \end{conditions}
\end{defn}

The final axiom, which we'll say more about in \cref{ssec:fibered-structure}, ensures that $\mE$ has fibrant univalent universes of relatively $\kappa$-presentable fibrations for sufficiently large inaccessible cardinals $\kappa$.

\begin{thm}[{\cite[6.3]{shulman}}] For any type-theoretic model topos $\mE$, there is a regular cardinal $\lambda$ such that $\mE$ interprets Martin-L\"of type theory with as many universe types as there are inaccessible cardinals larger than $\lambda$, closed under the standard type formers  and satisfying the univalence axiom.
  \end{thm}

Shulman then proves:

\begin{thm}[{\cite[6.6]{shulman}}]
  Every type-theoretic model topos is a model topos.
\end{thm}

For the converse, Shulman demonstrates:

\begin{thm}[{\cite[6.8, 8.29, 10.5, 11.1]{shulman}}]\label{thm:shulman-tmtt-presentation}
  Quillen's model structure on simplicial sets, its injective model structures on simplicial presheaves, and left exact localizations thereof all define type-theoretic model topoi.
  Consequently, every model topos is Quillen equivalent to a type-theoretic model topos, and thus every $\infty$-topos can be presented by a type-theoretic model topos.
\end{thm}

In summary, any $\infty$-topos $\iE$ may be presented by a simplicial model category $\mE \coloneqq \sSet^{\cD^\op}_S$ defined by left exact localization of the injective model structure, and such model categories are both model topoi and type-theoretic model topoi.
For our results, we will use one additional property satisfied by these model categories beyond those axiomatized in \cref{defn:ttmt}.

\begin{defn}
  A type-theoretic model topos is \textbf{fibration-extensive} when its fibrations are closed under coproduct.
\end{defn}

\begin{prop}\label{prop:fibration-extensive-trivial}
  In a fibration-extensive type-theoretic model topos, trivial fibrations are closed under coproduct.
\end{prop}
\begin{proof}
  Since the cofibrations are the monomorphisms, all objects are cofibrant.
  Thus closure of trivial cofibrations under coproducts implies closure of weak equivalences under coproducts by Ken Brown's lemma, which combined with fibration-extensivity implies closure of trivial fibrations under coproduct.
\end{proof}

\begin{lem}\label{lem:sset-coproducts} Quillen's model structure on simplicial sets is fibration-extensive.
\end{lem}
\begin{proof}
Since the codomains of the generating trivial cofibrations are representable and coproducts in $\Set$ are disjoint, any lifting problem against a coproduct of fibrations factors through one of the component inclusions and the pullback picks out a single component fibration. This solves the original lifting problem
\[ \begin{tikzcd}[/tikz/baseline=(B.base)] \Lambda^n_k \arrow[d, utcofarrow] \arrow[rr, bend left] \arrow[r, dotted] & Y^j \arrow[d, two heads, dotted, "f^j"'] \arrow[r, hook, dotted] \arrow[dr, phantom, "\lrcorner" very near start] & \coprod_i Y^i \arrow[d, two heads, "f^i"] \\ \Delta^n \arrow[ur, dashed] \arrow[r, dotted] \arrow[rr, bend right, ""{name=B}] & X^j \arrow[r, hook, dotted]& \coprod_i X^i \rlap{.}
\end{tikzcd} \qedhere
\]
\end{proof}

\begin{lem}\label{lem:injective-coproducts} For any small simplicially enriched category $\cD$, the injective model structure on simplicial presheaves $\sSet^{\cD^\op}$ is fibration-extensive.
\end{lem}
\begin{proof}
  Since coproducts in $\sSet^{\cD^\op}$ are defined pointwise, it follows directly from \cref{lem:sset-coproducts} that the \emph{projective} model structure is fibration-extensive.
  In \cite[8.22]{shulman}, Shulman characterizes the \emph{injective} fibrations as projective fibrations $f \colon Y \to X$ for which the pullback corner map in the naturality square of the unit of the cobar monad has a retraction
\begin{equation}\label{eq:injective-fib-char} \begin{tikzcd} Y \arrow[dd, two heads, "f"'] \arrow[dr, dashed] \arrow[rr, "\nu_Y"]& [-.25cm] & C(R,UR,UY) \arrow[dd, "{C(R,UR,Uf)}"] \\ [-.5cm] & \bullet \arrow[dl, dashed] \arrow[ur, dashed] \arrow[ul, bend left, dotted] \arrow[dr, phantom, "\lrcorner" very near start] \\ X \arrow[rr, "\nu_X"'] & & C(R,UR,UX) \rlap{.}
\end{tikzcd}\end{equation}

For any simplicial set $A$, while the functor $(-)^A$ does not preserve coproducts, the canonical comparison map is a pullback
\[ \begin{tikzcd} \coprod_i (X^i)^A \arrow[d, ] \arrow[r] \arrow[dr, phantom, "\lrcorner" very near start] & (\coprod_i X^i)^A \arrow[d] \\ \coprod_i \Delta^0  \arrow[r] & \Delta^0 \end{tikzcd}\] and hence for any family of maps, the natural square from $\coprod_i (f^i)^A$ to $(\coprod_if^i)^A$ is a pullback. It follows, by the formula for pointwise enriched right Kan extensions and descent, that the natural square
\[ \begin{tikzcd} \coprod_i C(R,UR,UY^i) \arrow[r] \arrow[d] \arrow[dr, phantom, "\lrcorner" very near start] & C(R,UR,U(\coprod_iY^i)) \arrow[d] \\ \coprod_i C(R,UR,UX^i) \arrow[r] & C(R,UR,U(\coprod_iX^i))\end{tikzcd}\] is a pullback. Since the category of simplicial sets is extensive, it follows that when $f \coloneqq \coprod_i f_i$ is a coproduct, the pullback corner map in \eqref{eq:injective-fib-char} is a coproduct of the corresponding maps for $f_i$, so it has a retraction defined by the coproduct of their retractions. Thus, injective fibrations in  $\sSet^{\cD^\op}$ are also closed under coproducts.
\end{proof}

\begin{lem}\label{lem:localization-coproducts}  For any small simplicially enriched category $\cD$, any left exact localization of the injective model structure on simplicial presheaves $\sSet^{\cD^\op}$ is fibration-extensive.
\end{lem}
\begin{proof}
  The fibrations in the localized model structure $\sSet^{\cD^\op}_S$ are injective fibrations that are $S$-local.
  Given \cref{lem:injective-coproducts}, it remains only to verify that a coproduct of such maps remains $S$-loca.
  This can be established at the $\infty$-categorical level, where $S$-locality becomes an orthogonality condition.

Let $\iE$ be an $\infty$-topos and $I$ a set. Our task is to show that the $I$-indexed coproduct functor $\coprod_I \colon \prod_I \iE \to \iE$ preserves right orthogonality against the maps in $S$. We may lift the orthogonal factorization system defined by $S$ from $\iE$ to the slice $\slice{\iE}{I}$ and factor the coproduct functor through the projection from the slice to define a functor $\coprod_I \colon \prod_I \iE \to \slice{\iE}{I}$, which is an equivalence, by descent. Our task is again to show that this functor preserves right orthogonality against the maps in $S$. For this it suffices to show that its (left adjoint) equivalence inverse $\slice{\iE}{I} \to \prod_I \iE$ preserves the class $S$. The components of this functor are defined by pullback, so this follows from the pullback stability of $S$ at the $\infty$-categorical level.
\end{proof}

In conclusion, we can slightly enhance the final claim of \cref{thm:shulman-tmtt-presentation}:

\begin{cor}\label{cor:model-topos-coproducts}
  Every $\infty$-topos can be presented by a fibration-extensive type-theoretic model topos.
\end{cor}
\begin{proof}
  By \cref{thm:shulman-tmtt-presentation}, every $\infty$-topos can be presented by a type-theoretic model topos in the form of a left exact localization of the injective model structure on some $\sSet^{\cD^\op}$.
  By \cref{lem:localization-coproducts}, such type-theoretic model topoi are fibration-extensive.
\end{proof}

\subsection{Notions of fibered structure}\label{ssec:fibered-structure}

We now discuss the final ingredient \ref{itm:notion-of-fibered} in Shulman's notion of type-theoretic model topos, though we present our \cref{defn:fibered-structure,defn:locally-representable} in a different (but equivalent) form from \cite[\S3]{shulman}. Throughout we fix a 1-category $\mE$ with pullbacks and write $\cart{\mE}$ for the category of arrows and pullback squares.\footnote{More generally, in the case where $\mE$ does not admit all pullbacks, we may interpret $\cart{\mE}$ as the category of arrows admitting base change along arbitrary maps and pullback squares.} The forgetful functor $\cod \colon \cart{\mE} \to \mE$ is a Grothendieck fibration with groupoidal fibers --- a \textbf{right fibration} in $\infty$-categorical terminology --- classifying the groupoid-valued pseudofunctor that sends $X \in \mE$ to the groupoid of arrows with codomain $X$ and acts contravariantly by pullback.

\begin{defn}[{\cite[3.1]{shulman}}]\label{defn:fibered-structure}
   A \textbf{notion of fibered structure} on $\mE$ is a right fibration over $\cart{\mE}$ that has small fibers.
\end{defn}

To align our notation with the literature we write $\FF$ for a notion of fibered structure on $\mE$ and denote the corresponding right fibration by $U_\FF \colon \alg{\FF} \to \cart{\mE}$. Following Shulman, we refer to an object in the fiber over a morphism $f$ in $\mE$ as an $\FF$-\textbf{algebra}, which is thought of as equipping $f$ with a chosen $\FF$-structure. The morphisms in the total category $\alg{\FF}$ are $\FF$-\textbf{morphisms}, pullback squares between $\FF$-algebras in which the $\FF$-structure on the pullback is induced from the $\FF$-structure on the arrow being pulled back.

\begin{defn}[{\cite[3.10]{shulman}}]\label{defn:locally-representable}
  A notion of fibered structure is \textbf{locally representable} just when the right fibration admits a right adjoint:
  \[ \begin{tikzcd} \alg{\FF} \arrow[r, bend left, "U_{\FF}"] \arrow[r, adj] & \cart{\mE} \arrow[l, bend left, dashed, "R_{\FF}"] \rlap{.} \end{tikzcd}\]
\end{defn}

Unpacking \cref{defn:locally-representable}, for a locally representable notion of fibered structure $\FF$, the counit equips each $f \colon Y \to X$ in $\mE$ with the universal $\FF$-algebra obtained as a pullback of $f$:
\[ \begin{tikzcd} P \arrow[d, "U_\FF R_\FF f"'] \arrow[r, "\epsilon'_f"] \arrow[dr, phantom, "\lrcorner" very near start] & Y \arrow[d, "f"] \\ \FF(f) \arrow[r, "\epsilon_f"'] & X \rlap{.} \end{tikzcd}\] For any $g \colon Z \to X$ and $\FF$-algebra structure on $g^*f$, there is a classifying lift of $g$ through $\epsilon_f$ defining an $\FF$-morphism that induces the specified $\FF$-algebra structure on $g^*f$ from the $\FF$-algebra structure on $\epsilon_f^*f$ specified by $R_\FF f$. In particular, $\FF$-algebra structures on $f$ are classified by sections of $\epsilon_f$.

Shulman gives a long list of examples in \cite[\S 3]{shulman}, but we will only need a few.

\begin{ex}[{\cite[3.4,3.13]{shulman}}] There is a notion of fibered structure $U_{\EE_\bullet} \colon \alg{\EE_\bullet} \to \cart{\mE}$ for which an $\EE_\bullet$-structure on a morphism is a section. The category $\alg{\EE_\bullet}$ is the category of split epimorphisms and pullback squares that commute with the specified sections. This notion of fibered structure is locally representable with right adjoint $R_\FF f$ defined by sending $f$ to one leg of its kernel pair, with the canonical diagonal section. Indeed, this is the ``universal'' locally representable notion of fibered structure in a sense, since sections of $\EE_\bullet$-structures on $f \colon Y \to X$ correspond to sections of $f$ itself.
\end{ex}

\begin{ex}[{\cite[3.2, 3.16]{shulman}}] If $\alg{\FF} \hookrightarrow \cart{\mE}$ is a fully faithful inclusion, then $\alg{\FF}$ is just a pullback-stable class of morphisms in $\mE$. In such a case, we say that $\alg{\FF}$ is a \textbf{full notion of fibered structure}, the key property being that any pullback square between $\FF$-algebras is uniquely an $\FF$-morphism. When $\mE$ is a presheaf category, a full notion of fibered structure is locally representable if and only if a map is in the pullback-stable class just when all of its pullbacks to representables are.
\end{ex}

In particular, if $\mE$ is a model category, the fibrations define a notion of fibered structure in this way, but this notion of fibered structure is not always locally representable.

\begin{ex}[{\cite[4.18]{shulman}}] When $\mE$ is locally presentable and locally cartesian closed, there is a regular cardinal $\lambda$ such that for any regular cardinal $\kappa$ with $\kappa \triangleright \lambda$, the relatively $\kappa$-presentable morphisms define a locally representable full notion of fibered structure $\EE^\kappa$.\footnote{An object is $\kappa$-presentable if its covariant representable preserves $\kappa$-filtered colimits. A morphism is relatively $\kappa$-presentable if its pullback over any $\kappa$-presentable object is $\kappa$-presentable. Here $\kappa \triangleright\lambda$ means that $\kappa > \lambda$ is such that any $\lambda$-presentable category is $\kappa$-presentable.}
\end{ex}

Locally representable notions of fibered structure can be constructed from other locally representable notions of fibered structure using various techniques described in \cite[\S 3]{shulman}. Many of these can be understood to derive from the following categorical result:

\begin{lem}\label{lem:loc-rep-pb} Consider a pullback of a Grothendieck fibration $P$ along a functor $F$:
  \[ \begin{tikzcd} \mE' \arrow[d, "P'"'] \arrow[r, "G"] \arrow[dr, phantom, "\lrcorner" very near start] & \mE \arrow[d, "P"] \\ \mB' \arrow[r, "F"'] & \mB \rlap{.} \end{tikzcd} \]
  \begin{parts}
\item\label{lem:loc-rep-pb:term} If $F$ admits a right adjoint $R$ and $\mE$ has a terminal object then so does $\mE'$.
\item\label{lem:loc-rep-pb:ra} If $F$ admits an indexed right adjoint \[ \begin{tikzcd} \slice{\mB'}{b'} \arrow[r, bend left, "F"] \arrow[r, adj] & \slice{\mB}{Fb'} \arrow[l, bend left, "R_{b'}" pos=.45]  \end{tikzcd} \] for all $b' \in \mB'$  and $P$ admits a right adjoint then so does $P'$.
  \end{parts}
\end{lem}
\begin{proof}
  For \ref{lem:loc-rep-pb:term}, let $t \in \mE$ be terminal and consider the counit component $\epsilon \colon FRPt \to Pt \in \mB$. If $e \in \mE$ denotes the domain of the cartesian lift then the pair $(RPt, e)$ induces an object in $\mE'$ which can be seen to be terminal.

  For \ref{lem:loc-rep-pb:ra}, our task is to construct a terminal object in the comma category $\slice{P'}{b'}$ for each $b' \in \mB'$; since $P$ has a right adjoint we have a terminal object $\epsilon_{Fb'} \in \slice{P}{Fb'}$. We have an induced pullback of comma categories:
\[ \begin{tikzcd} \slice{P'}{b'} \arrow[d] \arrow[r] \arrow[dr, phantom, "\lrcorner" very near start] & \slice{P}{Fb'} \arrow[d] \\ \slice{\mB'}{b'} \arrow[r, "F"'] & \slice{\mB}{Fb'} \rlap{,} \end{tikzcd} \] in which the right-hand vertical functor is again a cartesian fibration and the bottom functor has a right adjoint, satisfying the hypotheses of \ref{lem:loc-rep-pb:term}.
\end{proof}

\begin{ex}[{\cite[3.3,3.11]{shulman}}]\label{ex:intersection-fibered-structure}  If $\alg{\FF_1} \to \cart{\mE}$ and $\alg{\FF_2} \to \cart{\mE}$ are notions of fibered structure, then so is the product over $\alg{\FF_1 \times_{\mE} \FF_2} \to \cart{\mE}$ over $\cart{\mE}$ defined by the pullback
  \[\begin{tikzcd} \alg{\FF_1 \times_{\mE} \FF_2}\arrow[r] \arrow[d] \arrow[dr, phantom, "\lrcorner" very near start] & \alg{\FF_2} \arrow[d] \\ \alg{\FF_1} \arrow[r] & \cart{\mE} \end{tikzcd}\]  where an $\FF_1 \times_{\mE} \FF_2$-structure on a morphism is just a pair given by an $\FF_1$-structure and an $\FF_2$-structure. Moreover, when both notions of fibered structure are locally representable, so is the pullback by \cref{lem:loc-rep-pb} and composition of right adjoints.
\end{ex}

\begin{ex}[{\cite[2.1.16]{accrs}}]\label{ex:leibniz-pullback-application-loc-rep-fibered-structure} Suppose $\mE$ is locally cartesian closed and has a locally representable notion of fibered structure $\FF$. Then if $\mD$ has pullbacks, $\alpha$ is a natural transformation between pullback-preserving functors
\[ \begin{tikzcd} \mD \arrow[r, bend left, "L"] \arrow[r, phantom, "\Downarrow\alpha"] \arrow[r, bend right, "K"'] & \mE  \end{tikzcd}  \]
and the domain $L$ has an indexed right adjoint,
 then $\mD$ has a locally representable notion of fibered structure $\GG$ whose category of algebras is defined by pullback along the Leibniz pullback application\footnote{For a morphism $f \in \cD$, the \textbf{Leibniz pullback application} $\alpha \hat{\circ} f$ is the map to the pullback in the naturality square for $\alpha$ indexed by $f$.} of $\alpha$:
\[ \begin{tikzcd}
\alg{\GG} \arrow[d] \arrow[r] \arrow[dr, phantom, "\lrcorner" very near start] & \alg{\FF} \arrow[d] \\ \cart{\cD} \arrow[r, "\alpha\hat{\circ}-"'] & \cart{\mE} \rlap{.}
\end{tikzcd} \]
The proof is again by \cref{lem:loc-rep-pb} after observing that the functor $\alpha \hat{\circ} - \colon \cart{\mD} \to \cart{\mE}$ has an indexed right adjoint.
\end{ex}

A suitable notion of fibered structure on a suitable model category will have a ``classifying'' universe in the sense of the following definition.

\begin{defn}\label{defn:fib-universe} Let $\mE$ be a model category and $\FF$ a notion of fibered structure on $\mE$. A \textbf{universe} for $\FF$ is a cofibrant object $U$, together with an $\FF$-algebra $\pi \colon \tilde{U} \fto U$ satisfying the following condition: given any cofibration $i \colon A \rightarrowtail B$ in $\mE$,  $\FF$-algebra $p$,  and pair of solid squares, as displayed below, that are both $\FF$-morphisms, there exist dotted arrows defining a third square that is a pullback and an $\FF$-morphism and makes the diagram commute:
  \[
  \begin{tikzcd}
    D \arrow[dd, two heads, "q"'] \arrow[rr] \arrow[drr, phantom, "\lrcorner" very near start] \arrow[dddr, phantom, "\lrcorner" very near start]\arrow[dr] & & \tilde{U} \arrow[dd, two heads, "\pi"] \\ & E  \arrow[dr, phantom, "\lrcorner" very near start]\arrow[ur, dashed] & ~ \\ A \arrow[dr, tail, "i"'] \arrow[rr, "h" near end] & & U \rlap{.} \\ & B \arrow[ur, dashed, "k"'] \arrow[from=uu, two heads, crossing over, "p"' near start]
  \end{tikzcd}
  \]
\end{defn}

\begin{thm} In a Grothendieck 1-topos with a combinatorial model structure in which all cofibrations are monomorphisms, any small-groupoid-valued\footnote{The condition of being small-groupoid-valued means that for each $X \in \mE$, the isomorphism class of $\FF$-algebras with codomain $X$ is essentially small.} locally representable notion of fibered structure has a universe.
\end{thm}

The proof is by an adaptation of Quillen's small object argument that uses the hypotheses on the notion of fibered structure $\FF$ to ensure that it defines a \textbf{stack for cell complexes}, meaning the groupoid-valued pseudofunctor classified by $\alg{\FF} \to \cart{\mE}$ preserves coproducts, pushouts of cofibrations, and transfinite composites of cofibrations in the weak bicategorical sense.

It remains to connect a notion of fibered structure on a model category $\mE$ to the fibrations in that model category.

\begin{defn} A notion of fibered structure $\FF$ is \textbf{relatively acyclic} if, for any pullback
  \begin{equation}\label{eq:relative-acyclicity}
    \begin{tikzcd}
      D \arrow[r, "e"] \arrow[d, "q"'] \arrow[dr, phantom, "\lrcorner" very near start] & E \arrow[d, "p"] \\ A \arrow[r, tail, "i"'] &  B
    \end{tikzcd}
  \end{equation} with $q,p$ both $\FF$-algebras and $i$ a cofibration, there exists a new $\FF$-algebra structure on $p$ making the square an $\FF$-morphism.
\end{defn}

\begin{rmk}
  When $\FF$ is relatively acyclic, the $\FF$-algebras satisfy the following property sometimes known as \textbf{realignment}: given any cofibration $i \colon A \rightarrowtail B$ in $\mE$,  $\FF$-algebra $p$,  and pair of pullback squares, as displayed below, the map $q$ inherts an $\FF$-algebra structure from $\pi$ making the back square an $\FF$-morphism, and relative acyclicity can be used to equip $p$ with a new $\FF$-algebra structure making the front left square into an $\FF$-morphism as well. Now since $\pi$ is a universe, there exist arrows defining a third square that is a pullback and an $\FF$-morphism relative to the new $\FF$-algebra structure on $p$ and makes the diagram commute:
\[
\begin{tikzcd}
  D \arrow[dd, two heads, "q"'] \arrow[rr] \arrow[drr, phantom, "\lrcorner" very near start] \arrow[dddr, phantom, "\lrcorner" very near start]\arrow[dr] & & \tilde{U} \arrow[dd, two heads, "\pi"] \\ & E  \arrow[dr, phantom, "\lrcorner" very near start]\arrow[ur, dashed] & ~ \\ A \arrow[dr, tail, "i"'] \arrow[rr, "h" near end] & & U \rlap{.} \\ & B \arrow[ur, dashed, "k"'] \arrow[from=uu, two heads, crossing over, "p"' near start]
\end{tikzcd}
\]
  In particular, supposing $\mE$ has a strict initial object $0$, taking $A = 0$ shows that every $\FF$-algebra $p$ admits an $\FF$-morphism to $\pi$.
\end{rmk}

\begin{defn} A notion of fibered structure $\FF$ is \textbf{homotopy invariant} if every $\FF$-algebra is a fibration, and given any commutative square
  \[ \begin{tikzcd} Y' \arrow[d, "f'"', two heads] \arrow[r, wearrow] & Y \arrow[d, "f", two heads] \\ X' \arrow[r, wearrow] & X
  \end{tikzcd}\] whose vertical maps are fibrations and horizontal maps are weak equivalences, $f$ admits an $\FF$-structure if and only if $f'$ does.
\end{defn}

Note that homotopy invariance is automatically satisfied when the $\FF$-algebras are precisely the model-categorical fibrations.

\begin{thm}[{\cite[5.22]{shulman}}]\label{thm:shulman-universe} For any locally representable, relatively acyclic, and homotopy invariant notion of fibered structure $\FF$ on a right proper simplicial Cisinski model category, there is a regular cardinal $\lambda$ such that for any regular cardinal $\kappa \triangleright \lambda$ there exists a morphism $\pi \colon \tilde{U} \to U$ defining a fibrant, univalent universe for $\FF^\kappa \coloneqq \FF \times_{\mE} \EE^\kappa$.
In particular, $\pi$ is a $\kappa$-presentable $\FF$-algebra and every relatively $\kappa$-presentable $\FF$-algebra is a pullback of $\pi$.
\end{thm}

By a construction of Voevodsky, for any map $\tilde{U} \to U$ in a type-theoretic model topos, there is an object of equivalences factoring the diagonal on $U$, \[ \begin{tikzcd} U \arrow[r] & \Eq(\tilde{U}) \arrow[r] & U \times U \rlap{,} \end{tikzcd}\] and this construction is stable under pullback \cite[\S4]{shulman-reedy}.
In this context, a map $\tilde{U} \to U$ is \textbf{univalent} when the map $U \to \Eq(\tilde{U})$ is a weak equivalence, or equivalently when the induced map $\Id(U) \to \Eq(\tilde{U})$ over $U \times U$ from the path object of $U$ is a weak equivalence.

We observe the following consequence of Shulman's axiomatization, the existence of a similar notion of fibered structure for the trivial fibrations.
In our setting of a type-theoretic model topos, this can also be constructed directly; see \cite[\S 2.2]{accrs}.

\begin{lem}\label{lem:triv-fib-loc-rep}
  Any type-theoretic model topos $\mE$ admits a locally representable, relatively acyclic, and homotopy invariant notion of fibered structure $\TF$ such that the class of maps admitting $\TF$-algebra structures is the class of trivial fibrations.
\end{lem}
\begin{proof}
  For any $B \in \mE$ in a type-theoretic model topos, there is a right Quillen functor $\isContr_B \colon \mE_{/B} \to \mE_{/B}$ between the slice model categories defined by interpreting the homotopy type theoretic construction of $\isContr$ in context $B$ \cite[4.1]{shulman-reedy}. Given a fibration $p \colon E \twoheadrightarrow B$, a section to $ \pi_p \colon \isContr_B(p) \twoheadrightarrow B$ provides the data of a section $s$ to $p$  together with a homotopy in the model category $\mE_{/B}$ from $sp$ to $1_E$. Thus, a fibration $p$ is a trivial fibration if and only if $ \pi_p \colon \isContr_B(p) \twoheadrightarrow B$ admits a section.

By \cite[3.4, 3.13]{shulman}, there is a locally representable notion of fibered structure $\EE_\bullet$ such that $\EE_\bullet$-algebras are maps admitting sections. Since the endofunctors $\isContr_B$ commute with pullbacks up to coherent isomorphism, there is a pulled back locally representable notion of fibered structure $\CC$ on $\mE$ such that a $\CC$-algebra structure on $p$ is an $\EE_\bullet$-algebra structure on $\pi_p$ \cite[3.6]{shulman}. Finally, we define $\TF \coloneqq \FF \times_{\mE} \CC$, using the notion of fibered structure $\FF$ from axiom (iv) of \cref{defn:ttmt}.

As a pullback of locally representable notions of fibered structure, $\TF$ is locally representable by \cref{ex:intersection-fibered-structure}. The proof of \cite[5.16]{shulman} demonstrates that $\TF$ is relatively acyclic. A pullback square between $\TF$-algebras
\[
  \begin{tikzcd}
    D \arrow[r, "e"] \arrow[d, "q"'] \arrow[dr, phantom, "\lrcorner" very near start] & E \arrow[d, "p"] \\ A \arrow[r, tail, "i"'] &  B
  \end{tikzcd}
\]
induces a commutative square below-left:
\[ \begin{tikzcd} \isContr_Aq \arrow[d, "\pi_q"] \arrow[r, "{(e, i)}"] & \isContr_Bp \arrow[d, "\pi_p"']& A \arrow[d, "i"', tail] \arrow[r, "s_q"] & \isContr_Aq \arrow[r, "{(e,i)}"] & \isContr_Bp  \arrow[d, utfibarrow, "\pi_p"] \\ \arrow[u, bend left, dashed, "s_q"] A \arrow[r, "i"', tail]  & B \rlap{,} \arrow[u, bend right, dashed, "s_p"'] & B \arrow[rr, equals] \arrow[urr, dashed, "s_p'"] & & B \rlap{,} \end{tikzcd}
\]
in which the vertical morphisms admit not necessarily compatible sections. But since $p$ is a trivial fibration, and $\isContr_B$ is right Quillen, $\pi_p$ is a trivial fibration, so we may lift against the cofibration $i$ to define a new compatible choice of section. This new choice of section equips $p$ with a new $\CC$-algebra structure that makes the pullback square a $\CC$-morphism and since $\FF$ is relatively acyclic, $p$ admits a new $\FF$-algebra structure making the square an $\FF$-morphism as well. Together, these choices give $p$ a new $\TF$-algebra structure making the square a $\TF$-morphism.
\end{proof}

\subsection{Simplicial \texorpdfstring{$\infty$}{infinity}-topoi and left maps}\label{ssec:stt}

By \cite{RS,shulman, Weinberger}, the simplicial type theory has semantics in type-theoretic model topoi of the form $\smE \coloneqq \mE^{\DDelta^\op}$, where $\mE$ is a type-theoretic model topos.\footnote{Recent work of Rasekh extends these semantics to filter quotients of simplicial spaces \cite{rasekh2025simplicialhomotopytypetheory}, a semantic setting which is not considered here.} Diagrams in the simplicial type theory may be indexed by an abstract notion of ``shapes,'' which include the simplices and their subcomplexes, modelled by objects in $\smE$ that are not necessarily fibrant. Below, we construct a canonical cosimplicial object $\Delta^\bullet \colon \DDelta \to \smE$ defined as an instance of a more general \emph{discrete embedding} of simplicial sets into $\smE$. We introduce these objects and describe a few other features of this setting that will be needed below.

Any $\infty$-topos $\iE$ comes with a unique geometric morphism to the $\infty$-topos $\iS$ of spaces, this given by the \textbf{global sections} functor $\Gamma \colon \iE \to \iS$ and its left exact left adjoint $\Delta \colon \iS \to \iE$:
\[ \begin{tikzcd} \iE \arrow[r, bend right, "\Gamma"'] \arrow[r, adj] & \iS \arrow[l, bend right, "\Delta"'] \rlap{.} \end{tikzcd}\]
At the level of a type-theoretic model topos $\mE$ presenting $\iE$, the global sections functor is modeled by the enriched hom functor $\mE(1,-) \colon \mE \to \sSet$, and the left adjoint by the simplicial copower $- \otimes 1 \colon \sSet \to \mE$, and we have further adjunctions:
\[ \begin{tikzcd}[sep=large] \mE \arrow[r, bend right, "{\mE(1,-)}"{below}] \arrow[r, adj] & \sSet \arrow[l, bend right, "- \otimes 1"{above}] \arrow[r, bend right, "\lim=\ev_0"{below}] \arrow[r, adj] & \Set \arrow[l, bend right, "\const"{above}] \rlap{.} \end{tikzcd}\]
The composite of the constant functors defines the \textbf{discrete embedding} of sets into $\mE$. The objects in its image have the form of set-indexed copowers of the terminal object $\sqcup_I \ast$.

\begin{rmk}\label{rmk:discrete-map} The discrete embedding functor is not generally faithful (when $\cD$ is empty) or full (when $\cD$ has multiple components, maps between constant diagrams need not be constant). We use the term \textbf{discrete map} to refer to maps between discrete objects in $\iE$ in the image of the discrete embedding.
\end{rmk}

The discrete embedding and its adjoint induce pointwise-defined adjunctions between simplicial objects:
\[
  \begin{tikzcd}[sep=large] \smE \ar[r,phantom,"="] &[-3em] \mE^{\DDelta^\op} \arrow[r, bend right, "{\mE(1,-)^{\DDelta^\op}}"{below}] \arrow[r, adj] & \sSet^{\DDelta^\op} \arrow[l, bend right, "(- \otimes 1)^{\DDelta^\op}"{above}] \arrow[r, bend right, "\ev_0^{\DDelta^\op}"{below}] \arrow[r, adj] & \Set^{\DDelta^\op} \arrow[l, bend right, "\const^{\DDelta^\op}"{above}] \ar[r,phantom,"="] &[-3em] \sSet \rlap{.} \end{tikzcd}\]
We again refer to the composite of the constant functors as the \textbf{discrete embedding} of simplicial sets into $\smE$ and generally denote the discrete embedding silently when the context is clear, writing $K \in \smE$ for the discrete embedding of a simplicial set $K$. Due to \cref{rmk:discrete-map}, we must be clear to specify whether a map between discrete objects is a \textbf{discrete map}, meaning a map in the image of the discrete embedding, or a map in $\smE$. Here ``discreteness'' refers to the fact that for a discrete object $K \in \smE$ and $[m] \in \DDelta$ the component $K_m \in \mE$ decomposes as a set-indexed coproduct of the terminal object $K_m \cong \coprod_{K_m} \ast \in \mE$.

\begin{defn}
The discrete embedding defines a canonical cosimplicial object $\Delta^\bullet \colon \DDelta \to \smE$. Our conventions allow us to write $\Delta^n \in \smE$ for the image of $[n] \in \DDelta$. By the Yoneda lemma, the object $\Delta^n_m \in \mE$ decomposes as a coproduct $\Delta^n_m \cong \coprod_{\DDelta([m],[n])}\ast$.
\end{defn}

We now introduce the class of \emph{left maps} at the level of an $\infty$-category $\sE$ of simplicial objects in an $\infty$-category $\iE$. The parallel class of \emph{left fibrations} at the level of the type-theoretic model topos will be studied in \S\ref{ssec:left-fibration}.

\begin{defn}\label{defn:left-map} A map $f \colon Y \to X$ in $\sE$ is a \textbf{left map} if for all $m \geq 1$ the naturality square induced by $0 \colon [0] \to [m]$ in $\DDelta$  is a pullback in $\iE$:
  \begin{equation}\label{eq:left-map-square} \begin{tikzcd} Y_m \arrow[dr, phantom, "\lrcorner" very near start] \arrow[r, "\ev_0"] \arrow[d, "f_m"'] & Y_0 \arrow[d, "f_0"] \\ X_m \arrow[r, "\ev_0"'] & X_0 \rlap{.}
    \end{tikzcd}\end{equation}
  Dually, $f$ is a \textbf{right map} if the naturality square induced by $m \colon [0] \to [m]$ is a pullback in $\iE$ for $m \geq 1$.
  We say that an object $(Y,f) \in \slice{\sE}{X}$ in a slice is a \textbf{left object} if its underlying map $f \colon Y \to X$ is a left map.
\end{defn}

Left maps enjoy various closure properties.

\begin{lem}\label{lem:left-map-repleteness} Any map that is equivalent to a left map is a left map.
\end{lem}
\begin{proof}  The notion of pullback square in an $\infty$-category is equivalence-invariant.
\end{proof}

\begin{lem}\label{lem:left-map-pullback}
  If $\iE$ has pullbacks, then left maps are stable under pullbacks in $\sE$.
  \end{lem}
  \begin{proof}
    For each $n \geq 0$, the evaluation functor $\ev_n \colon \sE \to \iE$ preserves pullbacks. The result then follows by pullback composition and cancellation.
  \end{proof}

\begin{lem}\label{lem:left-map-comp-cancel} For any composable pair of maps $f$ and $g$, if $g$ is a left map, then $f$ is a left map if and only if $gf$ is a left map.
\end{lem}
\begin{proof} Pullback squares in an $\infty$-category satisfy pasting and one-sided cancellation.
\end{proof}

\begin{lem}\label{lem:left-fib-fiberwise-we} A map $h$ between left maps over a common base is an equivalence in $\smE$ if and only if its component $h_{0}$ is an equivalence in $\mE$.
\end{lem}
\begin{proof} Using the pullback squares \eqref{eq:left-map-square}, the component at $h_{m}$ is a pullback of the equivalence $h_0$.
\end{proof}

We will on occasion employ the notion of left map in the special case when the $\infty$-category $\sE$ is actually the $1$-category $\sSet$.

\begin{ex}\label{ex:opfibration-strict-left-map}
  A functor $F \colon \cC \to \cD$ of small 1-categories is a discrete opfibration if and only if the nerve $NF \colon N\cC \to N\cD$ is a left map in $\sSet$.
\end{ex}

Now we assume that $\iE$ is an $\infty$-topos.
Recall that the discrete embedding $\iS \to \sE$ defines a canonical cosimplicial object $\Delta^\bullet \colon \DDelta \to \sE$. In $\sE$, $\Delta^0$ is the terminal object, while we refer to $\Delta^1$ as the \textbf{walking arrow}, and write $\iota_0, \iota_1 \colon \Delta^0 \to \Delta^1$ for its source and target respectively.

\begin{lem}[{\cite[2.1.3]{KV}, \cite[4.1.3]{martini}, \cite[A.27]{RS}}]\label{lem:covariant-fib} A map $f \colon Y \to X$ in $\sE$ is a left map if and only if it is internally orthogonal to $\iota_0 \colon \Delta^0 \to \Delta^1$, meaning that
  \begin{equation}\label{eq:internally-orthogonal} \begin{tikzcd} Y^{\Delta^1} \arrow[d, "\ev_0"'] \arrow[dr, phantom, "\lrcorner" very near start]\arrow[r, "f^{\Delta^1}"] & X^{\Delta^1} \arrow[d, "\ev_0"] \\ Y \arrow[r, "f"'] & X \end{tikzcd}\end{equation} is a pullback in $\sE$.
\end{lem}
\begin{proof}
We adapt the argument of \cite[2.1.3]{KV} to this setting.

Assuming $f$ is internally orthogonal to $\iota_0$, then exponentiation, as a right adjoint, preserves this pullback, giving rise to pullback squares of the form
\begin{equation}\label{eq:exponentiated-internal-orthogonal} \begin{tikzcd} Y^{\Delta^1 \times \Delta^n} \arrow[d, "\ev_0"'] \arrow[dr, phantom, "\lrcorner" very near start]\arrow[r, "f^{\Delta^1 \times \Delta^n}"] & X^{\Delta^1 \times \Delta^n} \arrow[d, "\ev_0"] \\ Y^{\Delta^n} \arrow[r, "f^{\Delta^n}"'] & X^{\Delta^n} \end{tikzcd}\end{equation}
for $n \geq 0$. There is a retract diagram of discrete simplicial spaces
\[ \begin{tikzcd} \Delta^0 \arrow[d, "\iota_0"'] \arrow[r] & \Delta^n \arrow[d, "\iota_0"] \arrow[r] & \Delta^0 \arrow[d, "\iota_0"] \\ \Delta^{n+1} \arrow[r, "\gamma"'] & \Delta^1 \times \Delta^n \arrow[r, "\rho"'] & \Delta^{n+1} \rlap{,} \end{tikzcd} \qquad \gamma(i) \coloneqq \begin{cases} (0,0) & i = 0 \\ (1, i-1) & i > 0 \end{cases} \qquad \rho(i,j) \coloneqq i(j+1) \rlap{.} \]
Thus, the square below-left
\[ \begin{tikzcd} Y^{\Delta^{n+1}} \arrow[d, "\ev_0"'] \arrow[dr, phantom, "\lrcorner" very near start]\arrow[r, "f^{\Delta^{n+1}}"] & X^{\Delta^{n+1}} \arrow[d, "\ev_0"]  & & Y_{n+1} \arrow[d, "\ev_0"'] \arrow[r, "f_{n+1}"] \arrow[dr, phantom, "\lrcorner" very near start] & X_{n+1} \arrow[d, "\ev_0"]\\ Y \arrow[r, "f"'] & X \rlap{,} & & Y_0 \arrow[r, "f_0"'] & X_0\end{tikzcd}
\]
as a retract of \eqref{eq:exponentiated-internal-orthogonal} is also a pullback. This pullback is preserved by the functor $\ev_0 \colon \sE \to \iE$ which yields the pullback square above-right, as the diagram of ``zero spaces'' in $\mE$ associated to the above-left diagram in $\smE$. Thus $f$ is a left map.

Conversely, suppose $f$ is a left map. Since pullbacks in $\sE$ are created by the evaluation functors $\ev_n \colon \sE \to \iE$ for $n \geq 0$, it suffices to show that the image of the square \eqref{eq:internally-orthogonal} under each of these is a pullback in $\iE$. In the case $n=0$ this is the pullback of \eqref{eq:left-map-square} with $m=1$, so it remains to consider the case $n >0$. By pullback composition and cancellation, it suffices to show that the outer rectangle
\[
  \begin{tikzcd} (Y^{\Delta^1})_n \arrow[d, "\ev_0"'] \arrow[r, "{(f^{\Delta^1})_n}"] & (X^{\Delta^1})_n \arrow[d, "\ev_0"] \\ Y_n \arrow[r, "f_n"] \arrow[d, "\ev_0"'] \arrow[dr, phantom, "\lrcorner" very near start] & X_n \arrow[d, "\ev_0"] \\ Y_0 \arrow[r, "f_0"'] & X_0 \end{tikzcd}\] defines a pullback in $\iE$. Note that this outer rectangle is the diagram on zero-spaces for the square of simplicial objects
  \begin{equation}\label{eq:outer-simplicial-level-square} \begin{tikzcd} Y^{\Delta^1 \times \Delta^n} \arrow[d, "\ev_0"'] \arrow[r, "f^{\Delta^1 \times \Delta^n}"] & X^{\Delta^1 \times \Delta^n} \arrow[d, "\ev_0"] \\ Y \arrow[r, "f"'] & X \rlap{.} \end{tikzcd} \end{equation}
  The discrete simplicial space $\Delta^1 \times \Delta^n$ is a colimit of a diagram that glues together $n+1$ copies of the simplicial space $\Delta^{n+1}$ along $n$ interior boundary faces $\Delta^n$, and this colimit is preserved by the left adjoint $\Delta \colon \sS \to \sE$. Thus the square \eqref{eq:outer-simplicial-level-square} is a colimit of squares of the form
  \[ \begin{tikzcd} Y^{\Delta^m} \arrow[d, "\ev_0"'] \arrow[r, "f^{\Delta^m}"] & X^{\Delta^m} \arrow[d, "\ev_0"] \\ Y \arrow[r, "f"'] & X \rlap{.} \end{tikzcd} \]
    for $m = n$ or $m=n+1$. At the level of zero-spaces, each of these squares are pullbacks since $f$ is a left map, and thus, by descent, the diagram on zero-spaces of \eqref{eq:outer-simplicial-level-square} is too.
\end{proof}

For later use, we note a perhaps unexpected example.

\begin{ex}\label{ex:higher-cod-left-map}
  For $k \leq n \in \NN$, the inclusion $i_{\mathup{last}} \colon \Delta^k \to \Delta^n$ that is full on the final $k+1$ vertices is a left map in $\sS$ and in $\sE$.  In $\sS$, the square \eqref{eq:left-map-square} specializes to the diagram of discrete spaces
  \[ \begin{tikzcd} \DDelta([m],[k]) \arrow[d, "i_{\mathup{last}}"'] \arrow[dr, phantom, "\lrcorner" very near start] \arrow[r, "\ev_0"] & \DDelta([0],[k]) \arrow[d, "i_{\mathup{last}}"] \\ \DDelta([m],[n]) \arrow[r, "\ev_0"'] & \DDelta([0],[n]) \rlap{,} \end{tikzcd}\]
  which is a pullback square expressing the orthogonality of the initial functor $\iota_0 \colon \Delta^0 \to \Delta^m$ and the discrete opfibration $i_{\mathup{last}} \colon \Delta^k \to \Delta^n$ of 1-categories. This pullback is preserved by the discrete embedding $\Delta \colon \iS \to \iE$.
\end{ex}

\section{It's all in the weights}\label{sec:weights}

In this section, we establish the core technical ingredients that will be used to prove our main theorems in a much simpler setting: the category of simplicial sets. We think of simplicial sets as ``weights'' that can be deployed in a simplicial $\infty$-topos $\smE$ to construct a \emph{weighted limit}, an object in $\mE$.

In \cref{ssec:leibniz}, we briefly recall the general setting of weighted co/limits, which we use to illustrate even more general interactions between two-variable adjunctions and weak factorization systems. In \cref{ssec:weights}, we introduce the two central weights, which have the form
\[ \begin{tikzcd} {[n]} \arrow[r, "\Ws"] & \slice{\sSet}{\Delta^n} \rlap{,} & {[n]}^\op \arrow[r, "\Wc"] & \slice{\sSet}{\Delta^n} \end{tikzcd} \]
and establish their properties.
The functors $\arrtofun_n$ and $\funtoarr_n$ are defined by limits weighted by $\Wc$ and $\Ws$ respectively composed with suitable reindexing functors.
In this section, we reason on the left about the reindexed weighted colimits
\[  \begin{tikzcd}[column sep=large, row sep=small] \arrtofunwgt \colon \sSet^{[n]} \arrow[r, "{(\Delta^n)^*}"] & (\slice{\sSet}{\Delta^n})^{[n]} \arrow[r, "\Wc\otimes_{[n]}-"] & \slice{\sSet}{\Delta^n} \rlap{,} \\  \funtoarrwgt \colon \slice{\sSet}{\Delta^n}  \arrow[r, "\Ws\otimes -"] & (\slice{\sSet}{\Delta^n})^{[n]} \arrow[r, "(\Delta^n)_!"]   & \sSet^{[n]}  \rlap{.} \end{tikzcd}\]
The results in this section are fundamentally 1-categorical and are developed more generally for a generic 1-category $\cC$ in place of the ordinal category $[n]$.

In order to analyze the interaction between $\arrtofun_n$ and $\funtoarr_n$ and the left fibrations in $\smE$, we study the interaction between $\arrtofunwgt_n$ and $\funtoarrwgt_n$ and a particular class of maps in the category of simplicial sets: the \emph{left anodyne} maps. We define left anodyne maps and establish their basic properties, which are more or less classical, in \S\ref{ssec:left-anodyne}.
In \cref{ssec:homotopical}, we prove homotopical properties of $\arrtofunwgt_n$ and $\funtoarrwgt_n$ in these terms.

In \cref{ssec:naturality}, we consider naturality in the indexing category with respect to arbitrary functors $F \colon \cC \to \cD$ between 1-categories.
We show that $\arrtofunwgt$ is lax natural and $\funtoarr$ is pseudonatural.
These results will be used to show that $\funtoarr_n$ defines a component of a natural transformation between simplicial objects on the $\infty$-topos level. We will also apply them in the case of the projection $! \colon \cC \to {[0]}$ to generalize this constructions to slices of model topoi over a fixed context $\Gamma$.

\subsection{Leibniz adjunctions and weak factorization systems}\label{ssec:leibniz}

We briefly recall the general setting of weighted co/limits.

\begin{rec}\label{rec:weighted-setting}
  Suppose $\cK$, $\cM$, and $\cN$ are bicomplete 1-categories equipped with a two-variable adjunction defined by functors
  \[ \begin{tikzcd} \cK \times \cM \arrow[r, "\otimes"] & \cN \rlap{,} & \cK^\op \times \cN \arrow[r, "{\{-,-}\}"] & \cM \rlap{,} & \cM^\op \times \cN \arrow[r, "\Map"] & \cK \rlap{.} \end{tikzcd}\]
  Often $\cM=\cN$ (or even $\cK=\cM=\cN$) and these functors are given by some sort of tensor, cotensor, and enrichment, respectively. This explains our notation.

  Then for any small 1-categories $\cA$, $\cB$, and $\cC$, the functors
  \[ \begin{tikzcd}[column sep=2em] \cK^{\cB^\op \times \cA} \times \cM^{\cC^\op \times \cB} \arrow[r, "\otimes_\cB"] & \cN^{\cC^\op \times \cA} \rlap{,} & [-2em] (\cK^{\cB^\op \times \cA})^\op \times \cN^{\cC^\op \times \cA} \arrow[r, "{\{-}\}^\cA"] & \cM^{\cC^\op \times \cB} \rlap{,} & [-2em] (\cM^{\cC^\op \times \cB})^\op \times \cN^{\cC^\op \times \cA} \arrow[r, "\Map^\cC"] & \cK^{\cB^\op \times \cA} \end{tikzcd}\] defined by an coend over $\cB$ in the first case, an end over $\cA$ in the second, and an end over $\cC^\op$ in the third again define a two-variable adjunction. In particular, when $\cC={[0]}$ and $W \in \cK^{\cB^\op \times \cA}$ is a fixed ``weight'', the \textbf{weighted colimit} functor $W \otimes_\cB - \colon \cM^\cB \to \cN^\cA$ and \textbf{weighted limit} functor $\{W,-\}^\cA \colon \cN^\cA \to \cM^\cB$ define an adjunction between $\cB$-indexed diagrams and $\cA$-indexed diagrams.
\end{rec}

Continuing in the same setting, suppose $\cK$, $\cM$, and $\cN$ each are equipped with a weak factorization system such that the two-variable adjunction $(\otimes, \{-,-\}, \Map)$ is a \textbf{Leibniz two-variable adjunction}, meaning that the pushout Leibniz bifunctor $-\hat{\otimes}- \colon \cK^{[1]} \times \cM^{[1]} \to \cN^{[1]}$ preserves left classes or equivalently that the pullback Leibniz bifunctors $\widehat{\{-,-\}}$ or $\widehat{\Map}$ for either right adjoint preserves right classes \cite[C.2.11]{RV}. Note the Leibniz bifunctors again define a two-variable adjunction \cite[4.12]{RV-Reedy}.

From a weak factorization system $(\leftclass,\rightclass)$ on a category $\cK$, we may define \textbf{projective} and \textbf{injective} weak factorization systems on the diagram category $\cK^\cA$, which exist under certain conditions. (In particular, when $(\leftclass,\rightclass)$ is cofibrantly generated, $\cK$ is locally presentable, and $\cA$ is small, the projective weak factorization system exists.) The right class of the projective weak factorization system and the left class of the injective weak factorization system are those maps that are pointwise in $\rightclass$ or $\leftclass$, respectively. We quickly recall two special cases of a general result, which holds whenever the named weak factorization systems exist.

\begin{lem}[{\cite{gambino}}]\label{lem:leibniz-weighted}\leavevmode
\begin{parts}
  \item\label{lem:leibniz-weighted:tensor-to-E} The two-variable adjunction $(\otimes_\cB, \{-,-\}, \Map)$ is Leibniz for either
  \begin{parts}
    \item the projective weak factorization system on $\cK^{\cB^\op}$, the injective weak factorization system on $\cM^\cB$, and the given weak factorization system on $\cN$;
    \item the injective weak factorization system on $\cK^{\cB^\op}$, the projective weak factorization system on $\cM^\cB$, and the given weak factorization system on $\cN$.
\end{parts}
\item\label{lem:leibniz-weighted:tensor-with-E} The two-variable adjunction $(\otimes, \{-,-\}^\cA, \Map)$ is Leibniz for either
\begin{parts}
  \item\label{lem:leibniz-weighted:tensor-with-E:projective} the projective weak factorization system on $\cK^{\cA}$, the given weak factorization system on $\cM$, and the projective weak factorization system on $\cN^\cA$;
  \item\label{lem:leibniz-weighted:tensor-with-E:injective} the injective weak factorization system on $\cK^{\cA}$, the given weak factorization system on $\cM$, and the injective weak factorization system on $\cN^\cA$.
\end{parts}
\end{parts}
\end{lem}
\begin{proof}
By transposing, all of these statements are equivalent to one which asserts that one of the given bifunctors carries pointwise maps in the appropriate class to pointwise maps, which is true by the assumption on the original two-variable adjunction $(\otimes, \{-,-\}, \Map)$.
\end{proof}

Frequently we encounter Leibniz two-variable adjunctions in a setting where each of the three categories $\cK$, $\cM$, and $\cN$ each have a pair of nested weak factorization systems, e.g., in the presence of model structures on all three categories.
In that setting, one refers to the two-variable adjunction as a \textbf{Quillen two-variable adjunction} when it is Leibniz with respect to seven of the eight possible choices of weak factorization system, the exception being when the weak factorization systems on $\cK$ and $\cM$ are chosen to have the largest left class, while the weak factorization system on $\cN$ is chosen to have the smallest left class.
We extend that terminology to any nested pair of weak pair of weak factorization systems whether or not we are considering a model structure.

Given such a Quillen two-variable adjunction, \cref{lem:leibniz-weighted} then shows that the induced weighted (co)limit two-variable adjunction is Quillen for the corresponding projective and/or injective lifts of the nested pairs of weak factorization systems.

\subsection{Left anodyne maps}\label{ssec:left-anodyne}

The set of left horn inclusions of simplicial sets generates a weak factorization system whose left class is commonly referred to as the class of left anodyne maps. We will see that these maps can be thought of as avatars, on the weights side, for the left fibrations defined in a simplicial $\infty$-topos. In the setting of simplicial sets, the right class of the weak factorization system whose left class is the left anodyne maps is commonly referred to as the class of \textbf{left fibrations}. While closely related, this notion of left fibrations is not a special case of the notion we will introduce in \S\ref{ssec:left-fibration} here because $\Set$ is a 1-topos and not an $\infty$-topos.

\begin{defn}[{\cite{joyal-quasi}}]
The \textbf{left anodyne} morphisms are those maps of simplicial sets defined as retracts of transfinite composites of pushouts of coproducts of the left horn inclusions $\Lambda^n_k \hookrightarrow \Delta^n$ for $n \geq 1$ and $0 \leq k < n$.
\end{defn}

For later use, we recall a few facts about left anodyne maps of simplicial sets and related classes of maps.

\begin{lem}[{\cite[2.1.2.7]{lurie-topos}}]\label{lem:left-anodyne-quillen}
The cartesian self-enrichment of $\sSet$ is Quillen for the (monomorphism, trivial fibration) and (left anodyne, left fibration) weak factorization systems.
\end{lem}

\begin{lem}\label{lem:left-anodyne-cancel}
Consider a bicomplete simplicially enriched category $\cE$ with a nested pair of weak factorization systems $(\leftclass_0, \rightclass_0)$ and $(\leftclass_1, \rightclass_1)$ where $\leftclass_0 \subseteq \leftclass_1$ and whose enrichment is Quillen with respect to the (monomorphism, trivial fibration) and (left anodyne, left fibration) weak factorization systems on $\sSet$.
Then $\leftclass_0$-maps right cancel among $\leftclass_1$-maps in $\cE$.
\end{lem}

\begin{proof}
Working in a coslice, we show that every $\leftclass_1$-map $j \colon A \to B$ between $\leftclass_0$-objects is in $\leftclass_0$.
The inherited simplicial enrichment of the coslice is Quillen for the inherited nested pair of weak factorization systems since the forgetful functor to $\cE$ preserves cotensors.
In the square
\[
\begin{tikzcd}
  A
  \ar[r, "j"]
  \ar[d]
&
  B
  \ar[d, "1 \otimes B"]
\\
  B \sqcup_A \Delta^1 \otimes A
  \ar[r]
&
  \Delta^1 \otimes B \rlap{,}
\end{tikzcd}
\]
the left vertical map is the first factor of the ``mapping cylinder'' factorization of $j$; as it writes as a composite of base changes of $\emptyset \to B$ and the Leibniz tensor with $\partial \Delta^1 \hookrightarrow \Delta^1$ of $\emptyset \to A$, it is in $\leftclass_0$.
The bottom map is the Leibniz tensor with $\{0\} \hookrightarrow \Delta^1$ of $j$ and hence in $\leftclass_0$.
As $\{1\} \hookrightarrow \Delta^1$ is a split monomorphism, so is the right vertical map.
This makes $j$ a retract of a map in $\leftclass_1$.
\end{proof}

\begin{lem}\label{lem:left-anodyne-initial-vertex} The left anodyne maps of simplicial sets are generated under retract, transfinite composition, pushout, coproduct, and right cancellation among monomorphisms by either
  \begin{parts}
  \item\label{lem:left-anodyne-initial-vertex:point} the initial point inclusions $\iota_0 \colon \Delta^0 \to \Delta^m$ for all $m \geq 1$, or
  \item\label{lem:left-anodyne-initial-vertex:segment} the initial segment inclusions $\iota_{\leq m-1} \colon \Delta^{m-1} \to \Delta^m$ for all $m \geq 1$.
  \end{parts}
\end{lem}
\begin{proof}
As the left class of a weak factorization system, the left anodyne maps are closed under retract, transfinite composition, pushout, and coproduct.
By \cref{lem:left-anodyne-cancel}, they are also closed under right cancellation among monomorphisms.
Since the initial endpoint or initial segment inclusions are left anodyne, the class of maps described in the statement is contained within the left anodyne maps.

To see that the left anodyne maps are generated under these closure operations by the maps in \ref{lem:left-anodyne-initial-vertex:point} or \ref{lem:left-anodyne-initial-vertex:segment}, we argue by induction in $n$, that the left horn inclusions $\Lambda^n_k \hookrightarrow \Delta^n$ for $n \geq 1$ and $0 \leq k < n$ belong to the class of maps described in the statement. In the base case $n=1$, the left horn $\Lambda^1_0 \hookrightarrow \Delta^1$ coincides with the initial point inclusion $\iota_0 \colon \Delta^0 \to \Delta^1$ and the initial segment inclusion $\iota_{\leq 0} \colon \Delta^0 \to \Delta^1$.

As the domain of a generic left horn inclusion $\Lambda^n_k \hookrightarrow \Delta^n$ includes the final face of $\Delta^n$, we have a commutative diagram:
\[ \begin{tikzcd} \Delta^0 \arrow[dr, "\iota_0"'] \arrow[r, "\iota_0"] & \Lambda^n_k  \arrow[d, hook] & \Delta^{n-1} \arrow[dl, "\iota_{\leq n-1}"] \arrow[l, "\iota_{\leq n-1}"'] \\ & \Delta^n \rlap{.} \end{tikzcd} \]
The top horizontal maps factor as composites of pushouts of left horn inclusions with codomain $\Delta^m$ for $m < n$. In the inductive step, we may assume that these left horn inclusions belong to our class of maps, so $\iota_0 \colon \Delta^0 \to \Lambda^n_k$ is there as well. Now the left horn inclusion $\Lambda^n_k \hookrightarrow \Delta^n$ belongs as well, as a monomorphism obtained by right cancellation with the appropriate generating morphism.
\end{proof}

The class of left anodyne maps lifts to define a class of projective left anodyne maps in the functor category $\sSet^\cC$ indexed by any small category $\cC$ and also a class of left anodyne maps in the slice $\slice{\sSet}{K}$ over a fixed simplicial set $K$, defined to be those maps that forget to left anodyne maps of simplicial sets. We consider analogously-defined classes of monomorphisms in $\slice{\sSet}{K}$ and projective monomorphisms in $\sSet^\cC$.
We recall some facts from \cite{GambinoSattler}, instantiated to the cases of $\sSet$ and $\sSet^{\cC}$.

\begin{defn}
Given maps $f_0, f_1 \colon A \to B$ in $\sSet$ or $\sSet^{\cC}$, an (elementary, directed) \textbf{homotopy} $u \colon f_0 \sim f_1$ is a map $\Delta^1 \times A \to B$ restricting to $f_i$ along $\{i\} \hookrightarrow \Delta^1$ for $i = 0, 1$.
A map $f \colon A \to B$ is a \textbf{strong 0-oriented homotopy equivalence} if we have $g \colon B \to A$ together with $u \colon g f \sim \id_A$ and $v \colon f g \sim id_B$ such that $f u = v f$.
\end{defn}

\begin{ex}\label{initial-section-is-she}
Consider a fully faithful left adjoint $i \colon A \to B$ between small categories with right adjoint $r$, unit $\eta \colon \id \to r i$, and counit $\epsilon \colon i r \to \id$. Recall that $i$ being fully faithful means that $\eta$ is invertible. The functor $i$ is then said to be a \textbf{left adjoint right inverse} or \textbf{initial section} of the functor $r$. When $i$ is an initial section, $N i \colon N A \to N B$ is a strong 0-oriented homotopy equivalence in $\sSet$: the counit and invertible unit define natural transformations $\eta^{-1} \colon r i \to \id$ and $\epsilon \colon i r \to \id$ satisfying $i \eta^{-1}$ = $\epsilon i$.

For the same reasons, the nerve of an initial section $i \colon A \to B$ in the 2-category $\Cat^\cC$ is a strong 0-oriented homotopy equivalence in $\sSet^{\cC}$.
\end{ex}

\begin{ex}\label{initial-segment-is-she}
Initial segments $\iota_{\leq k} \colon \Delta^k \to \Delta^n$ are strong 0-oriented homotopy equivalences in $\sSet$.
This is a special case of \cref{initial-section-is-she} (with initial section $\iota_{\leq} \colon [k] \to [n]$).
\end{ex}

By \cite[4.3]{GambinoSattler}:

\begin{lem}\label{she-is-retract}
Strong 0-oriented homotopy equivalences are retracts of Leibniz products with $\{0\} \hookrightarrow \Delta^1$ of themselves. \qed
\end{lem}

\begin{cor}\label{she-is-left-anodyne}
  Monomorphisms in $\sSet$ that are strong 0-oriented homotopy equivalences are left anodyne, and projective monomorphisms in $\sSet^{\cC}$ that are strong 0-oriented homotopy equivalences are projective left anodyne.
\end{cor}
\begin{proof}
Pushout product with $\{0\} \hookrightarrow \Delta^1$ sends (projective) monomorphisms to (projective) left anodynes by \cref{lem:left-anodyne-quillen} and \cref{lem:leibniz-weighted}.\ref{lem:leibniz-weighted:tensor-with-E}.\ref{lem:leibniz-weighted:tensor-with-E:projective}.
The statements thus follow by \cref{she-is-retract} and closure of (projective) left anodynes under retract.
\end{proof}

\begin{defn}
Call a map of simplicial sets a \textbf{right transport fibration} if it right lifts against Leibniz products with $\{1\} \hookrightarrow \Delta^1$ of maps $\emptyset \to A$, that is, the map $\{1\} \times A \hookrightarrow \Delta^1 \times A$ for a simplicial set $A$.
\end{defn}

\begin{ex}
Joyal's right fibrations of simplicial sets, those maps that lift against right horn inclusions $\Lambda^n_k \to \Delta^n$ for $n \geq 0$ and $0 < k \leq n$, are right transport fibrations.
This is because right fibrations are alternatively generated by Leibniz products with $\{1\} \hookrightarrow \Delta^1$ of general cofibrations, in particular those with initial domain \cite[2.1.2.6]{lurie-topos}.
\end{ex}

\begin{ex}\label{grothendieck-fibration-is-weak-fibration}
Consider a Grothendieck fibration $p \colon E \to B$ of small categories.
Then $N p \colon N E \to N B$ is a right transport fibration.
To see this, recall that the left adjoint to the nerve preserves finite products.
The required lifting thus transposes to lifting $D \times \{0\} \hookrightarrow D \times [1]$ against $p$ for small categories $D$.
\end{ex}

For completeness, we note that the previous two examples have the following joint generalization.
We will not need this in our work here.

\begin{ex}
Every cartesian fibration of quasicategories is a right transport fibration.
This is can easily be seen by working with marked simplicial sets \cite[\S3.1]{lurie-topos}, as follows.
Given a map $p \colon Y \to X$, we mark all edges of $X$ and the cartesian edges of $Y$.
If $p$ is a cartesian fibration, it right lifts against Leibniz products of the right endpoint inclusion of the marked edge with cofibrations \cite[3.1.1.10, 3.1.2.3]{lurie-topos}.
Taking $\emptyset \to A$ with the minimal marking as the cofibration, this transposes to right lifting of $p$ against the Leibniz product of $\{1\} \hookrightarrow \Delta^1$ and $\emptyset \to A$ in simplicial sets.
\end{ex}

\begin{lem}\label{frobenius-she}
Pullback along right transport fibrations preserves strong 0-oriented homotopy equivalences.
\end{lem}

\begin{proof}
This is the case $k = 0$ of \cite[4.7]{GambinoSattler}, noting that the only lifting of the fibration required in the proof is against the Leibniz product with $\{1\} \hookrightarrow \Delta^1$ of the map $\emptyset \to X$.
\end{proof}

\begin{cor}\label{frobenius-left-anodyne}
Pullback along right transport fibrations preserves left anodynes.
\end{cor}

\begin{proof}
It suffices by \cref{lem:left-anodyne-initial-vertex} to show that pullback along right transport fibrations sends initial segments to left anodynes.
Initial segments are cofibrations that are strong 0-oriented homotopy equivalences by \cref{initial-segment-is-she}, which in turn are left anodyne by \cref{she-is-left-anodyne}.
Hence, the claim follows from stability of cofibrations under pullback and \cref{frobenius-she}.
\end{proof}

\subsection{The weights}\label{ssec:weights}

We now introduce a pair of weights of interest.

\begin{defn}
Let $\cC$ be any small 1-category. Its slices and coslices, with the composition functors, assemble into diagrams
\[ \begin{tikzcd}[row sep=tiny] \cC \arrow[r, "\slice{\cC}{-}"] & \slice{\Cat}{\cC} \rlap{,} & \cC^\op \arrow[r, "{\coslice{\cC}{-}}"] & \slice{\Cat}{\cC} \rlap{,} \\ x \arrow[r, maps to] & \slice{\cC}{x} \rlap{,} & x \arrow[r, maps to ] & \coslice{\cC}{x} \rlap{,} \end{tikzcd}\]
which compose with the nerve embedding to define weights valued in $\slice{\sSet}{N\cC}$.
We denote these weights using superscripts for ``slice'' and ``coslice''
\[ \begin{tikzcd}[row sep=tiny] \cC \arrow[r, "\Ws_\cC"] & \slice{\sSet}{N\cC} \rlap{,} & \cC^\op \arrow[r, "{\Wc_\cC}"] & \slice{\sSet}{N\cC} \rlap{,} \\ x \arrow[r, maps to] & N(\slice{\cC}{x}) \rlap{,} & x \arrow[r, maps to ] & N(\coslice{\cC}{x}) \end{tikzcd}\]
and omit the subscripts whenever the ambient category $\cC$ is clear from context.
\end{defn}

The weights $\Ws$ and $\Wc$ are the opposites of the weights used by Bousfield and Kan to define the homotopy limit and homotopy colimit of a $\cC$-indexed diagram, respectively \cite[XI.2.2]{bousfield-kan}.\footnote{As explained in \cite[18.1.11]{hirschhorn}, the opposites appear as a result of different conventions used in the definition of the nerve of a category. These opposites are immaterial to the definition of homotopy co/limits: both the weights $N(\slice{\cC}{-})$ and its opposite $N(\slice{\cC}{-})^\op$ define projective cofibrant replacements of the terminal weight for $\cC$-indexed diagrams (and similarly for the case of the coslices),  so result in weakly equivalent constructions. In our directed setting, the opposites do matter and should be omitted.}

\begin{ex}
  For $\cC = [n]$, the weight $\Ws_{[n]}$ sends $k \in [n]$ to the initial segment inclusion $\iota_{\le k} \colon \Delta^k \to \Delta^n$ in $\slice{\sSet}{\Delta^n}$.
  Similarly, the weight $\Wc_{[n]}$ sends $k \in [n]$ to the final segment inclusion $\iota_{\ge k} \colon \Delta^{n-k} \to \Delta^n$.
\end{ex}

\begin{defn}\label{defn:fun-arr-weights}
Using the weights $\Ws$ and $\Wc$, we define a pair of functors:
\[  \begin{tikzcd}[column sep=large, row sep=small] \arrtofunwgt \colon \sSet^\cC \arrow[r, "{(N\cC)^*}"] & (\slice{\sSet}{N\cC})^\cC \arrow[r, "\Wc\otimes_\cC-"] & \slice{\sSet}{N\cC} \rlap{,} \\  \funtoarrwgt \colon \slice{\sSet}{N\cC}  \arrow[r, "\Ws\otimes -"] & (\slice{\sSet}{N\cC})^\cC \arrow[r, "(N\cC)_!"]   & \sSet^\cC  \rlap{.} \end{tikzcd}\]
\end{defn}

By \cref{rec:weighted-setting}, both functors are left adjoints:
\begin{equation}\label{eq:fun-arr-weights-adj} \begin{tikzcd}[sep=3em] \sSet^\cC \arrow[r, "{(N\cC)^*}"] \arrow[r, adj={yshift=-8pt}] & \arrow[l, bend left, dashed, "{(N\cC)_*}" below](\slice{\sSet}{N\cC})^\cC \arrow[r, "\Wc\otimes_\cC-"] \arrow[r, adj={yshift=-8pt}] & \slice{\sSet}{N\cC}  \arrow[l, bend left, dashed, "{\{\Wc,-\}}" below] \rlap{,} &[-20pt]  \slice{\sSet}{N\cC}  \arrow[r, "\Ws\otimes -"] \arrow[r, adj={yshift=-8pt}] & (\slice{\sSet}{N\cC})^\cC \arrow[r, "(N\cC)_!"] \arrow[l, bend left, dashed, "{\{\Ws,-\}^{\cC}}" below] \arrow[r, adj={yshift=-8pt}] & \sSet^\cC \arrow[l, bend left, dashed, "{(N\cC)^*}" below] \rlap{.} \end{tikzcd}\end{equation}
The right adjoints are avatars of the $\arrtofun$ and $\funtoarr$ constructions, respectively, introduced in \S\ref{ssec:dua-statement}.

We now analyze the round-trip composites of these functors, $\funtoarrwgt \circ \arrtofunwgt \colon \sSet^\cC \to \sSet^\cC$ and $\arrtofunwgt \circ \funtoarrwgt \colon \sSet^{N\cC} \to \sSet^{N\cC}$.
Neither is isomorphic to an identity functor, but each is related to an identity functor by a zigzag of natural transformations.
In each case, we can give simple descriptions of the components of these transformations at discrete representables.
For the purposes of deriving a weak equivalence once we pass to weighted limits in \cref{sec:dua-model}, it will suffice to understand these components.

\begin{rmk}
  Recall that composition defines a functor
  \[
    \begin{tikzcd}
      \coslice{\cC}{j} \times_{\cC} \slice{\cC}{i}
      \ar[r, "\circ"]
      &
      \cC(j, i)
    \end{tikzcd}
  \]
  natural in $j \in \cC^\op$ and $i \in \cC$ and that its target $\cC(j,i)$ is a set.
\end{rmk}

\begin{lem}\label{lem:arr-to-fun-to-arr-point}
  We have a natural transformation
  \[
    \begin{tikzcd}
    &
      \slice{\sSet}{N \cC}
      \ar[dr, bend left=20, "\funtoarrwgt"]
    &
    \\
      \sSet^\cC
      \ar[ur, bend left=20, "\arrtofunwgt"]
      \ar[rr, bend left=26, phantomcenter, "\Downarrow \mathrlap{\kappa}"]
      \ar[rr, "\id"']
    & { } &
      \sSet^\cC
    \end{tikzcd}
  \]
  whose component at a representable $\cC(c,-) \in \Set^\cC \subset \sSet^\cC$ is isomorphic to the natural transformation
  \begin{equation}\label{eq:arr-to-fun-to-arr-component}
    \begin{tikzcd} N(\coslice{\cC}{c} \times_\cC \slice{\cC}{-}) \arrow[r, "\circ"] & \cC(c,-) \rlap{.} \end{tikzcd}
  \end{equation}
\end{lem}
\begin{proof}
  Unfolding the definitions of $\arrtofunwgt$ followed by $\funtoarrwgt$, we construct the natural transformation from their composite by the following steps:
  \[
    \begin{tikzcd}[column sep={between origins,6em}]
  &
    \slice{\sSet}{N \cC}
    \ar[dr, "\Ws \otimes_\cC -"]
    \ar[d, phantomcenter, "\cong"]
  &
  \\
    (\slice{\sSet}{N \cC})^\cC
    \ar[ur, "\Wc \otimes_\cC -"]
    \ar[rr, "N (\coslice{\cC}{-} \times_\cC \slice{\cC}{-}) \otimes_\cC -"]
    \ar[drr, phantomcenter, "\cong"]
  & { } &
    (\slice{\sSet}{N \cC})^\cC
    \ar[d, "(N \cC)_!"]
  \\
    \sSet^\cC
    \ar[u, "(N \cC)^*"]
    \ar[rr, "N (\coslice{\cC}{-} \times_\cC \slice{\cC}{-}) \otimes_\cC -"]
    \ar[rr, phantomcenter, bend right=22, "\Downarrow \mathrlap{{\circ} \otimes_\cC -}"]
    \ar[rr, bend right=40, "{\id \cong \cC(-, -) \otimes_\cC -}"']
  &&
    \sSet^\cC \rlap{.}
  \end{tikzcd}
  \]
  The top triangle commutes up to the isomorphism
  \[
  (\Ws \otimes \Wc)_{j, i} \cong N \coslice{\cC}{j} \times_{N \cC} N \slice{\cC}{i} \cong N(\coslice{\cC}{j} \times_\cC \slice{\cC}{i})
  \]
  in $\slice{\sSet}{N \cC}$ natural in $j \in \cC^\op$ and $i \in \cC$.
  The bottom cell commutes up to an isomorphism of the form
  \[
  S_! (A \times_S S^* B) \cong S_! A \times B
  \]
  for $S \in \sSet$ natural in $A \in \slice{\sSet}{S}$ and $B \in \sSet$.

  We now analyze the natural transformation $({{\circ} \otimes_\cC -}) \colon ({N(\coslice{\cC}{-} \times_\cC \slice{\cC}{-}) \otimes_\cC -}) \to ({\cC(-, -) \otimes_\cC -})$.
  By the coYoneda lemma, its domain and codomain applied to $\cC(c,-)$ simplify as displayed by the vertical isomorphisms
  \[
    \begin{tikzcd}[column sep=huge]
      \int^j N(\coslice{\cC}{j} \times_\cC \slice{\cC}{-}) \times \cC(c,j) \arrow[r,"{\circ \otimes_\cC \cC(c,-)}"] \arrow[d, "\cong"'] & \int^j \cC(j,-) \times \cC(c,j) \arrow[d,"\cong"] \\
      N(\coslice{\cC}{c} \times_\cC \slice{\cC}{-}) \arrow[r,"\circ"'] & \cC(c,-) \rlap{.}
    \end{tikzcd}
  \]
  The bottom horizontal map may be described as follows.
  For $i \in \cC$, the category $\coslice{\cC}{c} \times_\cC \slice{\cC}{i}$ decomposes as a disjoint union indexed by the set of morphisms from $c$ to $i$ in $\cC$, and the composition map sends each component of this disjoint union to its indexing element.
\end{proof}

\begin{lem}\label{lem:fun-to-arr-to-fun-point}
  We have a zigzag of natural transformations
  \[
    \begin{tikzcd}
    &
      \sSet^\cC
      \ar[dr, bend left=20, "\arrtofunwgt"]
    &
    \\
      \slice{\sSet}{N\cC}
      \ar[ur, bend left=20, "\funtoarrwgt"]
      \ar[r, "\dom^*"]
      \ar[rr, bend right=50, "\id"', ""{name=id}]
      \ar[rr, bend right=28, phantomcenter, "\Uparrow \mathrlap{\nu}"]
      \ar[rr, bend left=28, phantomcenter, "\Downarrow \mathrlap{\mu}"]
    & \slice{\sSet}{N\cC^{[1]}} \ar[r, "\cod_!"] &
      \slice{\sSet}{N\cC} \rlap{.}
    \end{tikzcd}
  \]
  The component of this zigzag at a point $c \colon \Delta^0 \to N\cC$ is isomorphic to the zigzag
  \[
    \begin{tikzcd}[column sep={between origins,5em}]
      N\coslice{\cC}{c} \arrow[dr, "\cod"'] \arrow[r, "\cong"]
      & N\coslice{\cC}{c} \arrow[d, "\cod" description] & \Delta^0 \arrow[dl, "c"] \arrow[l, "\id_c"'] \\
      & N\cC
    \end{tikzcd}
  \]
  in $\slice{\sSet}{N\cC}$.
\end{lem}
\begin{proof}
  The composite of $\funtoarrwgt$ followed by $\arrtofunwgt$ simplifies as follows:
  \begin{equation}\label{eq:fun-to-arr-to-fun-map}
    \begin{tikzcd}
      &&
      \sSet^\cC \arrow[dd, phantom, "\cong"]
      \ar[dr, "(N \cC)^*"]
      \\&
      (\slice{\sSet}{N \cC})^\cC
      \ar[ur, "(N \cC)_!"]
      \ar[dr, "\pi_2^*"]  \arrow[dd, phantom, "\cong"]
      &&
      (\slice{\sSet}{N \cC})^\cC
      \ar[dr, "\Wc \otimes_\cC -"]  \arrow[dd, phantom, "\cong"]
      \\
      \slice{\sSet}{N \cC}
      \ar[ur, "\Ws \otimes -"]
      \ar[dr, "\pi_2^*"']
      &&
      (\slice{\sSet}{N(\cC \times \cC)})^C
      \ar[ur, "(\pi_1)_!"]
      \ar[dr, "\pi_1^* \Wc \otimes_\cC -"]  \arrow[d, phantom, "\cong"]
      &&
      \slice{\sSet}{N \cC} \rlap{.}
      \\&
      \slice{\sSet}{N(\cC \times \cC)}
      \ar[ur, "\pi_2^* \Ws \otimes -"]
      \ar[rr, "\int^{i \in \cC}  N(\coslice{\cC}{i} \times \slice{\cC}{i}) \otimes -"']
      &~&
      \slice{\sSet}{N(\cC \times \cC)}
      \ar[ur, "(\pi_1)_!"']
    \end{tikzcd}
  \end{equation}
  The most interesting of these occupies the bottom triangle, which commutes up to the isomorphism computed in the slice over $N(\cC \times \cC)$:
  \[
    \pi_1^* \Wc \otimes_\cC \pi_2^* \Ws
    \cong
    \int^{i \in \cC} (N \coslice{\cC}{i} \times N \cC) \times_{N \cC \times N \cC} (N \cC \times N \slice{\cC}{i})
    \cong
    \int^{i \in \cC}  N(\coslice{\cC}{i} \times \slice{\cC}{i})
  \]
  where the last object lives over $N(\cC \times \cC)$ via the pairing of the codomain projection from the coslice and the domain projection from the slice.

  Thus, the composite $\arrtofunwgt\circ \funtoarrwgt$ is isomorphic to the top composite below. This compares with the identity as follows:
  \begin{equation}\label{eq:fun-to-arr-to-fun-zigzag}
      \begin{tikzcd}[column sep=small]
      &
      \slice{\sSet}{N(\cC \times \cC)}
      \ar[rrrr, bend left=5, "\int^{i \in \cC}  N(\coslice{\cC}{i} \times \slice{\cC}{i}) \otimes -"]
      \ar[rrrr, phantom, "\Downarrow \mathrlap{{\circ} \otimes -}"]
      \ar[rrrr, bend right=5, "N \cC^{[1]} \otimes -"']
      \ar[dr, "{(\cod, \dom)^*}" description]
      && & &
      \slice{\sSet}{N(\cC \times \cC)}
      \ar[dr, "(\pi_1)_!"]
      \\
      \slice{\sSet}{N \cC}
      \ar[ur, "\pi_2^*"]
      \ar[rr, "\dom^*"']
      \ar[drrr, bend right=10, equals]
      &  \arrow[u, phantom, "\cong"]  & \slice{\sSet}{N \cC^{[1]}}
      \ar[rr, equals]
      \ar[dr, "\id^*"']
      & \ar[d, phantom, "\Uparrow \mathrlap{\epsilon}"] &
      \slice{\sSet}{N \cC^{[1]}}
      \ar[ur, "{(\cod, \dom)_!}" description]
      \ar[rr, "\cod_!"']
      & &
      \slice{\sSet}{N \cC} \rlap{.}
      \\& \arrow[ur, phantom, "\cong" near end] & &
      \slice{\sSet}{N \cC}
      \ar[ur, "\id_!"']
      \ar[urrr, bend right=10, equals]
    \end{tikzcd}
  \end{equation}
  For the top cell, note that composition in $\cC$ defines the following map of simplicial sets:
  \[
    \begin{tikzcd}
      \int^{i \in \cC} N(\coslice{\cC}{i} \times \slice{\cC}{i})
      \ar[dr]
      \ar[rr, dashed, "\circ"]
      &&
      N \cC^{[1]}
      \ar[dl, "{(\cod, \dom)}"]
      \\&
      N(\cC \times \cC) \rlap{.}
    \end{tikzcd}
  \]
  Now we describe the component of \eqref{eq:fun-to-arr-to-fun-zigzag} at a point $c \colon \Delta^0 \to N\cC$.
  At a generalized element $a \colon A \to N\cC$, the component is obtained from the maps
  \[ \begin{tikzcd}[column sep={between origins,8em}]
      \int^{i \in \cC} N(\coslice{\cC}{i} \times \slice{\cC}{i}) \arrow[dr, "{(\cod,\dom)}"'] \arrow[r, "\circ"]
      & N\cC^{[1]} \arrow[d, "{(\cod,\dom)}" description] & N\cC \arrow[dl, "{(\cod,\dom)}"] \arrow[l, "\id"'] \\ & N(\cC\times\cC) \end{tikzcd}\]
  by pulling back along $N(a \times \cC) \colon N(A \times \cC) \to N(\cC \times \cC)$ and then regarding the result as a map over $N\cC$ via the codomain projection. When $a$ is a point $c$, the result has the form
  \[
    \begin{tikzcd}[column sep={between origins,8em}]
      \int^{i \in \cC}  N(\coslice{\cC}{i}) \times \cC(c,i) \arrow[dr, "\cod"'] \arrow[r,  "\circ"] & N(\coslice{\cC}{c})\arrow[d, "{\cod}" description] & \Delta^0 \arrow[l, "\id_c"'] \arrow[dl, "c"] \\ & N\cC
    \end{tikzcd}
  \]
  because the fiber of $\dom \colon N(\slice{\cC}{i}) \to N\cC$ over a point $c$ is just the set $\cC(c,i)$ of arrows from $c$ to $i$ in $\cC$.\footnote{This would be more complicated in the case of a generalized element.} By the coYoneda lemma, the domain coend reduces to $N(\coslice{\cC}{c})$.
\end{proof}

\subsection{Homotopical properties}\label{ssec:homotopical}
We now turn to the homotopical behavior of these functors. To start, recall:

\begin{lem}[{\cite[14.8.8]{hirschhorn}}]\label{lem:weights-projective-cofibrant}
The functors $\Ws$ and $\Wc$ are projective monomorphic.
\end{lem}
\begin{proof}
The two cases are dual to each other, so we only consider $\Ws \in \sSet^\cC$, forgetting the projection to $N\cC$, which is immaterial to the statement.
Since projective Reedy surjections in $\sSet^\cC$ agree with Reedy projective surjections in $(\Set^\cC)^{\DDelta^\op}$, projective Reedy inclusions in $\sSet^\cC$ agree with Reedy projective inclusions in $(\Set^\cC)^{\DDelta^\op}$ (as left classes of weak factorization systems whose right classes coincide).
We may thus equivalently show that the map $\emptyset \to \Ws \in (\Set^\cC)^{\DDelta^\op}$ is a Reedy projective inclusion, i.e., that the latching maps of $\Ws$ at each $[n] \in \DDelta$ are projective inclusions in $\Set^\cC$. The latching map at stage $[n] \in \DDelta$ has complement $A \in \Set^\cC$ sending $c \in \cC$ to the set of chains $f \colon [n] \to \slice{\cC}{c}$ whose step maps are non-identities.
It suffices to show that $\emptyset \to A$ is a projective inclusion.
This holds since $A$ is the left Kan extension along $\obj(\cC) \to \cC$ of the family that at $c \in \obj(\cC)$ is the subset of those chains $f$ with last vertex $(c, \id_c)$.
\end{proof}

\begin{lem}\label{lem:arr-to-fun-on-weights}
The functor
\[
\begin{tikzcd}[column sep=large]
  \arrtofunwgt\colon \sSet^\cC
  \ar[r, "{(N\cC)^*}"]
&
  (\slice{\sSet}{N\cC})^\cC
  \ar[r, "{\Wc \otimes_\cC -}"]
&
  \slice{\sSet}{N\cC}
\end{tikzcd}
\]
\begin{parts}
\item \label{lem:arr-to-fun-on-weights:cofibrations} sends projective monomorphisms to monomorphisms,
\item \label{lem:arr-to-fun-on-weights:left-anodynes} sends projective left anodynes to left anodynes.
\end{parts}
\end{lem}

\begin{proof}
For \ref{lem:arr-to-fun-on-weights:cofibrations}, the reindexing functor preserves projective monomorphisms, so we need only show that $\Wc \otimes_\cC -$ sends projective monomorphisms to monomorphisms. By \cref{lem:leibniz-weighted}.\ref{lem:leibniz-weighted:tensor-to-E} this follows from the fact that Leibniz products of monomorphisms are monomorphisms in $\slice{\sSet}{N\cC}$ and $\Wc \in (\slice{\sSet}{N\cC})^{\cC^\op}$ is injective (levelwise) monomorphic.

For \ref{lem:arr-to-fun-on-weights:left-anodynes}, recall that the forgetful functor $(N \cC)_!$ creates left anodynes, so it suffices to prove this result after postcomposing with the projection out of the slice.
After that postcomposition, the functor under consideration is isomorphic to  weighted colimit with $N(\coslice{\cC}{-}) \in \sSet^{\cC^\op}$, which is injective monomorphic.
The result now follows from \cref{lem:leibniz-weighted}.\ref{lem:leibniz-weighted:tensor-to-E} using \cref{lem:left-anodyne-quillen}.
\end{proof}

\begin{lem}\label{lem:fun-to-arr-on-weights}
The functor
\[
\begin{tikzcd}[column sep=large]
 \funtoarrwgt \colon \slice{\sSet}{N\cC}
  \ar[r, "{\Ws \otimes -}"]
&
  (\slice{\sSet}{N\cC})^\cC
  \ar[r, "{(N\cC)_!}"]
&
  \sSet^\cC
\end{tikzcd}
\]
\begin{parts}
\item \label{lem:fun-to-arr-on-weights:cofibrations} sends monomorphisms to projective monomorphisms,
\item \label{lem:fun-to-arr-on-weights:left-anodynes} sends left anodynes to projective left anodynes.
\end{parts}
\end{lem}

\begin{proof}
For \ref{lem:fun-to-arr-on-weights:cofibrations}, note that $\Ws$ is projective monomorphic by \cref{lem:weights-projective-cofibrant}.
Therefore, since Leibniz products of monomorphisms are monomorphisms in $\slice{\sSet}{N\cC}$, by Lemma \ref{lem:leibniz-weighted}.\ref{lem:leibniz-weighted:tensor-with-E}, weighted colimit with $\Ws$ sends monomorphisms to projective monomorphims.

For \ref{lem:fun-to-arr-on-weights:left-anodynes}, it suffices to demonstrate that the factor $\Ws \otimes -$ sends left anodynes to projective left anodynes.
By \cref{lem:left-anodyne-initial-vertex}.\ref{lem:left-anodyne-initial-vertex:segment} and \cref{lem:left-anodyne-cancel}, it suffices to consider its action on an initial segment $\iota_{\leq m} \colon \Delta^m \to \Delta^{m+1}$ living in the slice via some $\gamma \colon \Delta^{m+1} \to N \cC$.
Write $c = \gamma(m+1)$ for the last object of $\gamma \colon [m+1] \to \cC$.
Note that we have a unique lift
\[
\begin{tikzcd}
  \{m+1\}
  \ar[r, "{(c, \id)}"]
  \ar[d, hook]
&
  \slice{\cC}{c}
  \ar[d]
\\
  {[m+1]}
  \ar[r, "\gamma"]
  \ar[ur, dotted, "\gamma'" description]
&
  \cC \rlap{.}
\end{tikzcd}
\]
This induces the following section:
\[
\begin{tikzcd}
  {[m]}
  \ar[r,dotted] \ar[d, "\iota_{\leq m}"'] \arrow[dr, phantom, "\lrcorner" very near start]
&
  \bullet
  \ar[r] \ar[d, "\slice{\cC}{c} \times_\cC \iota_{\leq m}"'] \arrow[dr, phantom, "\lrcorner" very near start]
&
  {[m]}
  \ar[d, "\iota_{\leq m}"]
\\
  {[m+1]}
  \ar[r,dotted] \ar[dr, "\gamma'"']
&
  \bullet
  \ar[r] \ar[d] \arrow[dr, phantom, "\lrcorner" very near start]
&
  {[m+1]}
  \ar[d, "\gamma"]
\\&
  \slice{\cC}{c}
  \ar[r]
&
  \cC \rlap{.}
\end{tikzcd}
\]
Under Yoneda, this corresponds to a map to $\slice{\cC}{d} \times_{\cC} \iota_{\leq m}$ natural in $d \in \cC$:
\[
\begin{tikzcd}
  \cC(c, d) \times [m]
  \ar[r, dotted] \ar[d] \arrow[dr, phantom, "\lrcorner" very near start]
&
  \bullet
  \ar[r] \ar[d, "\slice{\cC}{d} \times_\cC \iota_{\leq m}"'] \arrow[dr, phantom, "\lrcorner" very near start]
&
  {[m]}
  \ar[d, "\iota_{\leq m}"]
\\
  \cC(c, d) \times [m+1]
  \ar[r, dotted]
  \ar[d, "{\cC(c, d) \times \gamma'}"']
&
  \bullet \arrow[dr, phantom, "\lrcorner" very near start]
  \ar[r]
  \ar[d]
&
  {[m+1]}
  \ar[d, "\gamma"]
\\
  \cC(c, d) \times \slice{\cC}{c}
  \ar[r]
&
  \slice{\cC}{d}
  \ar[r]
&
  \cC
\end{tikzcd}
\]
where the horizontal maps compose to projections.
The nerve sends the bulleted natural transformation to $\Ws \times_{N \cC} \iota_{\leq m}$; our goal is to show that it is projective left anodyne.
For this, it suffices to check that the nerve sends the left vertical natural transformation to a projective left anodyne map and the top left square to a pushout.

For the first claim, recall that the natural transformation $\emptyset \to \cC(c, -)$ is a projective injection.
The nerve sends $\iota_{\leq m}$ to a left anodyne map by \cref{lem:left-anodyne-initial-vertex}.
Hence, by \cref{lem:leibniz-weighted}.\ref{lem:leibniz-weighted:tensor-with-E}.\ref{lem:leibniz-weighted:tensor-with-E:projective}, their Leibniz tensor is projective left anodyne.

For the second claim, we check that the square in question at $d \in \cC$ forms a pushout under the nerve.
The square forms a square of discrete fibrations with target $[m+1]$ since the category $\cC(c, d)$ is discretely fibrant and the functors $\slice{\cC}{d} \to \cC$ and $\iota_{\leq m}$ are discrete fibrations, and discrete fibrations are closed under base change and composition.
By the dual of \cref{ex:opfibration-strict-left-map}, its nerve is a square of strict right maps.
Thus, by the dual of \cref{lem:left-fib-fiberwise-we} and the universality of pushouts in sets, the square in question is a pushout in degree $n$ already if it is a pushout in degree $0$.
We are thus reduced to checking that the original square in the above diagram is a pushout on objects.

Since the square is a pullback, it suffices to check for its actions on objects that base change along the complement of the right vertical inclusion makes the bottom function invertible.
The functor $\{m+1\} \hookrightarrow [m+1]$ complements $\iota_{\leq m}$ on objects, so it suffices to check that base change along it makes the bottom functor invertible.
This is the top left map in the below diagram:
\[
\begin{tikzcd}[/tikz/baseline=(\tikzcdmatrixname-\the\pgfmatrixcurrentrow-\the\pgfmatrixcurrentcolumn.base)]
  \cC(c, d)
  \ar[r, dotted, "\simeq"]
  \ar[d, "{\cC(c, d) \times (c, \id)}"']
&
  \bullet
  \ar[r]
  \ar[d]
  \ar[dr, phantom, "\lrcorner" very near start]
&
  1
  \ar[d, "c"]
\\
  \cC(c, d) \times \slice{\cC}{c}
  \ar[r]
&
  \slice{\cC}{d}
  \ar[r]
&
  \cC \rlap{.}
\end{tikzcd}
\qedhere
\]
\end{proof}

Analogously:

\begin{lem}\label{lem:arr-intermediate-on-weights}
The functor
\[
\begin{tikzcd}
  \slice{\sSet}{N\cC}
  \ar[r, "{\dom^*}"]
&
  \slice{\sSet}{N\cC^{[1]}}
  \ar[r, "\cod_!"]
&
  \slice{\sSet}{N\cC}
\end{tikzcd}
\]
\begin{parts}
\item \label{lem:arr-intermediate-on-weights:cofibrations} preserves monomorphisms,
\item \label{lem:arr-intermediate-on-weights:left-anodynes} preserves left anodynes.
\end{parts}
\end{lem}
\begin{proof}
\Cref{lem:arr-intermediate-on-weights:cofibrations} is obvious.
For \ref{lem:arr-intermediate-on-weights:left-anodynes}, note that $\dom \colon \cC^{[1]} \to \cC$ is a Grothendieck fibration, so $\dom \colon N \cC^{[1]} \to N \cC$ is a right transport fibration by \cref{grothendieck-fibration-is-weak-fibration}.
Therefore, $\dom^*$ preserves left anodynes by \cref{frobenius-left-anodyne}.
\end{proof}

\begin{lem}\label{lem:arr-to-fun-to-arr-point-equiv}
  The component of the natural transformation $\kappa \colon \funtoarrwgt \circ \arrtofunwgt \to \id$ of \cref{lem:arr-to-fun-to-arr-point} at a representable $\cC(c, -) \in \Set^\cC \subset \sSet^\cC$ is the retraction of a strong 0-oriented homotopy equivalence whose domain and codomain are projectively cofibrant objects.
\end{lem}
\begin{proof}
  The nerve functorially sends initial sections to strong 0-oriented homotopy equivalences (\cref{initial-section-is-she}).
  To show that $\kappa_{\cC(c,-)}$ is the retraction of a strong 0-oriented homotopy equivalence, it thus suffices to show that the composition functor
  \[
    \begin{tikzcd} \coslice{\cC}{c} \times_\cC \slice{\cC}{i} \arrow[r, "\circ"] & \cC(c,i) \end{tikzcd}
  \]
  has an initial section naturally in $i \in \cC$.
  The initial section sends $f \colon c \to i$ to the degenerate factorization $(\id, f)$.

  The codomain of $\kappa_{\cC(c,-)} \colon N(\coslice{\cC}{c} \times_\cC \slice{\cC}{-}) \to \cC(c,-)$, as a representable, is projective cofibrant. For the domain, we reprise the argument given in the proof of \cref{lem:weights-projective-cofibrant}. As there, it suffices to show that the latching map of $\emptyset \to N(\coslice{\cC}{c} \times_\cC \slice{\cC}{-})$ at each $[n] \in \DDelta$ is a projective inclusion in $\Set^\cC$. The latching map at $[n]$ has complement $A \in \Set^\cC$ sending $x \in \cC$ to the set of chains $f \colon [n] \to \coslice{\cC}{c} \times_\cC \slice{\cC}{x}$ whose step maps are all non-identities. It suffices to show that $\emptyset \to A$ is a projective inclusion. This holds since $A$ is the left Kan extension along $\obj(\cC) \to \cC$ of the family that at $x \in \obj(\cC)$ is the subset of those chains $f$ where the final object in $\coslice{\cC}{c} \times_\cC \slice{\cC}{x}$  has the identity for its $\slice{\cC}{x}$ component.
\end{proof}

\begin{lem}\label{lem:fun-to-arr-to-fun-point-equiv}
  The component of the zig-zag $\arrtofunwgt \circ \funtoarrwgt \overset{\mu}{\to} \cod_! \circ \dom^* \overset{\nu}{\leftarrow} \id$ of \cref{lem:fun-to-arr-to-fun-point} at a point $c \colon \Delta^0 \to N \cC$ has a left anodyne underlying map in $\sSet$.
\end{lem}
\begin{proof}
  The first component $\mu_c$ is invertible.
  The second component $\nu_c$ is the image under the nerve of the functor $\id_c \colon [0] \to \coslice{\cC}{c}$ selecting the initial object.
  This functor is an initial section, so its image under the nerve is a strong 0-oriented homotopy equivalence (\cref{initial-section-is-she}) and thus left anodyne by \cref{she-is-left-anodyne}.
\end{proof}

\subsection{Naturality}\label{ssec:naturality}

Now we address naturality in the category variable $\cC$. The mappings that send a category $\cC$ to the categories $\slice{\sSet}{N \cC}$, $\sSet^\cC$, and $(\slice{\sSet}{N \cC})^\cC$ extend to endopseudofunctors on $\Cat$.
The functorial actions on a morphism $F \colon \cC \to \cD$ are given by left Kan extension along $F$ and postcomposition along $N F \colon N \cC \to N \cD$.
By exchange of left Kan extensions, the mapping that sends $\cC$ to the functor $(N \cC)_! \colon (\slice{\sSet}{N \cC})^\cC \to \sSet^\cC$ extends to a pseudofunctor from $\Cat$ to the 2-category $\Funpseudo({[1]},\Cat)$ of functors, pseudonatural transformations, and modifications.
By taking mates, it follows that the mapping that sends $\cC$ to the functor $(N \cC)^* \colon \sSet^\cC \to (\slice{\sSet}{N \cC})^\cC$ extends to a pseudofunctor from $\Cat$ to the 2-category $\Funoplax({[1]},\Cat)$ of functors, oplax natural transformations, and modifications.

In this section, we will show that each mapping that sends a category $\cC$ to a functor considered in the previous section --- in particular, including the mappings $\cC \mapsto \arrtofunwgt_\cC$ and $\cC \mapsto\funtoarrwgt_\cC$ --- extends to a pseudofunctor from $\Cat$ to $\Funoplax({[1]},\Cat)$. In certain cases, noted explicitly below, these mapping lands the wide subcategory $\Funpseudo({[1]},\Cat)$, with pseudocommuting squares rather than oplax commuting squares.

\begin{lem}\label{lem:lan-preserves-strictly-left}
  For $F \colon \cC \to \cD$, the functor $\lan_F \colon (\slice{\sSet}{X})^{\cC} \to (\slice{\sSet}{X})^{\cD}$ preserves pointwise left objects.
\end{lem}
\begin{proof}
  For $Y(-) \in (\slice{\sSet}{X})^\cC$ to be a pointwise left object means that for each $m \geq 0$ and $c \in \cC$, $Y(c)_m$ is the base change of $Y(c)_0$ along $\ev_0 \colon X_m \to X_0$.
  The value at $d \in \cD$ of $\lan_F Y(-)_0$ is the colimit $\colim_{(c,f) \in \slice{F}{d}} Y(c)_0$.
  By local cartesian closure, this colimit is preserved by base change along $\ev_0 \colon X_m \to X_0$, pulling back to the colimit which calculates the value at $d \in \cD$ of $\lan_F Y(-)_m$.
  Thus $\lan_F Y(-) \in (\slice{\sSet}{X})^\cD$ is a left object.
\end{proof}

\begin{lem}\label{lem:colimit-Wslice-pseudonatural}
The mapping that sends a category $\cC$ to the functor $\Ws_\cC \otimes - \colon \slice{\sSet}{N\cC} \to (\slice{\sSet}{N\cC})^\cC$ extends to a pseudofunctor from $\Cat$ to the 2-category $\Funpseudo({[1]}, \Cat)$.
\end{lem}
\begin{proof}
Given a functor $F \colon \cC \to \cD$, we need to define a natural transformation as below-left:
\[
\begin{tikzcd}[column sep=large]
  \slice{\sSet}{N \cC}
  \ar[r, "(N F)_!"]
  \ar[d, "{\Ws_\cC \otimes -}"']
&
  \slice{\sSet}{N \cD}
  \ar[d, "{\Ws_\cD \otimes -}"]
  \ar[dl, phantom, "{\Rightarrow}"]
  \ar[dr, phantom, "="]
&
\slice{\sSet}{N \cC}
\arrow[r, equals]
\ar[d, "{\Ws_\cC \otimes -}"']
&
\slice{\sSet}{N \cC}
  \ar[dl, phantomcenter, "\mathrlap{\alpha^F \otimes -}"] \arrow[dl, phantomcenter, "\Rightarrow" pos=.6]
\ar[r, "{(N F)_!}"]
\ar[d, "{(NF)^* \Ws_\cD \otimes -}"]
&
\slice{\sSet}{N \cD}
\ar[d, "{\Ws_\cD \otimes -}"]
\\
  (\slice{\sSet}{N \cC})^\cC
  \ar[r, "(N F)_! \circ \lan_F"']
&
  (\slice{\sSet}{N \cD})^\cD
&
(\slice{\sSet}{N \cC})^\cC
\ar[r, "\lan_F"']
&
(\slice{\sSet}{N \cC})^\cD
\arrow[r, "(N F)_!"'] \arrow[ur, phantom, "\quad\cong"]
&
(\slice{\sSet}{N \cD})^\cD \rlap{.}
\end{tikzcd}
\]
We do so as shown above-right by constructing a natural transformation $\alpha^F \otimes -$ and pasting with a canonical natural isomorphism.

Left Kan extension along $F$ is given by the weighted colimit $\cD(F -, -) \otimes_\cC -$.
In the square to be filled by $\alpha^F \otimes -$ above, the left-then-bottom composite is thus weighted colimit with
\[
\cD(F -, -) \otimes_\cC N(\slice{\cC}{c}) \cong \lan_F N \slice{\cC}{-} \rlap{.}
\]
The top-then-right composite is weighted colimit with
\[
(N F)^* \Ws_\cD \cong N (\slice{F}{-}) \rlap{.}
\]
Thus, it remains to define a map in $(\slice{\sSet}{N\cC})^\cD$ between these weights:
\begin{equation}\label{colimit-Wslice-pseudonatural:map}
\begin{tikzcd}
  \lan_F N (\slice{\cC}{-})
  \ar[r, "\alpha^F"]
&
  N (\slice{F}{-}) \rlap{.}
\end{tikzcd}
\end{equation}
By adjoint transposition, this is equally given by a map in $(\slice{\sSet}{N \cC})^\cC$ from $N( \slice{\cC}{-})$ to $N(\slice{F}{F-})$.
We choose the image under the nerve of the natural map $F \colon \slice{\cC}{-} \to \slice{F}{F-}$ given by application of $F$.
This description is clearly pseudofunctorial in $F$.

We now argue that~\eqref{colimit-Wslice-pseudonatural:map} is invertible.
For arbitrary objects $c \in \cC$ and $d \in \cD$, the vertical maps in the below diagram are discrete opfibrations:
\[
\begin{tikzcd}
  \slice{\cC}{c}
  \ar[d]
&
  \slice{F}{d}
  \ar[r] \ar[d] \ar[dr, phantom, "\lrcorner" very near start]
&
  \slice{\cD}{d}
  \ar[d]
\\
  \cC \rlap{,}
&
  \cC
  \ar[r, "F"']
&
  \cD \rlap{.}
\end{tikzcd}
\]
Hence, they become left maps under the nerve (\cref{ex:opfibration-strict-left-map}).
By \cref{lem:lan-preserves-strictly-left}, the source and target of~\eqref{colimit-Wslice-pseudonatural:map} are then left objects over $N\cC$.
To check that~\eqref{colimit-Wslice-pseudonatural:map} is invertible, it hence suffices by \cref{lem:left-fib-fiberwise-we} to do so for its action in degree zero:
\[
\begin{tikzcd}[row sep=1.25em]
  \lan_F (N \slice{\cC}{-})_0 \ar[d, "\cong"']
  \ar[r, "\alpha^F"]
&
  N (\slice{F}{-})_0 \ar[d, "\cong"]  \\  \lan_F( \coprod\limits_{c \in \ob\cC}\cC(c,-)) \ar[r, "\alpha^F"'] & \coprod\limits_{c \in \ob\cC}\cD(Fc,-)
\end{tikzcd} \quad \in \Set^\cC \rlap{.}
\]
We conclude by observing that $(N \slice{\cC}{-})_0$ and $N (\slice{F}{-})_0$ are coproducts of representables and that the map $\alpha^F$ is the isomorphism witnessing that left Kan extension preserves coproducts and representables.
\end{proof}

\begin{cor}\label{cor:fun-to-arr-weights-pseudonatural}
The mapping that sends a category $\cC$ to the functor $\funtoarrwgt \colon \slice{\sSet}{N\cC} \to \sSet^\cC$ extends to a pseudofunctor from $\Cat$ to the 2-category $\Funpseudo({[1]},\Cat)$.
\end{cor}
\begin{proof}
Using the pseudonatural transformation defined in \cref{lem:colimit-Wslice-pseudonatural}, we have
  \[
\begin{tikzcd}
  \slice{\sSet}{N \cC}  \ar[r, "(N F)_!"]
\arrow[d, "\funtoarrwgt_\cC"']
  \ar[dr, phantomcenter, "{\cong}"]
  \ar[dr, draw=none, center, "\funtoarrwgt_F" below] & \slice{\sSet}{N\cD} \arrow[d, "\funtoarrwgt_\cD"]\\
\sSet^\cC \arrow[r, "\lan_F"'] & \sSet^\cD
\end{tikzcd}
\quad
\coloneqq
\quad
  \begin{tikzcd}[column sep=large]
  \slice{\sSet}{N \cC}
  \ar[r, "(N F)_!"]
  \ar[d, "{\Ws_\cC \otimes -}"']
&
  \slice{\sSet}{N \cD}
  \ar[d, "{\Ws_\cD \otimes -}"]
  \ar[dl, phantom, "{\cong}"]
\\
  (\slice{\sSet}{N \cC})^\cC
  \ar[r, "(N F)_! \circ \lan_F"']
  \arrow[d, "(N\cC)_!"']
&
  (\slice{\sSet}{N \cD})^\cD  \arrow[d, "(N\cD)_!"]\\
\sSet^\cC \arrow[r, "\lan_F"']
&
\sSet^\cD
\end{tikzcd}
\]
with the bottom square commuting up to natural isomorphism, as is most easily seen on the right adjoints $(N\cC)^* \circ \res_F \cong \res_F \circ (NF)^* \circ (N\cD)^*$.
\end{proof}

\begin{lem}\label{lem:colimit-Wcoslice-oplaxnatural}
The mapping that sends a category $\cC$ to the functor $\Wc_\cC \otimes_\cC - \colon \slice{\sSet}{N\cC} \to \sSet^\cC$ extends to a pseudofunctor from $\Cat$ to the 2-category $\Funoplax({[1]},\Cat)$.
\end{lem}
\begin{proof}
Given a functor $F \colon \cC \to \cD$, we need to define a natural transformation as below-left:
\[
\begin{tikzcd}[column sep=large]
  (\slice{\sSet}{N \cC})^\cC
  \ar[d, "{\Wc_\cC \otimes_\cC -}"']
  \ar[r, "(N F)_! \circ \lan_F"]
&
  (\slice{\sSet}{N \cD})^\cD
  \ar[d, "{\Wc_\cD \otimes_\cD -}"]
  \ar[dl, phantomcenter, "{\Rightarrow}"]
  \ar[dr, phantomcenter, "="]
&
  (\slice{\sSet}{N \cC})^\cC
  \ar[d, "{\Wc_\cC \otimes_\cC -}"']
  \ar[r, "\lan_F"]
&
  (\slice{\sSet}{N \cC})^\cD
  \ar[d, "{(N F)^* \Wc_\cD \otimes_\cC -}"]
  \ar[r, "(N F)_!"]
  \ar[dl, phantomcenter, "\mathrlap{\beta^F \otimes_{\cC} -}"] \arrow[dl, phantomcenter, "\Rightarrow" pos=.6]
&
  (\slice{\sSet}{N \cD})^\cD
  \ar[d, "{\Wc_\cD \otimes_\cD -}"]
\\
  \slice{\sSet}{N \cC}
  \ar[r, "(N F)_!"']
&
  \slice{\sSet}{N \cD}
&
  \slice{\sSet}{N \cC}
  \ar[r, equals]
&
  \slice{\sSet}{N \cC}
  \ar[r, "(N F)_!"'] \arrow[ur, phantom, "\quad\quad\cong"]
&
  \slice{\sSet}{N \cD} \rlap{.}
\end{tikzcd}
\]
We do so as shown above-right by constructing a natural transformation $\beta^F \otimes_{\cC} -$ and pasting with a canonical natural isomorphism.

As in the proof of \cref{lem:colimit-Wslice-pseudonatural}, left Kan extension along $F$ is given by the weighted colimit $\cD(F -, -) \otimes_\cC -$.
Therefore, the top-then-right composite of the square to be filled by $\beta^F \otimes_{\cC} -$ is weighted colimit with
\[
(N F)^* N(\coslice{\cD}{-}) \otimes_\cD \cD(F -, -)  \cong (N F)^* N(\coslice{\cD}{F -}) \cong N(\coslice{F}{F -}) \rlap{.}
\]
Thus, it remains to define a map in $(\slice{\sSet}{N\cC})^\cC$ between these weights:
\[
  \begin{tikzcd}
  N (\coslice{\cC}{-})
  \ar[r, "\beta^F"]
&
  N(\coslice{F}{F -}) \rlap{.}
\end{tikzcd}
\]
This arises under the nerve from the natural map $F \colon \coslice{\cC}{-} \to \coslice{F}{F -}$ given by application of $F$.
This description is clearly pseudofunctorial in $F$.
\end{proof}

\begin{cor}\label{cor:arr-to-fun-weights-oplaxnatural}
The mapping that sends a category $\cC$ to the functor $\arrtofunwgt \colon \sSet^\cC \to \slice{\sSet}{N\cC}$ extends to a pseudofunctor from $\Cat$ to the 2-category $\Funoplax({[1]},\Cat)$. Moreover:
\begin{parts}
\item\label{cor:arr-to-fun-weights-oplaxnatural:representable} For a functor $F \colon \cC \to \cD$, the natural transformation
\[
\begin{tikzcd}[column sep=large]
  \sSet^\cC
  \ar[d, "\arrtofunwgt_\cC"']
  \ar[r, "\lan_F"]
  \ar[dr, phantomcenter, "{\Rightarrow}"]
  \ar[dr, draw=none, center, "\arrtofunwgt_F" below]
&
  (\slice{\sSet}{N \cD})^\cD
  \ar[d, "\arrtofunwgt_\cD"]
\\
  \slice{\sSet}{N \cC}
  \ar[r, "(N F)_!"']
&
  \slice{\sSet}{N \cD}
\end{tikzcd}
\]
restricts on tensor with $C(c, -)$ to the natural transformation
\begin{equation}
  \label{eq:arr-to-fun-weighted-oplaxnatural-representables}
\begin{tikzcd}[column sep=large]
  \sSet
  \ar[r, bend left=20, "N(\coslice{\cC}{c}) \otimes -" above]
  \ar[r, phantomcenter, "\Downarrow"]
  \ar[r, bend right=20, "N(\coslice{\cD}{F c}) \otimes -" below]
&
\slice{\sSet}{N \cD}
\end{tikzcd}
\end{equation}
given by functoriality of the
tensor $- \otimes- \colon \slice{\sSet}{N\cD} \to \slice{\sSet}{\cD}$ with the nerve of the functor $F \colon \coslice{\cC}{c} \to \coslice{\cD}{F c}$ over $\cD$.
\item\label{cor:arr-to-fun-weights-oplaxnatural:point} The component of \eqref{eq:arr-to-fun-weighted-oplaxnatural-representables} on $\Delta^0 \in \sSet$ is in the closure under right cancellation of the class of maps whose underlying map is a strong 0-oriented homotopy equivalence.
\end{parts}
\end{cor}
\begin{proof}
  Using the oplax natural transformation defined in \cref{lem:colimit-Wcoslice-oplaxnatural}, we have:
\[
\begin{tikzcd}[column sep=large]
 \sSet^\cC \arrow[dd, "\arrtofunwgt_\cC"'] \arrow[r, "\lan_F"]
 \ar[ddr, phantomcenter, "{\Rightarrow}"]
 \ar[ddr, draw=none, center, "\arrtofunwgt_F" below]
 & \sSet^{\cD} \arrow[dd, "\arrtofunwgt_\cD"] &  & \sSet^\cC \arrow[r, "\lan_F"] \arrow[d, "(N\cC)^*"']
&
  \sSet^\cD \arrow[d, "(N\cD)^*"]
  \ar[dl, phantomcenter, "{\Rightarrow}"]
\\
 & & \coloneqq &   (\slice{\sSet}{N \cC})^\cC
  \ar[d, "{\Wc_\cC \otimes_\cC -}"']
  \ar[r, "(N F)_! \circ \lan_F"]
&
  (\slice{\sSet}{N \cD})^\cD
  \ar[d, "{\Wc_\cD \otimes_\cD -}"]
  \ar[dl, phantomcenter, "{\Rightarrow}"]
\\
\slice{\sSet}{N\cC} \arrow[r, "{(NF)_!}"'] & \slice{\sSet}{N\cD} & &
  \slice{\sSet}{N \cC}
  \ar[r, "(N F)_!"']
&
  \slice{\sSet}{N \cD} \rlap{.}
\end{tikzcd}
\]
The oplax natural transformation in the top square is the mate of the natural isomorphism
$(N\cC)^* \circ \res_F \cong \res_F \circ (NF)^* \circ (N\cD)^*$
discussed in the proof of \cref{cor:fun-to-arr-weights-pseudonatural}.

The observation \ref{cor:arr-to-fun-weights-oplaxnatural:representable} on restriction along tensor with $\cC(c,-)$ is a straightforward unfolding using the coYoneda lemma, where the lower and upper natural transformations are respectively defined by tensoring with the nerve of the following composable functors over $\cD$:
\[ \begin{tikzcd} \coslice{\cC}{c} \arrow[d] \arrow[r, "F"] & \coslice{F}{Fc} \arrow[r] \arrow[d]& \coslice{\cD}{Fc} \arrow[d] \\ \cC \arrow[r, equals] & \cC \arrow[r, "F"'] & \cD \rlap{.} \end{tikzcd}\]

For \ref{cor:arr-to-fun-weights-oplaxnatural:point}, the component at $\Delta^0$ is the nerve of a functor, which preserves the initial objects in each category: \[ \begin{tikzcd} & {[0]} \arrow[dl, "\id_c"'] \arrow[dr, "\id_{Fc}"] \\ \coslice{\cC}{c} \arrow[rr, "F"'] & & \coslice{\cD}{Fc} \rlap{.} \end{tikzcd}\]
By \cref{initial-section-is-she}, the inclusions of the initial objects are strong 0-oriented homotopy equivalences.
\end{proof}

\begin{lem}\label{lem:cod-dom-weights-oplaxnatural}
The mapping that sends a category $\cC$ to the functor $\cod_! \circ \dom^* \colon \slice{\sSet}{N\cC} \to \slice{\sSet}{N\cC}$ extends to a pseudofunctor from $\Cat$ to the 2-category $\Funoplax({[1]},\Cat)$.
\end{lem}
\begin{proof}
On account of the commutative square below-left we have the diagram of natural transformations below-right:
\[
\begin{tikzcd} N \cC^{[1]} \arrow[d,"{(\cod,\dom)}"'] \arrow[r, "NF^{[1]}"] & N \cD^{[1]} \arrow[d, "{(\cod,\dom)}"] \\ N\cC \times N\cC \arrow[r, "NF \times NF"'] & N\cD \times N\cD \rlap{,} \end{tikzcd}  \qquad \qquad
\begin{tikzcd}
  \slice{\sSet}{N\cC}
  \arrow[r, "{(NF)_!}"]
  \arrow[d, "\dom^*"'] \arrow[dr, phantom, "\Rightarrow"]&
  \slice{\sSet}{N\cD}
  \arrow[d, "\dom^*"] \\
  \slice{\sSet}{N\cC^{[1]}}
  \arrow[r, "{(NF^{[1]})_!}"']
  \arrow[d, "\cod_!"'] &
  \slice{\sSet}{N\cD^{[1]}}
  \arrow[d, "\cod_!"]
  \\
  \slice{\sSet}{N\cC}
  \arrow[r, "{(NF)_!}"'] & \slice{\sSet}{N\cD} \rlap{,}
\end{tikzcd}
\]
in which the bottom cell commutes and the natural transformation in the top cell is the mate of a similar identity natural transformation.
\end{proof}

For later use, we highlight some special cases of the natural transformations just constructed.

\begin{defn}\label{defn:unit-iso}
  In the case $\cD={[0]}$, we have equivalences $\slice{\sSet}{N{[0]}} \simeq \sSet \simeq \sSet^{{[0]}}$ and modulo these, $\arrtofunwgt$ and $\funtoarrwgt$ are the identity functors. Thus, the pseudonatural coherence cell of \cref{cor:fun-to-arr-weights-pseudonatural} defines a natural isomorphism $\colim_\cC \circ \funtoarrwgt_\cC \cong (N\cC)_!$:
  \[
  \begin{tikzcd}
    \slice{\sSet}{N\cC}
    \arrow[r, "{(N\cC)_!}"]
    \ar[d, "{\funtoarrwgt_\cC}"']
  &
    \sSet
    \ar[d, equals]
    \ar[dl, phantom, "{\cong}"]
   \\
   \sSet^{\cC}
   \ar[r, "\colim_\cC"']
&
    \sSet \rlap{.}
    \end{tikzcd}
  \]
\end{defn}

\begin{defn}\label{defn:unit-map}
  Similarly, in the case $\cD={[0]}$, the oplax coherence cell of \cref{cor:arr-to-fun-weights-oplaxnatural} defines a canonical natural transformation $\eta \colon (N\cC)_! \circ \arrtofunwgt_\cC \to \colim_\cC$:
  \[ \begin{tikzcd}
    \sSet^{\cC}
    \ar[d, "{\arrtofunwgt_\cC}"']
    \ar[r, "\colim_\cC"]
  &
    \sSet
    \ar[d, equals]
    \ar[dl, phantom,
    "{\Rightarrow\eta}"]
   \\
    \slice{\sSet}{N\cC} \arrow[r, "{(N\cC)_!}"'] &
    \sSet \rlap{.}
    \end{tikzcd} \]
By pseudofunctoriality, these natural transformations commute with the natural transformations $\arrtofunwgt_F$ defined for any functor $F \colon \cC \to \cD$.
\end{defn}

\begin{ex}\label{ex:constant-unit-map}
  We specialize \cref{defn:unit-map} to the case $\cC=[n]$ and calculate its components at objects of the form $K \times [n](k,-) \in \sSet^{[n]}$ where $K \in \sSet$ and $k \in [n]$.
  The functor $\arrtofunwgt_{[n]}$ pulls back over $\Delta^n$ then computes the colimit in $\slice{\sSet}{\Delta^n}$ weighted by $\Wc$.
  On $K \times [n](k,-)$, the weighted colimit reduces because of the representable to the tensor in $\slice{\sSet}{\Delta^n}$ of the value of $\Wc$ at $k$, which is $\iota_{\geq k} \colon \Delta^{n-k} \to \Delta^n$, with $\pi \colon \Delta^{n} \times K \to \Delta^n$.
  The result is $\iota_{\geq k}\pi \colon \Delta^{n-k} \times K \to \Delta^n$.
  The colimit of the diagram $K \times [n](k,-)$ is just $K$, and the component of $\eta$ at this object is just the projection $\eta_{K \times [n](k,-)} \colon \Delta^{n-k} \times K \to K$.
\end{ex}

\cref{cor:fun-to-arr-weights-pseudonatural,cor:arr-to-fun-weights-oplaxnatural,lem:cod-dom-weights-oplaxnatural} imply that the domains and the codomains of the natural transformations $\kappa$, $\mu$, and $\eta$ of \cref{lem:arr-to-fun-to-arr-point,lem:fun-to-arr-to-fun-point} are the actions on objects of pseudofunctors valued in $\Funoplax({[1]},\Cat)$. We now remark that the oplax coherence cells just constructed, which we denote generically by ``$\phi$'', commute with these natural transformations in a sense that might best be visualized as a three dimensional commutative diagram of 2-cells defined for each functor $F \colon \cC \to \cD$.

\begin{rmk}\label{rmk:constant-diagram-map-coherence}
  The natural transformations of \cref{lem:arr-to-fun-to-arr-point} commute with the oplax natural transformations $\phi$ associated to a functor $F \colon \cC \to \cD$, defining a commutative diagram of natural transformations from $\sSet^\cC$ to $\sSet^\cD$:
  \[ \begin{tikzcd}
    \lan_F\circ  \funtoarrwgt_\cC \circ \arrtofunwgt_\cC
    \arrow[r, "\lan_F\circ\kappa_\cC"] \arrow[d, "\phi"']
    & \lan_F \arrow[d, equals]  \\
     \funtoarrwgt_\cD \circ \arrtofunwgt_\cD \circ\lan_F
    \arrow[r, "\kappa_\cD\circ \lan_F"'] & \lan_F \rlap{.} \end{tikzcd} \]

In the case $\cD = {[0]}$, modulo the equivalence $\sSet^{{[0]}} \simeq \sSet$, $\lan_F$ is the functor $\colim_{\cC} \colon \sSet^\cC \to \sSet$, $\arrtofunwgt_{{[0]}}$ and $\funtoarrwgt_{{[0]}}$ are identity functors, and $\kappa_{{[0]}}$ is the identity. In that special case, this result identifies the two natural transformations $\colim_\cC (\kappa_\cC), \phi \colon \colim_\cC \circ \funtoarrwgt_\cC \circ \arrtofunwgt_\cC \to \colim_\cC$.
\end{rmk}

\begin{rmk}
The natural transformations of \cref{lem:fun-to-arr-to-fun-point} commute with the oplax natural transformations $\phi$ associated to a functor $F \colon \cC \to \cD$, defining a commutative diagram of natural transformations from $\slice{\sSet}{N\cC}$ to $\slice{\sSet}{N\cD}$:
  \[ \begin{tikzcd}
    (NF)_!\circ \arrtofunwgt_\cC \circ \funtoarrwgt_\cC
    \arrow[r, "(NF)_!\circ \mu_\cC"] \arrow[d, "\phi"']
    & (NF)_! \circ \cod_!\circ \dom^* \arrow[d, "\phi"] &
    \arrow[d, equals] \arrow[l, "(NF)_!\circ \nu_\cC"'] (NF)_!  \\
    \arrtofunwgt_\cD \circ \funtoarrwgt_\cD \circ (NF)_!
    \arrow[r, "\mu_{\cD} \circ (NF)_!"'] & \cod_! \circ \dom^* \circ (NF)_! & \arrow[l, "\nu_\cD \circ (NF)_!"] (NF)_! \rlap{.} \end{tikzcd} \]

In the case $\cD = {[0]}$, modulo the equivalence $\slice{\sSet}{N{[0]}}\simeq \sSet$, the bottom span is equivalent to the span of identity natural transformations at the functor $(N\cC)_! \colon \slice{\sSet}{N\cC} \to \sSet$, yielding a commutative diagram:
\begin{equation}\label{eq:how-we-pullback-the-span} \begin{tikzcd}
  (N\cC)_!\circ \arrtofunwgt_\cC \circ \funtoarrwgt_\cC
  \arrow[r, "(NC)_!\mu"] \arrow[dr, "\eta"']
  & (N\cC)_! \circ \cod_!\circ \dom^* \arrow[d, "\phi"] &
  \arrow[dl, equals] \arrow[l, "(NC)_!\nu"'] (N\cC)_!  \\ & (N\cC)_! \rlap{.} &  \end{tikzcd} \end{equation}
\end{rmk}

\section{Directed univalence in a model topos}\label{sec:dua-model}

In this section, we apply the results of \S\ref{sec:weights} to prove a directed univalence result, \cref{thm:dua-model}, at the level of a model topos of simplicial objects $\smE$.
For a small 1-category $\cC$ and an object $\Gamma \in \smE$, we exhibit a pair of right Quillen functors
\begin{equation}\label{eq:model-intro-equivalence}
  \begin{tikzcd} \slice{\smE}{N\cC \times \Gamma} \arrow[r, bend left, "\arrtofun^{\Gamma}_\cC" above] & (\slice{\smE}{\Gamma})^{\cC} \arrow[l, bend left, "\funtoarr^{\Gamma}_\cC" below] \end{tikzcd}
\end{equation}
between the (projective) covariant model structures that are inverse equivalences as homotopical functors.
In fact, we prove a more precise result: the homotopical equivalence is given by explicit natural transformations at the pointwise level whose components at fibrant objects are weak equivalences.
We derive the sliced functors above from unsliced versions, $\arrtofun_\cC \colon \slice{\smE}{N\cC} \to \smE^\cC$ and $\funtoarr_\cC \colon \smE^\cC \to \slice{\smE}{N\cC}$.
The sliced forms will be used to derive our $\infty$-topos level incarnation of directed univalence in \cref{sec:dua-types}.

To see how we obtain \eqref{eq:model-intro-equivalence} from the results of \S\ref{sec:weights}, consider the 2-category $\Cat_\complete$ obtained by restricting $\Cat$ to complete categories and continuous functors. The 2-exponential of $\mE$ with $\mC$ is given by the category $[\mC, \mE]_\cont$ of continuous functors from $\mC$ to $\mE$, or in the case of $[\mC^\op,\mE]_\cont$ of functors that send colimits in $\mC$ to limits in $\mE$. Given classes of maps $\mathcal{L}$ in $\mC$ and $\mathcal{R}$ in $\mE$, we have an induced class of maps $[\mathcal{L}, \mathcal{R}]_\cont$ in $[\mC^\op, \mE]_\cont$ containing those natural transformations $u \colon F \to G$ whose pullback application sends $\mathcal{L}$ to $\mathcal{R}$.
If $\mathcal{L}$ is the saturation of a set of maps $J$, then, since $F$ and $G$ are continuous, $u$ belongs to this class if its pullback application sends $J$ to $\mathcal{R}$.
As the categories $\sE$, $\smE^\cC$, and $\slice{\sE}{K}$ --- the latter under the hypothesis of infinitary extensivity --- can be re-expressed as categories of continuous functors from $\sSet$, $\sSet^\cC$, and $\slice{\sSet}{K}$, respectively, to $\mE$, we can use this framework to systematically transpose, in a 2-functorial manner, operations in the $\Set$-world interacting with left classes to operations in the $\mE$-world interacting with right classes.
The induced operations on the right are then just given by 2-functoriality of precomposition in $\Cat_\complete$.

In \cref{ssec:enrichment} we describe compatible enrichments that allow us to translate the results of \S\ref{sec:weights} into the settings in which we will ultimately work: the category $\smE$ of simplicial objects in a fibration-extensive type-theoretic model topos, its functor categories $\smE^{\cC}$, and its slices $\slice{\smE}{K}$ over a discretely embedded simplicial set. The functors $\arrtofun$ and $\funtoarr$ and their slice variants over $\Gamma$ are introduced in \S\ref{ssec:global-fun-arr}. \Cref{thm:dua-model} is proven in \S\ref{ssec:dua-model}.

\subsection{From weights to diagrams}\label{ssec:enrichment}

Let $\mE \in \Cat_\complete$ be a complete locally small 1-category. For any small category $\cC$, the weighted limit functor transposes to define an equivalence
\[ \begin{tikzcd} (\Set^{\DDelta^\op \times \cC})^\op \times \mE^{\DDelta^\op \times \cC} \arrow[r, "{\{-,-\}^{\DDelta^\op \times \cC}}"] & [+1em]\mE & \leftrightsquigarrow & \smE^\cC \arrow[r, "\simeq"] & {[(\sSet^\cC)^\op, \mE]_\cont} \rlap{.} \end{tikzcd}\]
Explicitly, the equivalence can be understood as cocontinuous extension along the Yoneda embedding:
\[ \smE^\cC \simeq [\DDelta\times\cC^\op, \mE^\op] \simeq [\Set^{\DDelta^\op \times \cC}, \mE^\op]_\cocont \simeq {[(\sSet^\cC)^\op, \mE]_\cont}.\]

This equivalence will allow us to transpose left adjoints between presheaf categories to right adjoints of $\mE$-valued presheaves. We introduce the following terminology for this purpose.

\begin{defn}\label{defn:EHom}
For a small category $\cC$, the \textbf{$\mE$-relative hom functor}
\[
\begin{tikzcd}[column sep=4em]
  (\sSet^\cC)^\op \times \smE^\cC
  \ar[r, "{\EHom}"]
&
  \mE
\end{tikzcd}
\]
is the weighted limit functor $\{-, -\}^{\DDelta^\op \times \cC}$.
\end{defn}

As noted above, a transposed version of the $\mE$-relative hom functor defines an equivalence $\EHom \colon \smE^\cC \simeq [(\sSet^\cC)^\op, \mE]_\cont$ involving a 2-exponential in the 2-category $\Cat_\complete$. Thus, we may act on $\mE$-valued diagram categories by restricting the first variable of $[(-)^\op, \mE]_\cont$ along left adjoint functors between simplicial presheaf categories.

Under additional hypotheses on $\mE$, these constructions can be extended to slices over discrete objects. Suppose now that $\mE$ is a bicomplete locally small 1-category. In particular, $\mE$ is enriched, tensored, and cotensored over sets:
\[ \begin{tikzcd} \Set \times \mE \arrow[r, "\cdot"] & \mE \rlap{,} & \Set^\op \times \mE \arrow[r, "{\{-,-}\}"] & \mE \rlap{,} & \mE^\op \times \mE \arrow[r, "{\mE(-,-)}"] & \Set \rlap{.} \end{tikzcd}\]
Using the terminal object $\ast \in \mE$, we obtain the adjunction below, which defines the \textbf{discrete embedding} of sets into $\mE$:
\[
\begin{tikzcd} \mE \arrow[r, bend right, "{\Gamma \coloneqq \mE(\ast,-)}"' pos=.6] \arrow[r, phantom, "\bot"] & \Set \arrow[l, bend right, "{\Delta\coloneqq - \cdot \ast}"' pos=.4] \rlap{.} \end{tikzcd}
\]
As in \S\ref{ssec:stt}, this induces a corresponding pointwise-defined discrete embedding of simplicial sets into simplicial objects in $\mE$. Adopting the conventions introduced elsewhere, we leave the discrete embedding implicit when the meaning is clear.

Recall a category is \textbf{infinitarily extensive} if for any coproduct $\coprod_{i \in I} A_i \in \mE$, pullback defines an equivalence of categories
\[ \slice{\mE}{\coprod_{i \in I}A_i} \simeq \prod\nolimits_{i \in I}\slice{\mE}{A_i}.\] As discrete objects $I \in \mE$ are coproducts $I \cong \coprod_{i \in I}\ast$, for any set $I$, we have an equivalence $\slice{\mE}{I} \simeq \mE^I$ between the slice over the discrete embedding of $I$ and the category of functors indexed by this discrete set.
Henceforth we assume that $\mE$ is also infinitarily extensive.

\begin{lem}\label{lem:slice-elements-equivalence}
Given a category $\cD$ and a set-valued presheaf $K \in \Set^{\cD^\op}$, we have an equivalence of categories \[\slice{\mE^{\cD^\op}}{K} \simeq \mE^{\el{K}^\op}\]  between the slice category and the category of contravariant functors indexed by the category of elements of $K$.
\end{lem}

In particular, for any simplicial set $K$, we have an equivalence of categories $\slice{\smE}{K} \simeq \mE^{\el{K}^\op}$.

\begin{proof}
We have a chain of natural equivalences. The first step uses the characterization of a functor category as the oplax limit of a constant diagram:
\begin{align*}
\slice{\mE^{\cD^\op}}{K}
&\simeq
\slice{\left(\textstyle{\lim}^\oplax_{\cD^\op} \mE\right)}{K} \rlap{.} \\
\intertext{The second step uses the characterization of a slice category as the oplax limit of the arrow picking out the object:}
&\simeq
\textstyle{\lim}^\oplax_{{[1]}}\left({[0]} \xrightarrow{K}\textstyle{\lim}^\oplax_{\cD^\op} \mE\right) \rlap{.} \\
\intertext{The third step exchanges oplax limits and repackages the interior oplax limits into slice categories:
}
&\simeq
\textstyle{\lim}^\oplax_{d \in \cD^\op} \textstyle{\lim}^\oplax_{[1]}({[0]} \xrightarrow{K(d)}\mE) \simeq
\textstyle{\lim}^\oplax_{d \in \cD^\op} \slice{\mE}{K(d)} \rlap{.} \\
\intertext{The fourth step uses infinitary extensivity:}
&\simeq
\textstyle{\lim}^\oplax_{d \in \cD^\op} \mE^{K_d} \rlap{.} \\
\intertext{The final step composes oplax limits and reassembles the result into the functor category:}
&\simeq \textstyle{\lim}^\oplax_{d \in \cD^\op} \textstyle{\lim}^\oplax_{k \in K_d}\mE\simeq
\textstyle{\lim}^\oplax_{(d,k) \in \el{K}^\op} \mE
\simeq
\mE^{\el{K}^\op} \rlap{.} \qedhere
\end{align*}
\end{proof}

Thus, for a simplicial set $K$, we have an equivalence
\[ \slice{\smE}{K} \simeq \mE^{\el{K}^\op} \simeq [\Set^{\el{K}^\op}, \mE]_\cont \simeq [(\slice{\sSet}{K})^\op, \mE]_\cont,\]
and thus the $\mE$-relative hom functor introduced in \cref{defn:EHom} also exists in this setting. For calculational purposes, it is more convenient to work with an alternate presentation of the sliced $\EHom$ for which we introduce special notation:

\begin{defn}\label{defn:Ehom-K}
Given a simplicial set $K$, the \textbf{$\mE$-relative hom functor} \[ \begin{tikzcd} (\slice{\sSet}{K})^\op \times \slice{\smE}{K} \arrow[r, "\EHom_{K}"] & \mE
\end{tikzcd} \]
is defined at $a \colon A \to K$ and $x \colon X \to K$ by the pullback
\[ \begin{tikzcd}
  \EHom_K(A,X) \arrow[r]\arrow[d]\arrow[dr, phantom, "\lrcorner" very near start] & \EHom(A,X) \arrow[d, "{\EHom(A,x)}"] \\ \ast \arrow[r, "a"'] & \EHom(A,K)
\end{tikzcd}
\]
along the image of the element $a \colon A \to K$ under the isomorphism \[\sSet(A,K) \cong \sSet(A, \mE(\ast,K)) \cong \mE(\ast, \EHom(A,K)).\]
\end{defn}

The bifunctor $\EHom_K$ has a left adjoint
\[ \begin{tikzcd} \slice{\sSet}{K} \times \mE \arrow[r, "\otimes_K"] & \slice{\smE}{K} \end{tikzcd}\]
as well as a right adjoint that we will not need.
Given $a \colon A \to K$ and $E \in \mE$, $A \otimes_K E$ is defined by the map $A \times E \to A \to K$, composing $a$ with the projection, where $K$ is embedded discretely and $E$ is embedded constantly into $\smE$.

\begin{lem}\label{lem:slice-EHom-weighted-limit}
Up to equivalence $\EHom_K$ coincides with the weighted limit functor
\[
\begin{tikzcd} (\Set^{(\el{K})^\op})^\op \times \mE^{(\el{K})^\op} \arrow[dr, "{\EHom}"] \arrow[d, "\rotatebox{90}{$\sim$}"'] \\ (\slice{\sSet}{K})^\op \times \slice{\smE}{K} \arrow[r, "\EHom_K"'] & \mE \rlap{.} \end{tikzcd} \]
\end{lem}
\begin{proof} We identify these functors by identifying their left adjoints:
 \[
 \begin{tikzcd}
  \slice{\sSet}{K} \times \mE \arrow[r, "{\otimes_K}"]\arrow[d, "\rotatebox{90}{$\sim$}"'] & \slice{\smE}{K} \arrow[d, "\rotatebox{90}{$\sim$}"] \\
  \Set^{(\el{K})^\op} \times \mE \arrow[r, "{\otimes_{\el{K}^\op}}"'] & \mE^{(\el{K})^\op} \rlap{.}
  \end{tikzcd}
  \]
The left bottom composite takes $a \colon A \to K$ and $E$ to the presheaf whose value at $\sigma \colon \Delta^n \to K$ is the fiber $A_\sigma$ of $a_n \colon A_n \to K_n$ over $\sigma \in K_n$ and then forms the tensors of these sets with $E$. The top right composite carries the map $A \times E \to A \to K$ in $\slice{\smE}{K}$ to the $\mE$-valued presheaf whose value at $\sigma \colon \Delta^n \to K$ is the fiber of the map $A_n \times E \to A_n \to K_n$ over $\sigma \colon \ast \to K_n$, which is again $A_n \times E$.
\end{proof}

We have established equivalences
\[ \EHom \colon  \smE^\cC \simeq [(\sSet^\cC)^\op, \mE]_\cont \qquad \text{and} \qquad \EHom_K \colon \slice{\smE}{K} \simeq [(\slice{\sSet}{K})^\op, \mE]_\cont, \]
which \cref{lem:slice-EHom-weighted-limit} allows us to treat uniformly.
For a small category $\cC$, let $\Psh{\cC} \coloneqq \Set^{\cC^\op}$ and write $\Prof$ for the 2-category whose objects are small categories, morphisms from $\cC$ to $\cD$ are left adjoints $L \colon \Psh{\cC} \to \Psh{\cD}$, and 2-morphisms are natural transformations between left adjoints, i.e., for the full subcategory of $\Cat_{\cocomplete}$ spanned by the presheaf categories.

\begin{rmk}
  Left adjoints $L \colon \Psh{\cC} \to \Psh{\cD}$ between presheaf categories correspond to bimodules $W_L \in \Set^{\cC \times \cD^\op}$: from $L$, we obtain $W_L$ by restricting $L$ along the Yoneda embedding, and from $W_L$, we obtain $L$ and its right adjoint as $W_L$-weighted colimit and $W_L$-weighted limit:
\[ \begin{tikzcd} \Psh{\cC} \arrow[r, bend left, "W_L \otimes_\cC-"] \arrow[r, phantom, "\bot"] & \Psh{\cD} \arrow[l, bend left, "{\{W_L,-\}^\cD}"] \rlap{.} \end{tikzcd}\]
Thus $\Prof$ is biequivalent to the bicategory of small categories, bimodules (``weights''), and bimodule homomorphisms (natural transformations of weights).
In the examples developed in \S\ref{sec:weights}, the left adjoint perspective is emphasized: for example, $\arrtofunwgt$ and $\funtoarrwgt$ are described as left adjoints between the categories $\sSet^\cC \simeq \Psh{\DDelta \times \cC^\op}$ and $\slice{\sSet}{N\cC} \simeq \Psh{\el{N\cC}}$.
\end{rmk}

Formation of $\mE$-valued presheaf categories extends to a pseudofunctor
\[
\begin{tikzcd}[column sep=4em]
  \Prof^{\co \op}
  \ar[r, "{\Psh[\mE]{-}}"]
&
  \Cat_\complete
\end{tikzcd}
\]
defined by restricting the first variable of $[-,\mE]_\cont$ along $(-)^\op \colon \Prof^\co \to \Cat_\complete$.
We introduce special terminology for the functorial action of ${\Psh[\mE]{-}}$ on the 1- and 2-cells of $\Prof$.

\begin{defn}
For a left adjoint $L \colon \Psh{\cC} \to \Psh{\cD}$, we refer to the corresponding map
\[
\begin{tikzcd} \Psh[\mE]{L} \colon { \Psh[\mE]{\cD} \coloneqq \mE^{\cD^\op} \simeq [\Psh{\cD}^\op, \mE]_\cont} \arrow[r, "-\circ L"] & {[\Psh{\cC}^\op, \mE]_\cont \simeq \mE^{\cC^\op} \eqqcolon \Psh[\mE]{\cC}} \end{tikzcd}\] as the \textbf{$\mE$-relative right adjoint} of $L$. Explicitly, the $\mE$-relative right adjoint of $L$ is defined by weighted limit with $W_L \coloneqq L \circ \yo \in \Set^{\cC\times\cD^\op}$. Given a natural transformation $u \colon L \to L'$ between left adjoints, the corresponding natural transformation
\[
\begin{tikzcd}
  \Psh{\cC} \arrow[r, bend left, "L"] \arrow[r, bend right, "L'"'] \arrow[r, phantom, "\Downarrow u"] & \Psh{\cD} & \mapsto & {\Psh[\mE]{\cC}} & {\Psh[\mE]{\cD}} \arrow[l, bend right, "{\Psh[\mE]{L}}"'] \arrow[l, bend left, "{\Psh[\mE]{L'}}"] \arrow[l, phantom, "{\Uparrow u^\dagger}"]
\end{tikzcd}
\]
is called the \textbf{$\mE$-relative conjugate} of $u$. When more concise notation is useful, we write $u^\dagger \coloneqq \Psh[\mE]{u}$ for the $\mE$-relative conjugate of $u$, as in the displayed diagram above-right.
\end{defn}

For later use, we record a few elementary observations relating $\EHom_K$ to $\EHom \colon \sSet^\op \times \smE \to \mE$.

\begin{lem}\label{lem:ehom-slice-pullback-ehom}
  Let $K \in \sSet$.
  For any morphisms $j \colon A \to B$ in $\slice{\sSet}{K}$ and $f \colon Y \to X$ in $\slice{\smE}{K}$, ${\widehat{\EHom_K}(j,f)}$ is the pullback of $\widehat{\EHom}(j,f)$ along the canonical projection
  \[
    \EHom_K(A,Y) \times_{\EHom_K(A,X)} \EHom_K(B,X) \to \EHom(A,Y) \times_{\EHom(A,X)} \EHom(B,X) \rlap{.}
  \]
\end{lem}
\begin{proof}
  Say that $A$ and $B$ lie over $K$ by maps $a \colon A \to K$ and $b \colon B \to K$ respectively.
  Consider the commutative prism
  \[
    \begin{tikzcd}[column sep={5.5em,between origins}, row sep=1.5em]
      \EHom_K(B,Y) \ar[dd] \ar[dr] \ar[rr] && \EHom(B,Y) \ar[dd] \ar[dr] & \\[-1em]
      & \EHom_K(A,Y) \ar[rr,crossing over] && \EHom(A,Y) \ar[dd] \\
      \EHom_K(B,X) \ar[dd] \ar[dr] \ar[rr] && \EHom(B,X) \ar[dd] \ar[dr] \\[-1em]
      & \EHom_K(A,X) \ar[uu,<-,crossing over] \ar[rr,crossing over] && \EHom(A,X) \ar[dd] \\
      \Delta^0 \ar[dr,equals] \ar[rr,"b"] && \EHom(B,K) \ar[dr] \\[-1em]
      & \Delta^0 \ar[uu,<-,crossing over] \ar[rr,"a"] && \EHom(A,K)
    \end{tikzcd}
  \]
  whose upper left face has $\widehat{\EHom_K}(j,f)$ as its gap map and whose upper right face has $\widehat{\EHom}(j,f)$ as its gap map.
  The lower back and composite back faces are pullbacks, so the upper back face is a pullback by pasting.
  The upper front face is a pullback for the same reason.
  This implies by iterated pasting that the gap map $\widehat{\EHom}_K(j,f)$ of the upper left face is a pullback of the gap map $\widehat{\EHom}(j,f)$ of the upper right face along the induced map between their codomains.
\end{proof}

\begin{lem}\label{lem:ehom-slice-vs-ehom}
  Let $K \in \sSet$.
  \begin{parts}
  \item\label{lem:ehom-slice-vs-ehom:absolute}
    For any $A \in \sSet$, $X \in \smE$, and $x \colon X \to K$, we have
\[ \EHom(A,X) \cong \coprod_{a \colon A \to K} \EHom_K ((A,a),(X,x)).\]
\item\label{lem:ehom-slice-vs-ehom:relative} The relative version of the same decomposition holds: for any $j \colon A \to B \in \sSet$, $f \colon Y \to X$ in $\smE$, and $x \colon X \to K$, we have
  \[ \widehat{\EHom}(j,f) \simeq \coprod_{b \colon B \to K} \widehat{\EHom_K}\left(\begin{tikzcd}[cramped,row sep=small,column sep=tiny] A \arrow[dr] \arrow[rr,"j"] && B \arrow[dl,"b"{pos=.4}] \\ & K\end{tikzcd}
      \:,\:
      \begin{tikzcd}[cramped,row sep=small,column sep=tiny] Y \arrow[dr] \arrow[rr,"f"] && X \arrow[dl,"x"{pos=.4}] \\ & K\end{tikzcd} \right) . \]
\end{parts}
\end{lem}
\begin{proof}
  For each $a \colon A \to K$, we have
  \[ \begin{tikzcd} \EHom_K((A,a),(X,x)) \arrow[r] \arrow[d] \arrow[dr, phantom, "\lrcorner" very near start] & \EHom(A,X) \arrow[d] \\ \Delta^0 \arrow[r, "a"'] & \EHom(A,K) \rlap{.}
  \end{tikzcd}\]
  As the discrete embedding is modeled at the pointset level by a functor with both left and right adjoints, $\EHom(A,K)$ is both the $A$-weighted limit of the discrete object $K$ and the discrete embedding of the $A$-weighted limit of the simplicial set $K$. This latter object is the discrete set of discrete maps from $J$ to $K$. So by extensivity, $\EHom(A,X) \to \EHom(A,K)$ decomposes in $\mE$ as a coproduct of the fibers $\EHom_K((A,a),(X,x))$ indexed by the set of maps $a \colon A \to K$ in $\sSet$.

  Now consider $j \colon A \to B$ in $\sSet$ and $f \colon Y \to X$ in $\smE$ with $x \colon X \to K$.
  Note first that the map \[ \widehat{\EHom}(j,f) \colon \EHom(B,Y) \to \EHom(B,X) \times_{\EHom(A,X)} \EHom(A,Y) \] lives over $\EHom(B,X)$, which lives over $\EHom(B,K)$.
  By \cref{lem:ehom-slice-pullback-ehom}, pulling back $\widehat{\EHom}(j,f)$ along a point $b \colon \Delta_0 \to \EHom(B,K)$ yields $\widehat{\EHom_K}(j \colon (A,bj) \to (B,b), f \colon (Y,xf) \to (X,x))$.
  Since $\EHom(B,K)$ is a discrete set, it follows by extensivity that $\widehat{\EHom}(j,f)$ is the coproduct of these fibers.
\end{proof}

\subsection{Reedy fibrations and left fibrations}\label{ssec:left-fibration}

We now suppose for the remainder of \cref{sec:dua-model} that $\mE$ is a fibration-extensive type-theoretic model topos, which subsumes our previous assumptions that $\mE$ is bicomplete, locally small, and infinitarily extensive. In this setting, $\smE$ is again a type-theoretic model topos equipped with the Reedy model structure. In this section, we study the model theoretic classes of \emph{Reedy fibrations} and \emph{Reedy trivial fibrations}, as well as a class of \emph{left fibrations} presenting the class of left maps in $\sE$ introduced in \S\ref{ssec:stt}.

All three of these classes may be defined using the formalism introduced in \S\ref{ssec:enrichment}.

\begin{rmk}\label{rmk:relative-adjoint-preservation}
Given classes of maps $\mathcal{L}$ in a cocomplete category $\mC$ and $\mathcal{R}$ in a complete category $\mE$, we have an induced class of maps $[\mathcal{L}, \mathcal{R}]_\cont$ in $[\mC^\op, \mE]_\cont$ containing those natural transformations $u \colon F \to G$ whose pullback application sends $\mathcal{L}$ to $\mathcal{R}$. Given a left adjoint $L \colon \mC_2 \to \mC_1$ between cocomplete categories that preserves left classes $L(\mathcal{L}_2) \subset \mathcal{L}_1$ and a right adjoint $R \colon \mE_1 \to \mE_2$ between complete categories that preserves right classes $R(\mathcal{R}_1) \subset \mathcal{R}_2$, the induced functor
\[ \begin{tikzcd}{[\mC_1^\op, \mE_1]_\cont} \ar[r, "R \circ - \circ L"] & {[\mC_2^\op, \mE_2]_\cont} \end{tikzcd}\] sends maps in $[\mathcal{L}_1, \mathcal{R}_1]_\cont$ to maps in $[\mathcal{L}_2, \mathcal{R}_2]_\cont$.
\end{rmk}

In the case of our type theoretic model topoi, we have equivalences
\[\EHom \colon \smE^\cC \simeq [(\sSet^\cC)^\op, \mE]_\cont \qquad \text{and} \qquad \EHom_K \colon \slice{\smE}{K} \simeq [ \slice{\sSet}{K}^\op, \mE]_\cont\rlap{.}\] To avoid the cumbersome notation just introduced, we refer to classes of maps in $\smE^\cC$ or $\slice{\smE}{K}$ defined as in \cref{rmk:relative-adjoint-preservation} as \textbf{created by the $\mE$-relative hom functor} from a left class of maps in the simplicial world and a right class of maps in $\mE$.

\begin{defn}\label{defn:reedy-fib}
The \textbf{Reedy fibrations} in $\smE$ are created by the $\mE$-relative hom functor from the class of monomorphisms in $\sSet$ and the class of fibrations in $\mE$. As the monomorphisms in $\sSet$ are generated by the simplex boundary inclusions, this is to say that a map $f \colon Y \to X$ in $\smE$ is a \textbf{Reedy fibration} if for all $n \geq 0$, the Leibniz weighted limit weighted by the map of simplicial sets $\partial^n \colon \partial\Delta^n \hookrightarrow\Delta^n$ is a fibration in $\mE$:
  \[ \begin{tikzcd} \EHom(\Delta^n, Y) \cong Y_n \arrow[r, "{(f_n, \partial^n)}", two heads] & X_n \times_{M_nX} M_nY \cong \EHom(\Delta^n,X) \times_{\EHom(\partial\Delta^n,X)} \EHom(\partial\Delta^n,Y) \rlap{.} \end{tikzcd}\]
  Similarly, the \textbf{Reedy trivial fibrations} in $\smE$ are created by the $\mE$-relative hom functor from the monomorphisms in $\sSet$ and the class of trivial fibrations in $\mE$, which is to say that a map
  $f \colon Y \to X$ is a \textbf{Reedy trivial fibration} just these Leibniz weighted limits are trivial fibrations in $\mE$.
\end{defn}

We now introduce \emph{left fibrations} in $\smE$ to parallel the \emph{left maps} in $\sE$, the evident extensions of the traditional definition for simplicial spaces.\footnote{The authors of \cite[1.7]{BdB} and \cite[2.1]{KV} require that $Y$ and $X$ are Segal spaces, while \cite[A.27]{RS} only require $Y$ and $X$ to be Reedy fibrant.} The notion of left fibration only makes sense at the level of the type-theoretic model topos $\smE$ presenting the $\infty$-topos $\sE$, since fibrations are only defined there.

\begin{defn}\label{defn:left-fibration}
  The \textbf{left fibrations} in $\smE$ are the subclass of Reedy fibrations in $\smE$ created by the $\mE$-relative hom functor from the class of left anodyne maps in $\sSet$ and the class of trivial fibrations in $\mE$. This is to say, by \cref{lem:left-anodyne-initial-vertex}, that a map $f \colon Y \to X$ in $\smE$ is a \textbf{left fibration} if it is a Reedy fibration and the fibration
  \[ \begin{tikzcd}[sep=large] \EHom(\Delta^m,Y) \cong Y_m \arrow[r, two heads, "{(f_m,\ev_0)}"] & X_m \times_{X_0} Y_0 \cong \EHom(\Delta^m,X) \times_{\EHom(\Delta^0,X)} \EHom(\Delta^0,Y) \end{tikzcd}\] in $\mE$ is a trivial fibration for $m \ge 1$. Note that for a Reedy fibration $f \colon Y \to X$, the square \eqref{eq:left-map-square} is a pullback in $\iE$ just when this map is a trivial fibration, and thus the left fibrations in $\smE$ are those Reedy fibrations that define a left map is $\sE$.
\end{defn}

\begin{rmk}\label{rmk:covariant-fib}
By \cref{lem:covariant-fib}, equivalently, at the level of the type-theoretic model topos, a Reedy fibration $f \colon Y \twoheadrightarrow X$ in $\smE$ is a left fibration if and only if the Leibniz exponential with $\iota_0 \colon \Delta^0 \to \Delta^1$ is a trivial fibration.
\end{rmk}

For a small category $\cC$ or a simplicial set $K$, we have characterizations of the corresponding classes of projective fibrations in $\smE^\cC$ and sliced fibrations $\slice{\smE}{K}$.

\begin{lem}\label{lem:leib-hom-projective-fibrations}
The projective Reedy (trivial) fibrations in $\smE^\cC$ are created by the $\mE$-relative hom functor from the class of projective Reedy monomorphisms in $\sSet^\cC$ and the class of (trivial) fibrations in $\mE$. The projective left fibrations in $\smE^\cC$ are the subclass of the projective Reedy fibrations that are created from the projective left anodynes in $\sSet^\cC$ and the trivial fibrations in $\mE$.
\end{lem}
\begin{proof}
  For a projective Reedy fibration, projective Reedy trivial fibration, or projective left fibration, Leibniz weighted limit will carry the corresponding left class in $\sSet^\cC$ to the corresponding right class in $\mE$ by \cref{lem:leibniz-weighted} and the definitions of these classes in \cref{defn:reedy-fib,defn:left-fibration}.

  It remains to argue that a map in $\smE^\cC$ with the stated preservation property under Leibniz weighted limit belongs to the appropriate projective right class, which is the case just when such maps are preserved by the functors $\ev_c \colon \smE^\cC \to \smE$ defined for all objects $c \in \cC$. The evaluation functor is the $\mE$-relative right adjoint of the left adjoint $\lan_c \colon \sSet \to \sSet^\cC$ to the corresponding evaluation functor $\ev_c \colon \sSet^\cC \to \sSet$. The left adjoint preserves the corresponding (projective) left classes in $\sSet$ and $\sSet^\cC$ because its right adjoint preserves the corresponding right classes.
\end{proof}

\begin{lem}\label{lem:leib-hom-fibrations-in-slice}
The Reedy (trivial) fibrations in $\slice{\smE}{K}$ are created by the $\mE$-relative hom functor from the class of Reedy monomorphisms in $\slice{\sSet}{K}$ and the class of (trivial) fibrations in $\mE$. The left fibrations in $\slice{\smE}{K}$ are the subclass of the Reedy fibrations that are created from the left anodynes in $\slice{\sSet}{K}$ and the trivial fibrations in $\mE$.
\end{lem}
\begin{proof}
  For a Reedy fibration, Reedy trivial fibration, or left fibration in $\slice{\smE}{K}$ and a monomorphism or left anodyne map in $\slice{\sSet}{K}$,
  we know by \cref{lem:ehom-slice-pullback-ehom} that the Leibniz application of $\EHom_K$ applied to these maps is a pullback of the Leibniz application of $\EHom$ to the corresponding maps out of the slices. So these maps are fibrations of the appropriate type in $\mE$ by the definitions of these classes in \cref{defn:reedy-fib,defn:left-fibration} and the closure of the various classes of fibrations under pullback.

  It remains to argue that a map $p \colon (Y,y) \to (X,x)$ in $\slice{\smE}{K}$ with the stated preservation property under Leibniz application of $\EHom_K$ belongs to the appropriate right class. As the right classes in $\slice{\smE}{K}$ are created from the right classes out of the slice, what we have to show, after consulting \cref{defn:reedy-fib,defn:left-fibration}, is the corresponding preservation property under Leibniz application of $\EHom$ to the maps out of the slice. Given $m \colon A \to B$ in $\smE$, \cref{lem:ehom-slice-vs-ehom} tells us that $\widehat{\EHom}(m,p)$ is the coproduct over all $b \colon B \to K$ of the maps $\widehat{\EHom_K}(m \colon (A,bm) \to (B,b),p)$.
  The result follows by fibration-extensivity of $\mE$ (and \cref{prop:fibration-extensive-trivial}).
\end{proof}

We conclude with a series of lemmas, for later use, that will allow us to detect weak equivalences in $\mE$.

\begin{lem}\label{lem:ehom-equiv-cancellation}
Let $u \colon A \to B$ and $v \colon B \to C$ be maps between projectively cofibrant objects in $\sSet^\cC$ and $f \colon Y \to X$ in $\smE^\cC$ be a projective Reedy fibration.
Assume $\leibEHomC{u}{f}$ is a weak equivalence in $\mE$.
Then $\leibEHomC{v}{f}$ is a weak equivalence exactly if $\leibEHomC{v u}{f}$ is a weak equivalence.
\end{lem}
\begin{proof}
  The Leibniz maps in the statement are related by the following diagram:
\[
\begin{tikzcd}[row sep=large, column sep=small]
\EHom(C,Y) \arrow[ddd, bend right=65, "f \circ -"']\arrow[dr, "-\circ v"]
  \arrow[d, dashed, "{\leibEHomC{v}{f}}" description] \arrow[dd, dashed, bend right=50, "{\leibEHomC{vu}{f}}" description, pos=.6]\\
  \bullet \arrow[d, dashed] \arrow[r, dotted] \arrow[dr, phantom, "\lrcorner" very near start] & \EHom(B,Y) \arrow[dr, "-\circ u"] \arrow[d, dashed, "{\leibEHomC{u}{f}}" description] \\
  \bullet \arrow[r, dotted] \arrow[d, dotted] \arrow[dr, phantom, "\lrcorner" very near start]& \bullet \arrow[r, dotted] \arrow[d, dotted] \arrow[dr, phantom, "\lrcorner" very near start] & \EHom(A,Y) \arrow[d, "f \circ -"]\\
  \EHom(C,X) \arrow[r, "-\circ v"'] & \EHom(B,X) \arrow[r, "-\circ u"'] & \EHom(A,X) \rlap{.}
\end{tikzcd}
\]
  By assumption, $\leibEHomC{u}{f}$ is a weak equivalence between fibrant objects in the slice over $\EHom(B, X)$.
  Since weak equivalences between fibrant objects are closed under base change, its base change is a weak equivalence between fibrant objects over $\EHom(C, X)$.
  By Leibniz calculus, $\leibEHomC{v u}{f}$ is the composite of $\leibEHomC{v}{f}$ followed by this base change, as indicated by the dashed commutative triangle.
  The claim follows by 2-out-of-3.
\end{proof}

\begin{cor}
  \label{cor:ehom-equiv-projectively-cofibrant-trivial-fibration}
  Let $u \colon A \to B$ in $\sSet^{\cC}$ be a morphism between projectively cofibrant objects and $f \colon Y \to X$ in $\smE^\cC$ be a projective trivial fibration.
  Then $\leibEHomC{u}{f}$ in $\mE$ is a weak equivalence.
\end{cor}
\begin{proof}
Consider the triangle
\[
\begin{tikzcd}
&
  \emptyset
  \ar[dl, tail, "!_A"']
  \ar[dr, tail, "!_B"]
\\
  A
  \ar[rr, "u"']
&&
  B \rlap{.}
\end{tikzcd}
\]
By \cref{lem:leib-hom-projective-fibrations}, $\leibEHomC{!_A}{f}$ and $\leibEHomC{!_B}{f}$ are trivial fibrations, in particular weak equivalences.
So $\leibEHomC{u}{f}$ is a weak equivalence by \cref{lem:ehom-equiv-cancellation}.
\end{proof}

\begin{cor}
  \label{cor:ehom-equiv-she-left-fibration}
  Let $u \colon A \to B$ in $\sSet^{\cC}$ be a strong 0-oriented homotopy equivalence between projectively cofibrant objects and $f \colon Y \to X$ in $\smE^\cC$ be a projective left fibration.
  Then $\leibEHomC{u}{f}$ in $\mE$ is a weak equivalence.
\end{cor}
\begin{proof}
  By \cref{she-is-retract}, $u$ is a retract of the Leibniz tensor of $\iota_0 \colon \Delta^0 \hookrightarrow \Delta^1$ with $u$.
  Since weak equivalences are closed under retracts and $\leibEHomC{-}{f}$ preserves retracts, it suffices to check that $\leibEHomC{-}{f}$ sends $(\iota_0 \colon \Delta^0 \hookrightarrow \Delta^1) \mathbin{\hat{\otimes}} u$ to a weak equivalence.
  The said map is isomorphic to the application of $\leibEHomC{-}{-}$ to $u$ and
  the pointwise-defined Leibniz weighted limit $\widehat{\{\iota_0,f\}}$.
  By \cref{defn:left-fibration}, $\widehat{\{\iota_0,f\}}$ is a projective trivial fibration, so the result follows from \cref{cor:ehom-equiv-projectively-cofibrant-trivial-fibration}.
\end{proof}

\subsection{Comparing arrows and functions}\label{ssec:global-fun-arr}

The functors converting between ``arrows'' and ``functions'' may be defined as $\mE$-relative right adjoints.

\begin{defn}\label{defn:generalized-fun-arr-maps}
  We define a pair of functors:
  \[  \begin{tikzcd}[column sep=large, row sep=small] \arrtofun_\cC \colon \slice{\smE}{N\cC} \arrow[r,  "{\{\Wc,-\}}"]&   (\slice{\smE}{N\cC})^\cC \arrow[r, "{(N\cC)_*}"] & \smE^\cC \rlap{,} \\
      \funtoarr_\cC \colon
        \smE^\cC  \arrow[r, "{(N\cC)^*}"]& (\slice{\smE}{N\cC})^\cC \arrow[r,  "{\{\Ws,-\}^{\cC}}"]& \slice{\smE}{N\cC} \end{tikzcd}\]
as the $\mE$-relative right adjoints to $\arrtofunwgt$ and $\funtoarrwgt$:
\[
  \begin{tikzcd}[sep=large] \slice{\sSet}{N\cC} \arrow[r, bend right, dashed, "{\{\Wc,-\}}"', "\bot"] & \arrow[r, bend right, dashed, "{(N\cC)_*}"', "\bot"](\slice{\sSet}{N\cC})^\cC \arrow[l, "\Wc\otimes_\cC-"'] & \sSet^\cC \arrow[l, "{(N\cC)^*}"'] \rlap{,} &      \sSet^\cC     \arrow[r, bend right, dashed, "{(N\cC)^*}"',  "\bot"]
 & (\slice{\sSet}{N\cC})^\cC \arrow[l, "(N\cC)_!"'] \arrow[r, bend right, dashed, "{\{\Ws,-\}^{\cC}}"', "\bot"]  & \slice{\sSet}{N\cC}       \arrow[l, "\Ws\otimes -"'] \rlap{,} \\
  \slice{\smE}{N\cC} \arrow[r,  "{\{\Wc,-\}}"']&   (\slice{\smE}{N\cC})^\cC \arrow[r, "{(N\cC)_*}"'] & \smE^\cC \rlap{,} &
        \smE^\cC  \arrow[r, "{(N\cC)^*}"']& (\slice{\smE}{N\cC})^\cC \arrow[r,  "{\{\Ws,-\}^{\cC}}"']& \slice{\smE}{N\cC} \rlap{.}
  \end{tikzcd}
\]
\end{defn}

\begin{cor}\label{cor:arr-fun-laxnatural}
The mappings that send a category $\cC$ to the functors $\funtoarr_\cC$ and $\arrtofun_\cC$ extend to pseudofunctors
\[ \begin{tikzcd} \Cat^{\op} \ar[r, "\funtoarr"] & \Funpseudo({[1]},\Cat) \rlap{,} &  \Cat^{\op} \ar[r, "\arrtofun"] & \Funlax({[1]},\Cat)
\end{tikzcd}\]
 whose components at a functor $F \colon \cC \to \cD$ are natural transformations
  \[ \begin{tikzcd}
    \smE^\cD \arrow[d, "\funtoarr_\cD"'] \arrow[r, "\res_F"] \arrow[dr, phantomcenter, "{\cong}"] \arrow[dr, draw=none, center, "\funtoarr_F" below] &     {\smE}^\cC \arrow[d, "\funtoarr_\cC"] & & \slice{\smE}{N\cD} \arrow[d, "\arrtofun_\cD"'] \arrow[r, "{(NF)^*}"]
    \ar[dr, phantomcenter, "{\Rightarrow}"]
    \ar[dr, draw=none, center, "\arrtofun_F" below]
    & \slice{\smE}{N\cC} \arrow[d, "\arrtofun_\cC"]
  \\
    \slice{\smE}{N\cD} \arrow[r, "{(NF)^*}"'] & \slice{\smE}{N\cC} \rlap{,} &  &    \smE^\cD \arrow[r, "\res_F"'] &    \smE^\cC \rlap{.}
  \end{tikzcd}\]
\end{cor}
\begin{proof}
  By composing pseudofunctors, we define these natural transformations to be the $\mE$-relative conjugates of the natural transformations of \cref{cor:fun-to-arr-weights-pseudonatural,cor:arr-to-fun-weights-oplaxnatural}.
\end{proof}

By \cref{rmk:relative-adjoint-preservation}, preservation properties of functors between presheaf categories transpose to corresponding preservation properties of their $\mE$-relative right adjoints. In particular, we have the following corollaries.

\begin{cor}\label{lem:arr-to-fun-fibrations}
The functor $\arrtofun \colon \slice{\smE}{N \cC} \to \smE^\cC$ sends Reedy trivial fibrations to projective Reedy trivial fibrations, Reedy fibrations to projective Reedy fibrations, and left fibrations to projective left fibrations.
\end{cor}

\begin{proof}
  By \cref{rmk:relative-adjoint-preservation,lem:leib-hom-projective-fibrations,lem:leib-hom-fibrations-in-slice}, the claims transpose to \cref{lem:arr-to-fun-on-weights}.
\end{proof}

\begin{cor}\label{lem:fun-to-arr-fibrations}
The functor $\funtoarr \colon \smE^\cC \to \slice{\smE}{N \cC}$ sends projective Reedy trivial fibrations to Reedy trivial fibrations, projective Reedy fibrations to Reedy fibrations, and projective left fibrations to left fibrations.
\end{cor}

\begin{proof}
By \cref{rmk:relative-adjoint-preservation,lem:leib-hom-projective-fibrations,lem:leib-hom-fibrations-in-slice},
the claims transpose to \cref{lem:fun-to-arr-on-weights}.
\end{proof}

\begin{cor}\label{lem:arr-intermediate-fibrations}
The functor $\dom_* \circ \cod^* \colon \slice{\smE}{N \cC} \to \slice{\smE}{N \cC}$ sends projective Reedy trivial fibrations to Reedy trivial fibrations, projective Reedy fibrations to Reedy fibrations, and projective left fibrations to left fibrations.
\end{cor}

\begin{proof}
Note that the functor in question is the $\mE$-relative right adjoint of $\cod_! \circ \dom^* \colon \slice{\sSet}{N \cC} \to \slice{\sSet}{N \cC}$.
Thus, by \cref{rmk:relative-adjoint-preservation,lem:leib-hom-fibrations-in-slice}, the claims transpose to \cref{lem:arr-intermediate-on-weights}.
\end{proof}

Fix an object $\Gamma \in \smE$. We define relative versions of the functors $\arrtofun$ and $\funtoarr$ with the signature
\[ \begin{tikzcd} \slice{\smE}{N\cC \times \Gamma} \arrow[r, bend left, "\arrtofun^\Gamma_\cC" above] & (\slice{\smE}{\Gamma})^\cC \arrow[l, bend left, "\funtoarr^\Gamma_\cC"] \rlap{,} \end{tikzcd}\]
dropping the subscripts when the category $\cC$ is fixed and clear from context.

\begin{defn}\label{defn:relative-arr-to-fun}
We define the functor $\arrtofun^\Gamma \colon \slice{\smE}{N\cC \times \Gamma} \to (\slice{\smE}{\Gamma})^\cC$ in terms of $\arrtofun \colon \slice{\smE}{N\cC} \to \smE^\cC$ as follows. Given an object $p \colon E \to N\cC \times \Gamma$, regard $p$ as a morphism over $N\cC$ via the projection and form the following pullback in $\smE^\cC$:
\[ \begin{tikzcd} \arrtofun^\Gamma(E) \arrow[d, "{\arrtofun^\Gamma(p)}"'] \arrow[r] \arrow[dr, phantom, "\lrcorner" very near start] & \arrtofun(E) \arrow[d, "{\arrtofun(p)}"] \\ \const(\Gamma) \arrow[r, "{\eta^\dagger}"'] & \arrtofun(N\cC \times \Gamma) \rlap{.} \end{tikzcd}\]
Here the ``unit map'' $\eta^\dagger$ is the component at $\Gamma \in \smE$ of the $\mE$-relative conjugate of the natural transformation $\eta$ of \cref{defn:unit-map}:
  \[ \begin{tikzcd}
    \sSet^{\cC}
    \ar[d, "{\arrtofunwgt_\cC}"']
    \ar[r, "\colim_\cC"]
  &
    \sSet
    \ar[d, equals]
    \ar[dl, phantom,
    "{\Rightarrow\eta}"]
    & &
    \smE^{\cC} & \smE \arrow[l, "\const"'] \arrow[d, equals] \arrow[dl, phantom, "\Leftarrow\eta^\dagger"]
   \\
    \slice{\sSet}{N\cC} \arrow[r, "{(N\cC)_!}"'] &
    \sSet \rlap{,} & &
    \slice{\smE}{N\cC} \arrow[u, "\arrtofun_\cC"] & \smE \arrow[l, "(N\cC)^*"] \rlap{.}
    \end{tikzcd}
  \]
The left-hand vertical morphism defines the object $\arrtofun^\Gamma(p) \in (\slice{\smE}{\Gamma})^{N\cC}$. We define the action of $\arrtofun^\Gamma$ on morphisms similarly.
\end{defn}

\begin{defn}\label{defn:relative-fun-to-arr}
  We define the functor $\funtoarr^\Gamma \colon  (\slice{\smE}{\Gamma})^\cC \to \slice{\smE}{N\cC \times \Gamma} $ in terms of $\funtoarr \colon \smE^\cC \to \slice{\smE}{N\cC}$ as follows. Given a $\cC$-indexed diagram in $\slice{\smE}{\Gamma}^\cC$ which we denote by $q^\bullet \colon A^\bullet \to \Gamma$, regard it as a morphism $q^\bullet \colon A^\bullet \to \const(\Gamma)$ in $\smE^\cC$ and apply the functor $\funtoarr$ to obtain a morphism in $\slice{\smE}{N\cC}$:
  \[ \begin{tikzcd} \funtoarr(A^\bullet) \arrow[d, "{\funtoarr(p^\bullet)}"] \\ \funtoarr(\const(\Gamma)) \rlap{.} \end{tikzcd}\]
The $\mE$-relative conjugate of the natural isomorphism $\colim_\cC \circ \funtoarrwgt_\cC \cong (N\cC)_!$ is a natural isomorphism whose component at $\Gamma$ has the form $N\cC \times \Gamma \cong \funtoarr(\const(\Gamma))$. Thus, the codomain of $\funtoarr(p^\bullet)$ is $N\cC \times \Gamma$ and we may define $\funtoarr^\Gamma(p^\bullet)$ to be this object of $\slice{\smE}{N\cC\times \Gamma}$. We define the action of $\funtoarr^\Gamma$ on morphisms similarly.
\end{defn}

See \cref{cons:arr-fun,cons:fun-arr} for unpackings of \(\funtoarr^\Gamma_\cC\) and \(\arrtofun^\Gamma_\cC\) respectively in the case where $\cC$ is a finite ordinal category.

\begin{rmk}
The functors $\arrtofun^\Gamma$ and $\funtoarr^\Gamma$ can also be defined analogously to $\arrtofun$ and $\funtoarr$ but with $\smE$ replaced by $\slice{\smE}{\Gamma}$.
  \[ \begin{tikzcd}[column sep=large, row sep=tiny]
  \arrtofun^\Gamma \colon \slice{\smE}{N\cC \times \Gamma} \arrow[r,  "{\{\Wc,-\}_\Gamma}"] &  (\slice{\smE}{N\cC \times \Gamma})^{N\cC} \arrow[r, "{(N\cC \times \Gamma \to \Gamma)_*}"] & (\slice{\smE}{\Gamma})^{\cC} \rlap{,} \\
  \funtoarr^\Gamma \colon (\slice{\smE}{\Gamma})^{\cC} \arrow[r, "{(N\cC \times \Gamma \to \Gamma)^*}"] &  (\slice{\smE}{N\cC \times \Gamma})^{\cC} \arrow[r,  "{\{\Ws,-\}^{\cC}_\Gamma}"]& \slice{\smE}{N\cC \times \Gamma} \rlap{.} \end{tikzcd} \]
  However, for the purposes of analyzing these constructions, it is easier to implement them outside the slices as we have done.
\end{rmk}

In particular, our implementation of the relativized versions of the $\arrtofun$ and $\funtoarr$ in terms of the absolute versions of these functors allows us to easily prove the following change of base results:

\begin{lem}[pullback stability]
  \label{lem:arr-to-fun-pullback}
  Given a pullback square in $\slice{\smE}{N\cC}$ of the form
  \begin{equation}
    \begin{tikzcd} \label{eq:arr-to-fun-pullback-input}
      F \arrow[dr, phantom, "\lrcorner" very near start] \ar[d,"q"'] \ar[r,"u"] & E \ar[d,"p"] \\
      N\cC \times \Phi \ar[r,"N\cC \times t"'] & N\cC \times \Gamma \rlap{,}
    \end{tikzcd}
  \end{equation}
  we have an induced pullback square in $\smE^\cC$
  \[
    \begin{tikzcd}[column sep=large]
      \funtoarr^\Phi(F) \arrow[dr, phantom, "\lrcorner" very near start] \ar[d,"\funtoarr^\Phi(q)"'] \ar[r] & \funtoarr^\Gamma(E) \ar[d,"\funtoarr^\Gamma(p)"] \\
      \const(\Phi) \ar[r,"\const t"'] & \const(\Gamma) \rlap{.}
    \end{tikzcd}
  \]
\end{lem}
\begin{proof}
  We have a commutative cube
  \[
    \begin{tikzcd}[column sep={9em,between origins}]
     \arrtofun^\Phi(F) \ar[dd] \ar[dr,dashed] \ar[rr] && \arrtofun(F) \ar[dd,"\arrtofun(q)"{near end}] \ar[dr,"\arrtofun(u)"] & \\[-1em]
      & \arrtofun^\Gamma(E) \ar[rr,crossing over] && \arrtofun(E) \ar[dd,"\arrtofun(p)"] \\
      \const(\Phi) \ar[dr,"\const(t)"'] \ar[rr,"\eta^\dagger"{near end}] && \arrtofun(N\cC \times \Phi) \ar[dr,"\arrtofun(t)"] \\[-1em]
      & \const(\Gamma) \ar[uu,<-,crossing over] \ar[rr, "\eta^\dagger"'] && \arrtofun(N\cC \times \Gamma)
    \end{tikzcd}
  \]
  where the front and back faces are the pullback squares defining $\arrtofun^\Phi(F)$ and $\arrtofun^\Gamma(E)$, the right face is the action of $\arrtofun$ on \eqref{eq:arr-to-fun-pullback-input}, and the bottom face is a naturality square.
  Since $\arrtofun$ is a right adjoint and thus preserves pullbacks, the right face is a pullback square.
  Hence the left face is a pullback by pasting.
\end{proof}

\begin{lem}[pullback stability]
  \label{lem:fun-to-arr-pullback}
  Given a pullback square in $\smE^{\cC}$ of the form
  \begin{equation} \label{eq:fun-to-arr-pullback-input}
    \begin{tikzcd}[column sep=large]
      B^\bullet \arrow[dr, phantom, "\lrcorner" very near start] \ar[d,"q^\bullet"'] \ar[r, "u^\bullet"] & A^\bullet \ar[d,"p^\bullet"] \\
      \const(\Phi) \ar[r,"\const(t)"'] & \const(\Gamma) \rlap{,}
    \end{tikzcd}
  \end{equation}
  we have an induced pullback square in $\slice{\smE}{N\cC}$
  \begin{equation}
    \begin{tikzcd} \label{eq:fun-to-arr-pullback-output}
      \funtoarr^{\Phi}B^\bullet \arrow[dr, phantom, "\lrcorner" very near start] \ar[d,"\funtoarr^{\Phi}(q^\bullet)"'] \ar[r] & \funtoarr^{\Gamma}A^\bullet \ar[d,"\funtoarr^{\Gamma}(p^\bullet)"] \\
      N\cC \times \Phi \ar[r,"N\cC \times t"'] & N\cC \times \Gamma \rlap{.}
    \end{tikzcd}
  \end{equation}
\end{lem}
\begin{proof}
  By definition of $\funtoarr^\Gamma$/$\funtoarr^\Phi$ and the isomorphism of \cref{defn:unit-iso}, applying $\funtoarr$ to \eqref{eq:fun-to-arr-pullback-input} produces a square of the form \eqref{eq:fun-to-arr-pullback-output} up to isomorphism.
  As $\funtoarr$ is a right adjoint, it preserves pullbacks.
\end{proof}

The natural transformation $\kappa \colon \funtoarrwgt \circ \arrtofunwgt \to \id$ from \cref{lem:arr-to-fun-to-arr-point} induces an $\mE$-relative conjugate $\kappa^\dagger \colon \id \to \arrtofun \circ \funtoarr$.

\begin{lem}\label{lem:fun-arr-fun-equiv-base}
  For $\Gamma \in \sE$, the transformation $\kappa$ evaluated at $\const(\Gamma)$ factors as
  \[
    \begin{tikzcd}
      & \arrtofun(N\cC \times \Gamma) \ar[dr,"\cong"] \\
      \const(\Gamma) \arrow[ur,"\eta^\dagger"] \arrow[rr, "\kappa^\dagger"'] & & \arrtofun(\funtoarr(\const(\Gamma))) \rlap{.}
    \end{tikzcd}
  \]
\end{lem}
\begin{proof}
  \cref{rmk:constant-diagram-map-coherence} identifies these natural transformations on the weights side.
\end{proof}

\begin{lem}\label{lem:fun-arr-fun-equiv} Let $p^\bullet \colon A^\bullet \twoheadrightarrow \const(\Gamma)$ be a projective left fibration in $\smE^{\cC}$. Then the canonical comparison map to the pullback in the naturality square for $\kappa^\dagger \colon \id \to \arrtofun \circ \funtoarr$ at $p^\bullet$ is a weak equivalence of projective left fibrations over $\const(\Gamma)$:
  \[
  \begin{tikzcd}
  A^\bullet \arrow[ddd, two heads, "p^\bullet"'] \arrow[rr, "\kappa^\dagger"]  \arrow[dr, dashed] & [-4em] ~ & \arrtofun(\funtoarr A^\bullet) \arrow[ddd, two heads, "{\arrtofun(\funtoarr p^\bullet)}"]\\ & [-4em]\arrtofun^\Gamma(\funtoarr^\Gamma A^\bullet)  \arrow[ddl, two heads, "{\arrtofun^\Gamma(\funtoarr^\Gamma p^\bullet)}"] \ar[ur] \arrow[ddr, phantom, "\lrcorner" very near start]  \\ & & \\  \const(\Gamma) \arrow[rr, "\kappa^\dagger"'] & &  \arrtofun(\funtoarr(\const(\Gamma))) \rlap{.}
  \end{tikzcd}
  \]
  \end{lem}
  \begin{proof}
    The comparison map is the Leibniz application of $\kappa^\dagger$, the $\mE$-relative conjugate of $\kappa$, to the morphism $p^\bullet$ in $\smE^{\cC}$.
   By \cref{lem:left-fib-fiberwise-we}, it suffices to check that the component of this map at $c \in \cC$ and $[0] \in \DDelta$ is a weak equivalence in $\mE$. This component of the gap map is the Leibniz application of the $\mE$-relative conjugate of $\kappa_{\cC(c,-)}$ to $p^\bullet$, thus identified with $\leibEHomC{\kappa_{\cC(c,-)}}{p^\bullet}$. Here $\kappa_{\cC(c,-)} \colon N(\coslice{\cC}{c} \times_\cC \slice{\cC}{-}) \to \cC(c,-)$, the component of the natural transformation of \cref{lem:arr-to-fun-to-arr-point} at the representable $\cC(c,-)$, is the retraction of a strong 0-oriented homotopy equivalence with projective cofibrant domain and codomain by \cref{lem:arr-to-fun-to-arr-point-equiv}. Since $p^\bullet$ is a projective left fibration in $\smE^\cC$, \cref{cor:ehom-equiv-she-left-fibration} tells us that this map is a weak equivalence.
  \end{proof}

\begin{lem}\label{lem:arr-fun-arr-equiv} Let $p \colon E \twoheadrightarrow N\cC \times \Gamma$ be a left fibration in $\slice{\smE}{N\cC}$. Then there is a canonical span of weak equivalences of left fibrations over $N\cC \times \Gamma$:
    \[ \begin{tikzcd}\funtoarr^\Gamma (\arrtofun^\Gamma E) \arrow[d, "\funtoarr^\Gamma(\arrtofun^\Gamma p)"', two heads] & \dom_*^\Gamma (\cod^*_\Gamma E) \arrow[l] \arrow[r] \arrow[d, two heads] & E \arrow[d, two heads, "p"]  \\ N\cC \times \Gamma \arrow[r, equals] & N\cC \times \Gamma \arrow[r, equals] & N\cC \times \Gamma \rlap{.}
    \end{tikzcd}
    \]
\end{lem}
\begin{proof}
The $\mE$-relative conjugates of the natural transformations $\mu \colon \arrtofunwgt \circ \funtoarrwgt \to \cod_! \circ \dom^*$ and $\nu \colon \id \to \cod_! \circ \dom^*$ of \cref{lem:fun-to-arr-to-fun-point} define a span of natural transformations between endofunctors of $\slice{\smE}{N\cC}$  \[ \begin{tikzcd} \funtoarr\circ\arrtofun & \dom_* \circ \cod^* \ar[l, "\mu^\dagger"'] \ar[r, "\nu^\dagger"] & \id \end{tikzcd}\] whose naturality squares at the map $p$ in $\slice{\smE}{N\cC}$ have the form:
  \begin{equation}\label{eq:mu-nu-span} \begin{tikzcd} \funtoarr(\arrtofun E) \arrow[d, two heads, "{\funtoarr(\arrtofun p)}"'] & \dom_*\cod^* E \arrow[d, "\dom_*\cod^*p", two heads] \arrow[l] \arrow[r] & E \arrow[d, two heads, "p"] \\ \funtoarr(\arrtofun(N\cC \times \Gamma)) & \dom_*\cod^*(N\cC \times \Gamma) \arrow[l] \arrow[r] & N\cC \times \Gamma \rlap{.}
  \end{tikzcd} \end{equation}
The codomain of this morphism of spans is given by the components of $\mu^\dagger \circ (N\cC)^*$ and $\nu^\dagger \circ (N\cC)^*$ at $\Gamma \in \smE$. Thus, we may pull back this span along the components at $\Gamma$ of the $\mE$-relative conjugates of the natural transformations of \eqref{eq:how-we-pullback-the-span} to obtain a span of maps over $N\cC \times \Gamma$. By \cref{defn:relative-arr-to-fun,defn:relative-fun-to-arr}, the left fibration $\funtoarr^\Gamma(\arrtofun^\Gamma p)$ is defined by pulling back the map $\funtoarr(\arrtofun p)$ along the unit map:
    \[
      \begin{tikzcd}
      \funtoarr^\Gamma(\arrtofun^\Gamma E)  \arrow[d, two heads, "{\funtoarr^\Gamma(\arrtofun^\Gamma p)}"'] \arrow[r] \arrow[dr, phantom, "\lrcorner" very near start] & \funtoarr(\arrtofun E) \arrow[d, two heads, "{\funtoarr(\arrtofun p)}"] \\ N\cC \times \Gamma \arrow[r, "\eta^\dagger"'] & \funtoarr(\arrtofun (N\cC \times \Gamma)) \rlap{,}
      \end{tikzcd}
    \]
  so this constructs the span of the statement.\footnote{As our notation suggests, the middle left fibration can be equivalently be constructed by pulling $p$ back along $\cod \times \Gamma \colon N\cC^{[1]}\times \Gamma \to N\cC \times \Gamma$ then pushing forward along $\dom \times \Gamma \colon N\cC^{[1]}\times \Gamma \to N\cC\times \Gamma$.}

  It remains to argue that the left and right maps in this span are weak equivalences.
  By pullback pasting, the left map is the pullback of $\leibEHomNC{\mu}{p}$, the gap map in the left-hand square of \eqref{eq:mu-nu-span}, while the right map is the pullback of $\leibEHomNC{\nu}{p}$, so it suffices to show that these maps are weak equivalences. By \cref{lem:left-fib-fiberwise-we} it suffices to demonstrate that their components at $[0] \in \DDelta$ are weak equivalences. This corresponds to taking components of the natural transformations $\mu$ and $\nu$ at objects of the form $c \colon \Delta^0 \to N\cC$, so our task is to show that the Leibniz maps $\leibEHomNC{\mu_c}{p}$ and $\leibEHomNC{\nu_c}{p}$ are weak equivalences in $\mE$.
  By \cref{lem:fun-to-arr-to-fun-point-equiv}, the underlying maps of $\mu_c$ and $\nu_c$ in $\sSet$ are left anodynes.
  Hence $\leibEHomNC{\mu_c}{p}$ and $\leibEHomNC{\nu_c}{p}$ are weak equivalences (in fact trivial fibrations) by \cref{lem:leib-hom-fibrations-in-slice}.
\end{proof}

\subsection{Directed univalence}\label{ssec:dua-model}

We now state and prove our model-topos-level directed univalence theorem using the \textbf{covariant model structure} on slice categories of the form $\slice{\smE}{\Gamma}$, whose fibrant objects are left fibrations over $\Gamma$. This was developed when $\mE$ is the model topos of spaces by Rasekh \cite[3.12]{rasekh} based on prior work of Lurie \cite[\S2.1.4]{lurie-topos} and may be established in the same way at the level of a generic type-theoretic model topos.

\begin{defn}[covariant model structure]\label{defn:cov-model-str}
Let $\mE$ be a type-theoretic model topos and fix any $\Gamma \in \smE$. Write $S$ for the set of maps defined as the Leibniz product of a generating cofibration in $\smE$ with the discretely embedded map $\iota_0\colon \Delta^0 \to \Delta^1$, and define $S_\Gamma$ to be the inverse image of this set along the forgetful functor $\Gamma_! \colon \slice{\smE}{\Gamma} \to \smE$.

The {\bf covariant model structure} on $\slice{\smE}{\Gamma}$ is the left Bousfield localization of the sliced model structure induced from the Reedy model structure $\smE$ induced from the type theoretic model structure on $\mE$ at the set of cofibrations $S_\Gamma$.
\end{defn}

\begin{rmk}\label{rmk:cov-model-str}
  Note that the cofibrations in the covariant model structure coincide with the cofibrations in the sliced Reedy model structure. Since the model structure on $\smE$ is type theoretic, these are just the monomorphisms.

  Since type theoretic model structures are left proper and simplicial, the same is true of the covariant model structure. Here the simplicial set of maps $\Map_\Gamma(a,x)$ between objects $a \colon A \to \Gamma$ and $x \colon X \to \Gamma$ is defined by the pullback using the simplicial enrichment $\Map \colon \smE^\op \times \smE \to \sSet$:
\[
\begin{tikzcd} \Map_\Gamma(a,x) \arrow[d] \arrow[r] \arrow[dr, phantom, "\lrcorner" very near start] & \Map(A,X) \arrow[d, "{\Map(A,x)}"] \\ \Delta^0 \arrow[r, "a"'] & \Map(A,\Gamma) \rlap{.} \end{tikzcd}
\]
\end{rmk}

\begin{prop}[covariant model structure]\label{prop:cov-model-str}
  Let $\mE$ be a type-theoretic model topos and consider any $\Gamma \in \smE$. The fibrant objects in the covariant model structure are precisely the left fibrations over $\Gamma$.
\end{prop}
\begin{proof}
By definition, a fibrant object in the localized model structure is a fibrant object in $\slice{\smE}{\Gamma}$ (i.e., a Reedy fibration $x \colon X \twoheadrightarrow \Gamma$) such that the map of simplicial sets $\Map_\Gamma(s,x)$ induced by any $s \colon (A,a) \to (B,b)$ in $S_\Gamma$ is a trivial fibration:
\[
\begin{tikzcd}
\Map_\Gamma(b,x) \arrow[dr, dashed, two heads, "{\Map_\Gamma(s,x)}"'] \arrow[rrr, dotted] \arrow[ddd, dotted] & & &  {\Map(B,X)} \arrow[ddd, two heads, "{\Map(B,x)}"' near end] \arrow[rrr, "{\Map(s,X)}"] \arrow[dr, dashed, two heads, "{\widehat{\Map}(s,x)}" description] & & & \Map(A,X) \arrow[ddd, two heads, "{\Map(A,x)}"] \\ &  \Map_\Gamma(a,x) \arrow[ddl, dotted] \arrow[urr, dotted] \arrow[rrr, dotted] \arrow[ddrr, phantom, "\lrcorner" very near start] & & & \bullet \arrow[urr, dotted] \arrow[ddl, dotted] \arrow[ddrr, phantom, "\lrcorner" very near start] \\ \\  \Delta^0 \arrow[rrr, "b"] \arrow[rrrrrr, "a"', bend right=5] & & & {\Map(B,\Gamma)} \arrow[rrr, "{\Map(s,\Gamma)}"] & & & \Map(A,\Gamma) \rlap{.}
\end{tikzcd}
\]
As left fibrations are also Reedy fibrations, it suffices to show for every Reedy fibration $x \colon X \twoheadrightarrow \Gamma$ that the maps $\Map_\Gamma(s,x)$ are trivial fibrations of simplicial sets for all $s \in S_\Gamma$ if and only if the Leibniz exponential $\widehat{\intHom}(\iota_0,x)$ of $x$ with $\iota_0 \colon \Delta^0 \to \Delta^1$ is a trivial fibration in $\smE$, which characterizes the left fibrations by \cref{rmk:covariant-fib}.

We first claim that $x$ is an $S_\Gamma$-local fibrant object if and only if $\widehat{\Map}(s,x)$ is a trivial fibration of simplicial sets for all $s \in S$. Note that Reedy fibrancy of $x$, plus the fact that maps in $S$ and also $S_\Gamma$ are cofibrations, implies already that both $\widehat{\Map}(s,x)$ and $\Map_\Gamma(s,x)$ are Kan fibrations. Recall that a Kan fibration is a trivial fibration if and only if all of its fibers are contractible Kan complexes. The fibers of the maps $\widehat{\Map}(s,x)$ as $s$ ranges over $S$ are isomorphic to the fibers of the maps $\Map_\Gamma(s,x)$ as $s$ ranges over $S_\Gamma$. Thus, these two conditions are equivalent.

Thus, it remains to show that $\widehat{\Map}(s,x)$ lifts against (generating) monomorphisms of simplicial sets for all $s \in S$ if and only if $\widehat{\intHom}(\iota_0,x)$ lifts against (generating) monomorphisms in $\smE$. These each transpose to right lifting properties of the map $x \colon X \twoheadrightarrow \Gamma$ in $\smE$. In the former case, we are lifting against the class of Leibniz tensor products of (generating) monomorphisms in simplicial sets with maps in $S$, where the tensor is the left adjoint to the simplicial enrichment defined by $\Map$. In the latter case, we are lifting against the class of Leibniz products of (generating) monomorphisms in $\smE$ with the map $\iota_0 \colon \Delta^0 \to \Delta^1$, discretely embedded into $\smE$, which is to say we are lifting against the class $S$. The functor defined by Leibniz tensor with the map of simplicial sets $! \colon \emptyset \to \Delta^0$ is the identity, so the latter class is contained in the former, which means that the former lifting property implies the latter.

To see that the latter lifting property implies the former, suppose $m$ is a (generating) monomorphism of simplicial sets and $c$ is a generating cofibration in $\smE$ corresponding to a map $s = c \hat{\times} \iota_0$ in $S$. By associativity of the tensor product $U \otimes (A \times B) \cong (U \otimes A) \times B$ for $U \in \sSet$ and $A,B \in \smE$, for any monomorphism $m$ of simplicial sets, $m \hat{\otimes} (c \hat{\times} \iota_0) \cong (m \hat{\otimes} c) \hat{\times} \iota_0$. Since $\smE$ is a simplicial model category, the map $m \hat{\otimes}c$ is again a cofibration. So the latter lifting property implies the former and our conditions coincide.
\end{proof}

\begin{thm}[directed univalence]\label{thm:dua-model}
Let $\mE$ be a fibration-extensive type-theoretic model topos.
\begin{parts}
\item\label{itm:dua-model-quillen} For $\Gamma \in \smE$ and a 1-category $\cC$ the functors
\[ \begin{tikzcd} \slice{\smE}{N\cC \times \Gamma} \arrow[r, bend left, "\arrtofun^\Gamma" above] & (\slice{\smE}{\Gamma})^\cC \arrow[l, bend left, "\funtoarr^\Gamma"] \end{tikzcd}\]
define right Quillen functors between the covariant model structure on $\slice{\smE}{N\cC \times \Gamma}$ and the projective covariant model structure on $(\slice{\smE}{\Gamma})^\cC$.
\item\label{itm:dua-model-equiv}
For $\Gamma \in \smE$ and a 1-category $\cC$, there are natural transformations between right Quillen functors
\[
  \begin{tikzcd}
  &
    \slice{\smE}{N \cC\times \Gamma}
    \ar[dr, bend left=20, "\arrtofun^\Gamma"]
  & &
  &
  (\slice{\smE}{\Gamma})^\cC
  \ar[dr, bend left=20, "\funtoarr^\Gamma"]
&
  \\
    (\slice{\smE}{\Gamma})^\cC
    \ar[ur, bend left=20, "\funtoarr^\Gamma"]
    \ar[rr, bend left=26, phantomcenter, "\Uparrow \mathrlap{\kappa^\dagger}"]
    \ar[rr, "\id"']
  & { } &
    (\slice{\smE}{\Gamma})^\cC \rlap{,}
&     \slice{\smE}{N\cC\times \Gamma}
    \ar[ur, bend left=20, "\arrtofun^\Gamma"]
    \ar[r, "\cod^*_\Gamma"]
    \ar[rr, bend right=50, "\id"', ""{name=id}]
    \ar[rr, bend right=28, phantomcenter, "\Downarrow \mathrlap{\nu^\dagger}"]
    \ar[rr, bend left=28, phantomcenter, "\Uparrow \mathrlap{\mu^\dagger}"]
  & \slice{\smE}{N\cC^{[1]}\times\Gamma} \ar[r, "\dom_*^\Gamma"] &
    \slice{\smE}{N\cC\times\Gamma}
  \end{tikzcd}
\]
whose components at fibrant objects are weak equivalences.
\item\label{itm:dua-model-naturality}
For $t \colon \Phi \to \Gamma \in \smE$ and a functor $F \colon \cC \to \cD$, there are natural transformations between right Quillen functors
\[ \begin{tikzcd}
  (\slice{\smE}{\Gamma})^\cD \arrow[d, "\funtoarr_\cD^\Gamma"'] \arrow[r, "t^*\circ \res_F"] \ar[dr, phantomcenter, "{\cong}"] \ar[dr, draw=none, center, "\funtoarr_F^t" below] & (\slice{\smE}{\Phi})^\cC \arrow[d, "\funtoarr_\cC^\Phi"] & & \slice{\smE}{N\cD\times\Gamma} \arrow[d, "\arrtofun_\cD^\Gamma"'] \arrow[r, "{(NF \times t)^*}"]
  \ar[dr, phantomcenter, "{\Rightarrow}"]
  \ar[dr, draw=none, center, "\arrtofun_F^t" below]
  & \slice{\smE}{N\cC\times\Phi} \arrow[d, "\arrtofun_\cC^\Phi"]
\\
  \slice{\smE}{N\cD\times\Gamma} \arrow[r, "{(NF \times t)^*}"'] & \slice{\smE}{N\cC\times\Phi} \rlap{,} &  &    (\slice{\smE}{\Gamma})^\cD \arrow[r, "t^*\circ \res_F"'] &    (\slice{\smE}{\Phi})^\cC
\end{tikzcd}\]
whose components at fibrant objects are weak equivalences. These squares define the action on morphisms of a pair of pseudofunctors
\[ \begin{tikzcd} \Cat^{\op} \times \smE^\op \ar[r, "\funtoarr"] & \Funpseudo({[1]},\Cat) \rlap{.} &  \Cat^{\op}  \times \smE^\op\ar[r, "\arrtofun"] & \Funlax({[1]},\Cat)\rlap{.}
\end{tikzcd}\]
\end{parts}
\end{thm}
\begin{proof}
For \ref{itm:dua-model-quillen}, first note that both model structure are left proper and simplicial. Thus, by \cite[A.3.7.2]{lurie-topos}, to show that $\arrtofun^\Gamma$ and $\funtoarr^\Gamma$ are right Quillen, it suffices to show that these functors preserve trivial fibrations and fibrant objects. Per \cref{defn:relative-arr-to-fun}, the action of $\arrtofun^\Gamma$ on morphisms is a pullback of the action of $\arrtofun$, while the action of $\funtoarr^\Gamma$ is the action of $\funtoarr$. Thus it suffices to show that the non-sliced versions of these functors preserve left fibrations and trivial fibrations, which is demonstrated by \cref{lem:arr-to-fun-fibrations,lem:fun-to-arr-fibrations}.

  For \ref{itm:dua-model-equiv}, the natural transformations are defined using the $\mE$-relative conjugates of the natural transformations of \cref{lem:arr-to-fun-to-arr-point,lem:fun-to-arr-to-fun-point}, and their components are shown to be weak equivalences in \cref{lem:fun-arr-fun-equiv,lem:arr-fun-arr-equiv}.

For \ref{itm:dua-model-naturality}, the natural isomorphism $\funtoarr_F^t$ is the composite of two natural isomorphisms as below:
\[
\begin{tikzcd}[row sep=large]
  (\slice{\smE}{\Gamma})^\cD \arrow[d, "\funtoarr_\cD^\Gamma"'] \arrow[dr, phantomcenter, "\cong"] \arrow[r, "\res_F"] &   (\slice{\smE}{\Gamma})^\cC \arrow[r, "t^*"]  \arrow[d, "\funtoarr_\cC^\Gamma" description] \arrow[dr, phantomcenter, "\cong"] &   (\slice{\smE}{\Phi})^\cC \arrow[d, "\funtoarr_\cC^\Phi"]
\\
  \slice{\smE}{N\cD\times\Gamma} \arrow[r, "{(NF)^*}"'] &  \slice{\smE}{N\cC\times\Gamma} \arrow[r, "t^*"'] &  \slice{\smE}{N\cC\times\Phi} \rlap{.}
\end{tikzcd}
\]
The left-hand square arises from the natural isomorphism of \cref{cor:arr-fun-laxnatural}; its component at an object of $(\slice{\smE}{\Gamma})^\cD$ may be extracted from the corresponding component of that natural isomorphism at the corresponding morphism of ${\smE}^\cD$. \Cref{lem:fun-to-arr-pullback} gives the natural isomorphism in the right-hand square.

For $\arrtofun_F^t$, we similarly have a composite
\[ \begin{tikzcd}[row sep=large]
 \slice{\smE}{N\cD\times\Gamma} \arrow[d, "\arrtofun_\cD^\Gamma"'] \arrow[r, "{(NF)^*}"]
 \ar[dr, phantomcenter, "{\Rightarrow}"]
 \ar[dr, draw=none, center, "\arrtofun_F^\Gamma" below]
 &  \slice{\smE}{N\cC\times\Gamma} \arrow[d, "\arrtofun_\cC^\Gamma" description]\arrow[r, "t^*"] \arrow[dr, phantomcenter, "\cong"] & \slice{\smE}{N\cC\times\Phi} \arrow[d, "\arrtofun_\cC^\Gamma"]
\\
   (\slice{\smE}{\Gamma})^\cD \arrow[r, "\res_F"'] &    (\slice{\smE}{\Gamma})^\cC \arrow[r, "t^*"']  &  (\slice{\smE}{\Phi})^\cC
\end{tikzcd}\]
where the natural transformation in the left-hand square arises from the natural transformation of \cref{cor:arr-fun-laxnatural}, while \cref{lem:arr-to-fun-pullback} gives the natural isomorphism in the right-hand square.

It remains for \ref{itm:dua-model-naturality} to check that the components of $\arrtofun_F^\Gamma$ at fibrant objects are weak equivalences.
When an object $p \colon E \to N\cD \times \Gamma$ is regarded as a morphism in $\slice{\smE}{N\cD}$, the naturality square in $\smE^\cC$ for the non-sliced transformation $\arrtofun_F$ at $p$ is the square displayed below:
\[ \begin{tikzcd} & \res_F(\arrtofun_\cD E) \arrow[d] \arrow[r] & \arrtofun_\cC((NF)^*E) \arrow[d] \\ {\const(\Gamma)} \arrow[r, "\eta^\dagger"] \arrow[rr, bend right=10, "\eta^\dagger"'] & \res_F(\arrtofun_\cD(N\cD \times \Gamma)) \arrow[r] & \arrtofun_\cC((NF)^*(N\cD \times \Gamma)) \rlap{.} \end{tikzcd}\]
We that see the unit maps from $\const(\Gamma)$ commute with the bottom horizontal map from the commutative diagram of functors
\[ \begin{tikzcd}[sep=small] \cC \arrow[rr, "F"] \arrow[dr,"!"'] & & \cD \arrow[dl, "!"] \\ & {[0]} \rlap{.}
  \end{tikzcd}\]
The component of the natural transformation $\arrtofun_F^\Gamma$ at the object $p$ is defined by pulling back the gap map in this naturality square along the unit maps $\eta$.

The component of this natural transformation at a left fibration $p \colon E \twoheadrightarrow N\cD \times \Gamma$ is the morphism in $(\slice{\smE}{\Gamma})^\cC$ whose component at $c \in \cC$ is a pullback of $\leibEHomND{(\arrtofunwgt_F)_{\cC(c,-)}}{p}$, that is, the Leibniz weighted limit of the component of the natural transformation of \cref{cor:arr-to-fun-weights-oplaxnatural} at the discrete representable $\cC(c,-) \in \sSet^{\cC}$ with $p$, regarded as a morphism in $\slice{\smE}{N\cD}$. By \cref{lem:ehom-slice-pullback-ehom}, this map is a pullback of $\leibEHom{(\arrtofunwgt_F)_{\cC(c,-)}}{p}$, where we're now considering these maps outside of the slices. By \cref{cor:arr-to-fun-weights-oplaxnatural}.\ref{cor:arr-to-fun-weights-oplaxnatural:point}, the component is the map $NF \colon N(\coslice{\cC}{c}) \to N(\coslice{\cD}{Fc})$ and it is in the closure under right cancellation of the class of strong 0-oriented homotopy equivalences.
As the Leibniz weighted limit of a left fibration with a strong 0-oriented homotopy equivalence is a weak equivalence in $\smE$ by \cref{she-is-left-anodyne}, it follows by 2-of-3 that the map under consideration is also a weak equivalence.
\end{proof}

\section{The universal left fibration}\label{sec:universe}

We now turn our attention to directed univalence of \emph{universes} $\picov \colon \EUcov \to \Ucov$ of left fibrations, relativizing the directed univalence theorem of \cref{sec:dua-model} to the class of left fibrations classified by $\Ucov$.
We can state the relativized theorem in terms of the $\infty$-category $\cE$ presented by our type-theoretic model topos $\mE$.
Restricting to the case of ordinal index categories $\cC = [n]$, we assert an equivalence $\Ucov^{\Delta^\bullet} \simeq \Fun_\bullet$ between two simplicial objects in $\sE$: the simplicial object $\Ucov^{\Delta^\bullet}$ which at $n \in \mathbb{N}$ is the object of $n$-simplices in $\Ucov$, and a simplicial object $\Fun_\bullet$ which at $n \in \mathbb{N}$ is the object of $[n]$-indexed diagrams in $\Ucov$.

We will prove this equivalence in \cref{sec:dua-types}.
In the present section, we establish the existence of left fibration classifiers and prove some preliminary results.
In \cref{ssec:structured-left-fibrations}, we use Shulman's notions of fibred structure to show that for any type-theoretic model topos $\mE$, the type-theoretic model topos $\smE$ supports fibrant, univalent universes $\picov \colon \EUcov \to \Ucov$ of relatively $\kappa$-presentable left fibrations for suitable $\kappa$.
In \cref{ssec:dua-statement}, we describe the pair of simplicial objects $\Ucov^{\Delta^\bullet}$ and $\Fun_\bullet$ in $\smE$.
In \cref{ssec:univalence-of-fun}, we derive (non-directed) univalence theorems for $\Ucov^{\Delta^n}$ and $\Fun_n$, characterizing non-directed homotopies between maps $\Gamma \to \Ucov^{\Delta^n}$ and $\Gamma \to \Fun_n$ respectively in terms of the objects they classify.

\subsection{Structured left fibrations}\label{ssec:structured-left-fibrations}

Let $\mE$ be a type-theoretic model topos. Then $\smE$ is again a type-theoretic model topos, with the Reedy/injective model structure \cite[8.29]{shulman}. In particular, $\smE$ admits a locally representable, relatively acyclic, and homotopy invariant notion of fibered structure $\FF$ for which $\FF$-algebras are precisely Reedy fibrations.
In \cref{lem:triv-fib-loc-rep}, we demonstrated that $\mE$ also admits a locally representable, relatively acyclic, and homotopy invariant notion of fibered structure $\TF$ for which $\TF$-algebras are precisely trivial fibrations.
Using $\FF$ and $\TF$, we now define a locally representable, relatively acyclic, and homotopy invariant notion of fibered structure $\LL$ on $\smE$ for which the $\LL$-algebras are precisely the left fibrations of \cref{defn:left-fibration}.
Using $\LL$, we will obtain universes for left fibrations via \cref{thm:shulman-universe}.

Inspired by \cref{defn:left-fibration}:

\begin{defn}\label{defn:structured-left-fibration} A \textbf{structured left fibration} in $\smE$ is an $\FF$-algebra $f \colon Y \twoheadrightarrow X$ equipped with trivial fibration structures on the maps
  \[ \begin{tikzcd}[sep=large] Y_m \arrow[r, two heads, "{(f_m,\ev_0)}"] & X_m \times_{X_0} Y_0 \end{tikzcd}\] for all $m \geq 1$. That is we define a notion of fibered structure $\LL$ on $\smE$ by $\LL \coloneqq \FF \times_{\mE} \LL'$ where the category of $\LL'$-algebras and $\LL'$-morphisms is defined by the pullback
\[ \begin{tikzcd}[column sep=large]  \alg{\LL'} \arrow[r] \arrow[d] \arrow[dr, phantom, "\lrcorner" very near start] & \prod_{m \geq 1}\alg{\TF} \arrow[d] \\ \smE^{[1]}_{\mathup{pb}} \arrow[r, "{\widehat{\{\iota_0,-\}}}^{\Delta}"'] & \prod_{m \geq 1} \mE^{[1]}_{\mathup{pb}}
\end{tikzcd}
\]
along the product of the Leibniz weighted limit functors associated to the maps of weights $\iota_0 \colon \Delta^0 \to \Delta^m$ for $m \ge 1$. Because the Leibniz weighted limit functor preserves pullbacks, this construction defines a notion of fibered structure.
\end{defn}

By comparing \cref{defn:left-fibration} with \cref{defn:structured-left-fibration}, we see that a map admits an $\LL$-algebra structure just when it is a left fibration.

\begin{lem}\label{lem:left-fib-htpy-inv-rel-acyclic} The notion of fibered structure $\LL$ is homotopy invariant and relatively acyclic.
\end{lem}
\begin{proof}
To see that $\LL$ is relatively acyclic, consider a pullback square \eqref{eq:relative-acyclicity} between $\LL$-algebras. We first give $p$ a new $\FF$-algebra structure making the square an $\FF$-morphism and then give $p$ a new $\LL'$-algebra structure making the square an $\LL'$-morphism, which is possible by the relative acyclicity of $\TF$.

  Homotopy invariance of $\LL$ follows from its definition and \cref{lem:left-map-repleteness}: among fibrations, the left maps are equivalence-invariant.
\end{proof}

A slightly more involved argument is required to demonstrate that $\LL$ is locally representable. Since the codomain weight $\Delta^m$ is representable, the weighted limit functor $\lim_{\Delta^m} \colon \smE \to \mE$ is just defined to be evaluation at the $m$th object in the simplicial object. Hence, as an evaluation functor, $\lim_{\Delta^m}$ has both left and right adjoints, defined by left and right Kan extension. The right adjoint permits the construction that proves the following lemma.

\begin{lem}\label{lem:left-fib-loc-rep} The notion of fibered structure $\LL$ is locally representable.
\end{lem}
\begin{proof}
  By \cref{lem:triv-fib-loc-rep}, $\mE$ admits a locally representable notion of fibered structure $\TF$ for the trivial fibrations.
  It follows directly that the notion of fibred structure $\prod_{m \ge 1} \TF$ on $\prod_{m \ge 1} \mE$ appearing in \cref{defn:structured-left-fibration} is also locally representable.
In \cref{defn:structured-left-fibration}, the notion of fibered structure $\LL'$ on $\smE$ is defined by pulling back $\prod_{m \ge 1} \TF$ along the Leibniz application associated to the natural transformation
\[ \begin{tikzcd} \smE \arrow[r, bend left, "\ev_m"{above}] \arrow[r, phantom, "\Downarrow\ev_0"] \arrow[r, bend right, "\ev_0"{below}] & \prod_{m \geq 1} \mE \rlap{.} \end{tikzcd} \]
Since the domain has a right adjoint defined by right Kan extension, it in particular has a dependent right adjoint, so \cref{ex:leibniz-pullback-application-loc-rep-fibered-structure} applies to show that $\LL'$ is locally representable.
Since $\FF$ is locally representable by hypothesis, it then follows by \cref{ex:intersection-fibered-structure} that $\LL$ is locally representable as well.
\end{proof}

By \cref{lem:left-fib-htpy-inv-rel-acyclic,lem:left-fib-loc-rep}, we may apply \cref{thm:shulman-universe} to conclude:

\begin{prop}\label{prop:left-universe} There is a regular cardinal $\lambda$ such that for any regular cardinal $\kappa \triangleright \lambda$ there exists a morphism $\picov \colon \EUcov \to \Ucov$ such that:
  \begin{parts}
    \item $\picov \colon \EUcov \to \Ucov$ is a relatively $\kappa$-presentable left fibration.
    \item Every relatively $\kappa$-presentable left fibration is a pullback of $\picov$.
    \item $\Ucov$ is fibrant.
    \item $\picov$ is univalent. \qed
  \end{parts}
\end{prop}

We now fix some regular cardinal $\kappa$ for which we have universes $\pifib \colon \EUfib \to \Ufib$ for $\FF_\kappa$ and $\picov \colon \EUcov \to \Ucov$ for $\LL_\kappa$.

\begin{rmk}\label{rmk:ucov-univalence-application} Univalence of $\picov$ implies that maps $\Gamma \to \Ucov$ are determined up to homotopy by the equivalence class of the left fibrations they classify (see  \cite[\S A]{SattlerVezzosi-Partial}). Thus, to prove the (homotopy) commutativity of a diagram whose final vertex is $\Ucov$, it will suffice to construct an equivalence between the corresponding left fibrations over the initial vertex in the diagram.
\end{rmk}

By \cref{thm:shulman-universe}, since $\picov \colon \EUcov \to \Ucov$ is a relatively $\kappa$-presentable fibration, there exists a pullback square
\begin{equation}\label{eq:universe-comparison}
  \begin{tikzcd} \EUcov \arrow[d, "\picov"'] \arrow[r] \arrow[dr, phantom, "\lrcorner" very near start] & \EUfib \arrow[d, "\pifib"] \\  \Ucov \arrow[r, "\iota"'] & \Ufib \rlap{.}
  \end{tikzcd}
\end{equation}
Since $\pifib \colon \EUfib \to \Ufib$ is univalent, this square is homotopically unique (see \cite[\S A]{SattlerVezzosi-Partial}), providing a well-defined  classifying map $\iota \colon \Ucov \to \Ufib$.

\begin{prop}\label{prop:universal-prop-fibration}
  The map $\iota \colon \Ucov \to \Ufib$ is $-1$-truncated, defining  a homotopy monomorphism.
\end{prop}

\begin{proof}
  Since $\Ucov$ and $\Ufib$ are fibrant, $\iota \colon \Ucov \to \Ufib$ is a homotopy monomorphism just when the induced square on identity types is a homotopy pullback. By the pullback stability of the equivalence object factorization \cite[\S4]{shulman-reedy} at \(\iota\), we have a diagram
  \[ \begin{tikzcd} \Id(\EUcov) \arrow[ddr, two heads] \arrow[rr]\arrow[dr, wearrow] & &[-2em] \Id(\EUfib) \arrow[ddr, two heads] \arrow[dr, wearrow] \\ &  \Eq(\EUcov) \arrow[d, two heads] \arrow[rr, crossing over] \arrow[drr, phantom, "\lrcorner" very near start]  & & \Eq(\EUfib) \arrow[d, two heads] \\ & \Ucov\times\Ucov \arrow[rr] & &  \Ufib \times \Ufib \rlap{.}
  \end{tikzcd}
  \]
  By univalence of \(\Ucov\) and \(\Ufib\), the natural maps from the path objects to the objects of equivalences are weak equivalences. Thus the back square is a homotopy pullback, which is what we wanted to show.
\end{proof}

\subsection{The universes of arrows and functions}\label{ssec:dua-statement}

We interpret ``directed univalence'' in the $\infty$-categorical setting as a pointwise equivalence between simplicial objects in $\smE$ between a pair of simplicial objects we now define.

\begin{defn}[the simplicial object of ``arrows'']
By exponentiating the object $\Ucov \in \smE$ with the discretely embedded simplicies $\Delta^\bullet \colon \DDelta \to \smE$, we obtain a simplicial object $[n] \mapsto \Ucov^{\Delta^n} \colon \DDelta^\op \to \smE$ that we denoted by $\Ucov^{\Delta^\bullet}$.
\end{defn}

The other simplicial object is also defined from the universe $\pi \colon \EUcov \to \Ucov$ via the following lemma:

\begin{lem}[{\cite[7.1.4(ii)]{Jacobs}, \cite[6.2.2]{stenzel}}]\label{lem:map-internal-category}
In a locally cartesian closed 1-category $\mE$:
\begin{parts}
  \item\label{lem:map-internal-category:object} Any morphism $p \colon E \to B$ defines an internal category object with base $B$ and thus a simplicial object $\Fun^p_\bullet \colon \DDelta^\op \to \mE$.
  \item\label{lem:map-internal-category:functor} Moreover, for any $\Gamma \in \mE$, the 1-category obtained by post-composing with the covariant representable functor $\mE(\Gamma,-)$ has a canonical fully faithful functor to $\mE_{/\Gamma}$.
\end{parts}
\end{lem}
\begin{proof}
  The object of arrows, with its source and target maps, is the exponential in the slice over $B \times B$ of the maps $p \times B$ and $B \times p$. The idea of this construction is that an arrow $f$ over a pair of generalized elements $a, b \colon \Gamma \to B$ is given by lift below as below-left, which transposes to the data of a map over $\Gamma$ as below-right between the fibers of $p$ over $a$ and $b$
  \[ \begin{tikzcd} \Gamma \arrow[dr, "{(a, b)}"'] \arrow[rr, dashed, "f"] & & [-1cm] \hom_{B \times B}(E \times B, B \times E) \arrow[dl, "{(s,t)}"] \arrow[dr, phantom, "\leftrightsquigarrow"] & E_a\arrow[dr, "a^*p"'] \arrow[rr, dashed, "f"] & & E_b \arrow[dl, "b^*p"] \\ & B \times B & ~& ~& \Gamma \rlap{.} \end{tikzcd}\]
  This proves \ref{lem:map-internal-category:functor}.

  To finish the proof of \ref{lem:map-internal-category:object}, we describe the rest of the internal category structure. The identity map of $p$ in the slice over $B$ transposes to define the internal identity map:
  \[ \begin{tikzcd} B \arrow[dr, "{(B, B)}"'] \arrow[rr, dashed, "i"] & & [-1cm] \hom_{B \times B}(E \times B, B \times E) \arrow[dl, "{(s,t)}"] \arrow[dr, phantom, "\leftrightsquigarrow"] & E\arrow[dr, "p"'] \arrow[rr, dashed, "E"] & & E \arrow[dl, "p"] \\ & B \times B & ~& ~& B \rlap{.} \end{tikzcd}\]
  The composition map transposes to a map of the form
  \[ \begin{tikzcd} E \underset{B}{\times} \hom_{B \times B}(E \times B, B \times E) \underset{B}{\times} \hom_{B \times B}(E \times B, B \times E) \arrow[rr, "\circ", dashed] \arrow[dr] & [-1cm] &  E \arrow[dl, "p"] \\ & B \rlap{,}
  \end{tikzcd}\]
  which is obtained by composing two copies of the evaluation map:
  \begin{equation}\label{eq:internal-cat-evaluation}
  \begin{gathered}
  \begin{tikzcd} &[-4.3cm] \arrow[dl, "\pi"'] \arrow[dr, phantom, "\lrcorner" pos=.01] E \underset{B}{\times} \hom_{B\times B}(E \times B, B \times E)\underset{B}{\times}\hom_{B \times B}(E \times B, B\times E) \arrow[rr, dashed, "\ev"] & [-.7cm] &[-.8cm] \arrow[dl, "\pi"'] \arrow[ddl, phantom, "\llcorner" very near start] E \underset{B}{\times} \hom_{B\times B}(E \times B, B \times E) \arrow[rr, dashed, "\ev"] \arrow[d, "\pi"] &   &[-1cm] E \arrow[d, "p"] \\  E \times_B \arrow[d, "\pi"'] \hom_{B \times B}(E \times B, B \times E) \arrow[rr, dashed, "\ev"] & & E \arrow[d, "p"'] &  \hom_{B\times B}(E\times B, B \times E) \arrow[rr, "t"] \arrow[dl, "s"] &  & B \rlap{.} \\ \hom_{B \times B}(E \times B, B \times E) \arrow[rr, "t"'{name=B}] && B
  \end{tikzcd}\\[-1.5\normalbaselineskip]
  \end{gathered}
  \end{equation}
\end{proof}

\begin{defn}[the simplicial object of ``functors'']\label{defn:fun-n-object}
In particular, $\pi \colon \EUcov \to \Ucov$ in $\smE$ defines an internal category object $\Fun_\bullet \colon \DDelta^\op \to \smE$, dropping the dependence on $\pi$ from the notation. By \cref{lem:map-internal-category} its generalized elements $(A^\bullet,f^\bullet)\colon \Gamma \to \Fun_n$ are codes $A_k \colon \Gamma \to \Ucov$ for left fibrations over $\Gamma$ for $0 \leq k \leq n$ together with with maps $f_k$ over $\Gamma$ for $1 \leq k \leq n$ between the decoded left fibrations as displayed below:
\begin{equation}\label{eq:Fun-n-in-context} \begin{tikzcd} A^0 \arrow[drr, "p^0"', two heads] \arrow[r, "f^1"] & A^{1} \arrow[dr, "p^{1}" description, two heads] \arrow[r, "f^{2}"] & \cdots \arrow[r, "f^{n-1}"] & A^{n-1} \arrow[r, "f^n"] \arrow[dl, "p^{n-1}" description] & A^n \arrow[dll, "p^n", two heads] \\ & &  \Gamma \rlap{.} \end{tikzcd} \end{equation}
We can describe the objects $\Fun_n$ in the internal language as follows:
\[ \Fun_n \coloneqq \sum_{A^0,\ldots, A^n : \Ucov} (A^0 \to A^{1}) \times \cdots \times (A^{n-1} \to A^n) \rlap{,} \]
where we have silently coerced terms in $\Ucov$ in context $\Gamma$ to types over $\Gamma$.

The decoding \eqref{eq:Fun-n-in-context} defines an object in $(\slice{\smE}{\Gamma})^{[n]}$ or more usefully a morphism $p^\bullet \colon A^\bullet \to \const(\Gamma)$ in $\smE^{[n]}$ whose codomain is a constant diagram.
\end{defn}

\begin{rmk}\label{rmk:Fun-n-changing-context}
  We write $\psi_n^\bullet \colon \widetilde{\Fun}_n^\bullet \twoheadrightarrow \const(\Fun_n)$ for the universal diagram \eqref{eq:Fun-n-in-context} classified by the identity map at $\Fun_n$.
  A generalized element $(A^\bullet,f^\bullet) \colon \Gamma \to \Fun_n$ gives rise to a pullback square in $\smE^{[n]}$ between the decoded diagrams
  \[ \begin{tikzcd}[column sep=huge] A^\bullet \arrow[d, two heads, "p^\bullet"'] \arrow[r] \arrow[dr, phantom, "\lrcorner" very near start] & \widetilde{\Fun}_n^\bullet \arrow[d, two heads, "\psi_n^\bullet"] \\ \const(\Gamma) \arrow[r, "{\const(A^\bullet,f^\bullet)}"'] & \const(\Fun_n) \rlap{.} \end{tikzcd} \]
\end{rmk}

\subsection{Univalence for the function classifier}\label{ssec:univalence-of-fun}

By \cref{rmk:ucov-univalence-application}, a weak equivalence between left fibrations over $\Delta^n \times \Gamma$ induces a homotopy between their classifying maps $\Gamma \to \Ucov^{\Delta^n}$.
We now prove a similar result for \(\Fun_n\): a homotopy natural weak equivalence between sequences of $n$ composable maps between left fibrations over $\Gamma$ induces a homotopy of classifying maps $\Gamma \to \Fun_n$ (see \cref{cor:fun-univalence}).
This should be one component of a univalence principle for $\Fun_n$ --- an instance of the structure identity principle --- but for our purposes we only need and so only construct the mapping.

We first recall a consequence (in fact an equivalent formulation) of univalence \cite[2.13, 2.14]{van-den-berg:20}, specialized to the setting of a simplicial model category where we have cylinder objects given by tensor with the simplicial set $I \coloneqq \Delta^1$. (We use the notation ``$I$'' because the 1-simplex plays the role of the undirected interval here.)

\begin{prop}
  \label{prop:univalence-equivalence-to-homotopy}
  Let $\pi \colon \tilde{U} \twoheadrightarrow U$ be a universe for a notion of fibred structure $\FF$ on a simplicial model category.
  Suppose we have pullback squares
  \begin{equation}
    \label{eq:univalence-equivalence-to-homotopy:algebras}
    \begin{tikzcd}
      A \ar[dr, phantom, "\lrcorner" very near start] \ar[d,two heads,"p"'] \ar[r,"\tilde{a}"] & \tilde{U} \ar[d,two heads,"\pi"] \\
      \Gamma \ar[r,"a"'] & U \rlap{,}
    \end{tikzcd}
    \qquad
    \begin{tikzcd}
      B \ar[dr, phantom, "\lrcorner" very near start] \ar[d,two heads,"q"'] \ar[r,"\tilde{b}"] & \tilde{U} \ar[d,two heads,"\pi"] \\
      \Gamma \ar[r,"b"'] & U
    \end{tikzcd}
  \end{equation}
  exhibiting $p$ and $q$ as $\FF$-algebras.
  If $\pi$ is univalent, then any weak equivalence $e \colon (A,p) \wto (B,q)$ over $\Gamma$ induces simplicial homotopies $h \colon a \sim b$ and $h' \colon \tilde{a} \sim \tilde{b} e$ making the square
  \begin{equation}
    \label{eq:univalence-equivalence-to-homotopy:homotopies}
    \begin{tikzcd}
      I \otimes A \ar[d,two heads,"I \otimes p"'] \ar[r,"h'"] & \tilde{U} \ar[d,two heads,"\pi"] \\
      I \otimes \Gamma \ar[r,"h"'] & U
    \end{tikzcd}
  \end{equation}
  commute. \qed
\end{prop}

We have the following reformulation that does not mention the classifier.

\begin{prop}
  \label{prop:univalence-equivalence-to-glue}
  Let $\FF$ be a notion of fibred structure on a simplicial model category $\mE$ admitting a univalent universe.
  Then for any weak equivalence
\[    \begin{tikzcd}
      A \ar[dr,two heads,"p"'] \ar[rr,we',"e"] && B \ar[dl,two heads,"q"] \\
      & \Gamma &
    \end{tikzcd}
\]  in $\slice{\mE}{\Gamma}$ between $\FF$-algebras over $\Gamma$, there is an $\FF$-algebra $r \colon H \twoheadrightarrow I \otimes \Gamma$ and a map $g \colon I \otimes A \to H$ over $I \otimes \Gamma$ fitting into pullback squares
  \begin{equation}
    \label{eq:univalence-equivalence-to-glue}
    \begin{tikzcd}
      A \ar[dr, phantom, "\lrcorner" very near start] \ar[d,equals] \ar[r,"i_0"] & I \otimes A \ar[d,"g"] & A \ar[dl, phantom, "\llcorner" very near start] \ar[l,"i_1"'] \ar[d,we',"e"] \\
      A \ar[dr, phantom, "\lrcorner" very near start] \ar[d,two heads,"p"'] \ar[r,"r^*i_0"] & H \ar[d,two heads,"r"] & \ar[l,"r^*i_1"'] B \ar[dl, phantom, "\llcorner" very near start] \ar[d,two heads,"q"] \\
      \Gamma \ar[r,"i_0"'] & I \otimes \Gamma & \ar[l,"i_1"] \Gamma \rlap{.}
    \end{tikzcd}
  \end{equation}
\end{prop}
\begin{proof}
  Given homotopies $h \colon a \sim b$ and $h' \colon \tilde{a} \sim \tilde{b} e$ fitting in the square \eqref{eq:univalence-equivalence-to-homotopy:homotopies}, we take $r \colon H \twoheadrightarrow I \otimes \Gamma$ to be the $\FF$-algebra classified by $h \colon I \otimes \Gamma \to U$, then define $g \colon I \otimes A \to H$ by universal property of $H$:
  \[
    \begin{tikzcd}[/tikz/baseline=(\tikzcdmatrixname-\the\pgfmatrixcurrentrow-\the\pgfmatrixcurrentcolumn.base)]
      I \otimes A \ar[dr,dashed,"g"] \ar[ddr,bend right,two heads,"I \otimes p"'] \ar[drr,bend left,"h'"] &[-1em] & \\[-1em]
      & H \ar[dr, phantom, "\lrcorner" very near start] \ar[d,two heads,"r"'] \ar[r] & \tilde{U} \ar[d,two heads,"\pi"] \\
      & I \otimes \Gamma \ar[r,"h"] & U \rlap{.}
    \end{tikzcd} \qedhere
    \]
\end{proof}

\begin{lem}
  \label{lem:natural-equivalence-to-glue-morphism}
  Let $\mE$ be an extensive right proper Cisinski simplicial model category and let $\FF$ be a notion of fibred structure on $\mE$ whose algebras are fibrations.
  Fix a homotopy commutative square
  \begin{equation}
    \label{eq:natural-equivalence-to-glue-morphism:input}
    \begin{tikzcd}
      (A^0,p^0) \ar[dr,phantom,"\sim"] \ar[d,we,"e^0"']  \ar[r,"a"] & (A^1,p^1) \ar[d,we',"e^1"] \\
      (B^0,q^0) \ar[r,"b"'] & (B^1,q^1)
    \end{tikzcd}
  \end{equation}
  of morphisms between $\FF$-algebras over $\Gamma \in \mE$.
  If for each $k \in \{0,1\}$ we have an $\FF$-algebra $r^k \colon H^k \twoheadrightarrow I \otimes \Gamma$ and a map $g^k \colon I \otimes A^k \to H^k$ over $I \otimes \Gamma$ fitting in the diagram \eqref{eq:univalence-equivalence-to-glue}, then there is a morphism $H^0 \to H^1$ over $I \otimes \Gamma$ that pulls back along $i_0 \colon \Gamma \to I \otimes \Gamma$ to $a$ and along $i_1 \colon \Gamma \to I \otimes \Gamma$ to $b$.
\end{lem}
\begin{proof}
  Observe that the composite $g^1(I \otimes a) \colon I \otimes A^0 \to H^1$ over $I \otimes \Gamma$ restricts like so:
  \[
    \begin{tikzcd}
      A^0 \ar[d,"a"'] \ar[r,"i_0"] & I \otimes A^0 \ar[d,"g^1(I \otimes a)"] & \ar[l,"i_1"'] A^0 \ar[d,"e^1a" ] \\
      A^1 \ar[r,"(r^1)^*i_0"'] & H^1 & \ar[l,"(r^1)^*i_1"] B^1 \rlap{.}
    \end{tikzcd}
  \]
  Using the homotopy \eqref{eq:natural-equivalence-to-glue-morphism:input}, we may adjust this map to obtain some $g' \colon I \otimes A^0 \to H^1$ over $I \otimes \Gamma$ which instead fits in the diagram
  \[
    \begin{tikzcd}
      A^0 \ar[d,"a"'] \ar[r,"i_0"] & I \otimes A^0 \ar[d,dashed,"g'"] & \ar[l,"i_1"'] A^0 \ar[d,"be^0" ] \\
      A^1 \ar[r,"(r^1)^*i_0"'] & H^1 & \ar[l,"(r^1)^*i_1"] B^1 \rlap{.}
    \end{tikzcd}
  \]
  From this we have in particular an induced map $[g',(r^1)^*i_1b] \colon (I \otimes A^0) \sqcup_{A^0} B^0 \to H^1$.

  By right properness, the map $(r^0)^*{i_1} \colon B^0 \to H^0$ is a trivial cofibration.
  It follows that the induced map $[g^0,(r^0)^*i_1] \colon (I \otimes A^0) \sqcup_{A^0} B^0 \to H^0$ is a weak equivalence.
  We factor this map as a trivial cofibration followed by a trivial fibration through an object $K^0$, then construct the desired map $H^0 \to H^1$ as a composite of solutions to two lifting problems:
  \[
    \begin{tikzcd}[column sep=large, /tikz/baseline=(\tikzcdmatrixname-\the\pgfmatrixcurrentrow-\the\pgfmatrixcurrentcolumn.base)]
      A^0 \sqcup B^0 \ar[r,"i_0 \sqcup B"] \ar[dd,tail,"(r^0)^*\partial"'] & (I \otimes A^0) \sqcup_{A^0} B^0 \ar[r,"{[g',(r^1)^*i_1b]}"] \ar[d,tcofarrow] & H^1 \ar[dd,two heads,"r^1"] \\
      { } & K^0 \ar[d,tfibarrow] \ar[ur,dashed] \\
      H^0 \ar[ur,dashed] \ar[r,equals] & H^0 \ar[r,two heads,"r^0"'] & I \otimes \Gamma \rlap{.}
    \end{tikzcd} \qedhere
  \]
\end{proof}

\begin{cor}[Univalence map for $\Fun_n$]\label{cor:fun-univalence}
  Let $\mE$ be a type theoretic model topos. Given a homotopy commutative diagram
  \[
    \begin{tikzcd}
      (A^0,p^0) \ar[dr,phantomcenter,"\sim"] \ar[d,we,"e^0"']  \ar[r,"a^0"] & (A^1,p^1) \ar[dr,phantomcenter,"\sim"] \ar[d,we,"e^1"'] \ar[r,"a^1"] & \cdots \ar[dr,phantomcenter,"\sim"] \ar[r,"a^{n-1}"] & (A^n,p^n) \ar[d,we',"e^n"] \\
      (B^0,q^0) \ar[r,"b^0"'] & (B^1,q^1) \ar[r,"b^1"'] & \cdots \ar[r,"b^{n-1}"'] & (B^n,q^n)
    \end{tikzcd}
  \]
  of left fibrations over an object $\Gamma \in \smE$, the maps $\Gamma \to \Fun_n$ classifying $p^\bullet$ and $q^\bullet$ are homotopic.
\end{cor}
\begin{proof}
  Apply \cref{prop:univalence-equivalence-to-glue} (using univalence of $\picov$) and then \cref{lem:natural-equivalence-to-glue-morphism} to construct first the objects and then the morphisms of a diagram
  \[
    \begin{tikzcd}
      (H^0,r^0) \ar[r,"h^0"'] & (H^1,r^1) \ar[r,"h^1"'] & \cdots \ar[r,"h^{n-1}"'] & (H^n,r^n)
    \end{tikzcd}
  \]
  of $\LL$-algebras over $I \otimes \Gamma$ that restricts along $i_0 \colon \Gamma \to I \otimes \Gamma$ to $p^\bullet$ and along $i_1 \colon \Gamma \to I \otimes \Gamma$ to $q^\bullet$. This is classified by a map $h \colon I \otimes \Gamma \to \Fun_n$, defining the desired homotopy between the maps $\Gamma \to \Fun_n$ that classify $p^\bullet$ and $q^\bullet$.
\end{proof}

\section{Directed univalence of the universal left fibration}\label{sec:dua-types}

Continuing with the setup of the previous section, we assume a fibration-extensive type-theoretic model topos $\mE$, presenting an $\infty$-topos $\cE$, and fix a universe $\picov \colon \EUcov \to \Ucov$ in $\smE$ classifying left fibrations.
We now prove the $\infty$-categorical directed univalence equivalence $\Ucov^{\Delta^\bullet} \simeq \Fun_{\bullet}$ in $\sE$, relating the ``arrow'' and ``function'' classifiers defined in \cref{ssec:dua-statement}, by applying the model-topos level results of \cref{sec:dua-model}.

In an arbitrary context $\Gamma \in \smE$, maps $E \colon \Gamma \to \Ucov^{\Delta^n}$ are codes for left fibrations $p \colon E \twoheadrightarrow \Delta^n \times \Gamma$, while maps $A^\bullet \colon \Gamma \to \Fun_n$ are codes for left fibrations $p^k \colon A^k \twoheadrightarrow \Gamma$ for $0 \leq k \leq n$ together with maps between the decoded fibrations \eqref{eq:Fun-n-in-context}.
Thus these data correspond to fibrant objects in $\slice{\smE}{\Delta^n \times \Gamma}$ and $(\slice{\smE}{\Gamma})^{[n]}$ respectively.
We deduce the directed univalence theorem for $\Ucov$ from the model-topos-level directed univalence theorem,  \cref{thm:dua-model}, using the special case where the indexing category $\cC$ is a finite ordinal $[n]$ for $n \in \NN$.
In \S\ref{ssec:ordinal-case}, we specialize our previous results to this case and give a more explicit description of the functors \[ \begin{tikzcd} \slice{\smE}{\Delta^n \times \Gamma} \arrow[r, bend left, "\arrtofun^\Gamma_n" above] & (\slice{\smE}{\Gamma})^{[n]} \arrow[l, bend left, "\funtoarr^\Gamma_n"] \rlap{.} \end{tikzcd}\]

By \cref{thm:dua-model}, the functors $\arrtofun_n^\Gamma$ and $\funtoarr_n^\Gamma$ convert between left fibrations in each setting.
In \S\ref{ssec:dua-types}, we apply these functors to the universal left fibration over $\Delta^n \times \Ucov$ and the universal family of left fibrations over $\Fun_n$ respectively to obtain left fibrations encoded by maps $\arrtofun_n \colon \Ucov^{\Delta^n} \to \Fun_n$ and $\funtoarr_n \colon \Fun_n \to \Ucov^{\Delta^n}$.
We then prove the directed univalence theorem for $\Ucov$, \cref{thm:dua-types}, which says that they define inverse equivalences $\Ucov^{\Delta^n} \simeq \Fun_n$.

We conclude in \S\ref{ssec:applications} with some sample applications of \cref{thm:dua-types} to the theory of internal $\infty$-categories in an $\infty$-topos. A famous consequence of the non-directed univalence theorem is the structure identity principle~\cite[\S9.8]{HoTT}, characterizing identity types in types built from the universe. Directed univalence implies a similar \emph{structure homomorphism principle}~\cite[\S7.3.2]{weaver:24}~\cite[\S7.3]{gratzer-weinberger-buchholtz-univalence}, that we illustrate in a few basic cases.

\subsection{Converting between arrows and functions}\label{ssec:ordinal-case}

Specializing the 1-category $\cC$ of \cref{defn:generalized-fun-arr-maps} to ordinal categories $[n]$ for $n \in \NN$, we obtain functors $\arrtofun_n$ and $\funtoarr_n$ between the category $\slice{\smE}{\Delta^n}$ of $n$-ary ``arrows'' in $\smE$ and the category $\smE^{[n]}$ of $n$ composable ``functors'' or ``functions'' in $\smE$. We provide explicit descriptions of the constructions in this special case for sake of concreteness, in order to facilitate comparisons with other work, and so that we can verify that the sliced versions of these functors preserve relative $\kappa$-presentability.

\begin{defn}\label{defn:fun-arr-n-maps}
For each $n \in \NN$, the weights
\[ \begin{tikzcd}[row sep=tiny] [n] \arrow[r, "\Ws"] & \slice{\sSet}{\Delta^n} \rlap{,} & ([n])^\op \arrow[r, "{\Wc}"] & \slice{\sSet}{\Delta^n} \rlap{,} \\ k \arrow[r, maps to] & N(\slice{[n]}{k}) \rlap{,} & k \arrow[r, maps to ] & N(\coslice{[n]}{k}) \end{tikzcd}\]
define diagrams of the following form:
\[
     \begin{tikzcd}
        \Delta^0 \arrow[drr, hook, "\iota_{\leq 0}"'] \arrow[r, hook,   "\iota_{\leq 0}"] & \Delta^1 \arrow[dr, "\iota_{\leq 1}" description, hook] \arrow[r, hook, "\iota_{\leq 1}"] & \cdots \arrow[r, hook, "\iota_{\leq n-2}"] & \Delta^{n-1} \arrow[r, hook, "\iota_{\leq n-1}"] \arrow[dl, "\iota_{\leq n-1}" description, hook'] & \Delta^n \arrow[dll, equals] \arrow[dr, phantom, "\text{and}"] &         \Delta^0 \arrow[drr, hook, "\iota_{\geq n}"'] \arrow[r, hook,   "\iota_{\geq 1}"] & \Delta^1 \arrow[dr, "\iota_{\geq n-1}" description, hook] \arrow[r, hook, "\iota_{\geq 1}"] & \cdots \arrow[r, hook, "\iota_{\geq 1}"] & \Delta^{n-1} \arrow[r, hook, "\iota_{\geq 1}"] \arrow[dl, "\iota_{\geq 1}" description, hook'] & \Delta^n \arrow[dll, equals] \\
       \Ws \coloneqq   & & \Delta^n & &   ~& ~ &~  & \Delta^n & & \eqqcolon \Wc \rlap{.}
      \end{tikzcd}
\]
That is, $\Ws(k)$ is the initial segment inclusion $\iota_{\leq k} \colon \Delta^k \hookrightarrow \Delta^n$, while $\Wc(k)$ is the final segment inclusion $\iota_{\geq k} \colon \Delta^{n-k} \hookrightarrow \Delta^n$.

  For $n \in \NN$, \cref{defn:generalized-fun-arr-maps} specializes to functors
\[ \begin{tikzcd}[column sep=large, row sep=tiny]
  \arrtofun_n \colon \slice{\smE}{\Delta^n} \arrow[r,  "{\{\Wc,-\}}"] &  (\slice{\smE}{\Delta^n})^{[n]} \arrow[r, "{(\Delta^n)_*}"] & \smE^{[n]} \rlap{,} \\
  \funtoarr_n \colon \smE^{[n]} \arrow[r, "{(\Delta^n)^*}"] &  (\slice{\smE}{\Delta^n})^{[n]} \arrow[r,  "{\{\Ws,-\}^{[n]}}"]& \slice{\smE}{\Delta^n} \rlap{.} \end{tikzcd} \]
\end{defn}

For $\Gamma \in \smE$, we also have relative versions of these functors
\[ \begin{tikzcd} \slice{\smE}{\Delta^n \times \Gamma} \arrow[r, bend left, "\arrtofun^\Gamma_n" above] & (\slice{\smE}{\Gamma})^{[n]} \arrow[l, bend left, "\funtoarr^\Gamma_n"] \end{tikzcd}\]
via the constructions of \cref{defn:relative-arr-to-fun,defn:relative-fun-to-arr}. Using that $\Gamma$-indexed generalized elements in $\Fun_n$ encode sequences of $n$ composable arrows between left fibrations over $\Gamma$, we now unpack each of these constructions in more detail.

\begin{cons}[arrows to functions]\label{cons:arr-fun} Consider a left fibration $p \colon E \twoheadrightarrow \Delta^n \times \Gamma$. By \cref{lem:arr-to-fun-fibrations}, when we apply $\arrtofun_n^\Gamma \colon \slice{\smE}{\Delta^n \times \Gamma} \to (\slice{\smE}{\Gamma})^{[n]}$ to $p$ regarded as an object in $\slice{\smE}{\Delta^n \times \Gamma}$, we obtain a diagram of $n$ composable functions between left fibrations over $\Gamma$.

Per \cref{defn:relative-arr-to-fun}, $\arrtofun_n^\Gamma(p) \in (\slice{\smE}{\Gamma})^{[n]}$ is obtained by regarding $p \colon E \twoheadrightarrow \Delta^n \times \Gamma$ as a morphism in $\smE_{/\Delta^n}$, applying $\arrtofun_n \colon \slice{\smE}{\Delta^n} \to \smE^{[n]}$,
  and then pulling back along a canonical map $\eta$ whose domain is the constant $[n]$-indexed diagram at $\Gamma$:
\begin{equation}\label{eq:arr-to-fun-pullback} \begin{tikzcd} \arrtofun_n^\Gamma(E) \arrow[d,  "\arrtofun^\Gamma_n(p)"', two heads] \arrow[r] \arrow[dr, phantom, "\lrcorner" very near start] & \arrtofun_n(E) \arrow[d, "\arrtofun_n(p)", two heads] \\ \const(\Gamma) \arrow[r, "\eta"'] & \arrtofun_n(\Delta^n \times\Gamma) \rlap{.} \end{tikzcd}\end{equation}
By \cref{lem:arr-to-fun-fibrations}, $\arrtofun_n(p)$ is a pointwise left fibration over $\Gamma$, and by closure of left fibrations under pullback, so is $\arrtofun^\Gamma_n(p)$.

By unfolding the functor $\arrtofun_n$ of \cref{defn:fun-arr-n-maps}, we may describe this construction more explicitly.
To construct the functorial action of $\arrtofun_n$ on the morphism $p$, first form the pullbacks of $p \colon E \twoheadrightarrow \Delta^n \times \Gamma$, regarded as a map over $\Delta^n$, along the diagram of final simplex inclusions
\[ \begin{tikzcd} E_n \arrow[d, "p_n"', two heads] \arrow[r, hook] \arrow[dr, phantom, "\lrcorner" very near start] & E_{\geq n-1}\arrow[d, "p_{\geq n-1}"', two heads] \arrow[r, hook] \arrow[dr, phantom, "\lrcorner" very near start] & \cdots  \arrow[r, hook] & E_{\geq 1} \arrow[d, "p_{\geq 1}"', two heads] \arrow[dr, phantom, "\lrcorner" very near start] \arrow[r, hook] & E \arrow[d, "p", two heads] \\ \Delta^0 \times \Gamma \arrow[r, hook, "i_\last"] & \Delta^1 \times \Gamma \arrow[r, hook, "i_\last"] & \cdots \arrow[r, hook, "i_\last"] & \Delta^{n-1} \times \Gamma \arrow[r, hook, "i_\last"]& \Delta^n \times \Gamma \rlap{,} \end{tikzcd} \]
push each vertical map $p_{\geq k}$ forward along the final inclusion $\iota_{\geq k} \colon \Delta^{n-k} \hookrightarrow \Delta^n$ to form the $\Wc$-weighted limit in the slice over $\Delta^n$, and then push each object in this diagram forward along the projection away from $\Delta^n$ to form a diagram
\[ \begin{tikzcd} \Pi_{\Delta^n} E \arrow[d] \arrow[r, "\res"] & \Pi_{\Delta^{n-1}} E_{\geq 1}\arrow[d] \arrow[r, "\res"]  & \cdots  \arrow[r, "\res"] & \Pi_{\Delta^1} E_{\geq n-1} \arrow[d]  \arrow[r, "\res"] & E_n \arrow[d] \\ \Gamma^{\Delta^n} \arrow[r, "\res"] & \Gamma^{\Delta^{n-1}} \arrow[r, "\res"] & \cdots \arrow[r, "\res"] & \Gamma^{\Delta^1}\arrow[r, "\res"]& \Gamma \end{tikzcd} \]
where the restrictions are along the final inclusions. This is the diagram $\arrtofun_n(p)$ appearing on the right-hand side of the pullback \eqref{eq:arr-to-fun-pullback}.

From the description of $\eta$ given in \cref{ex:constant-unit-map}, we see that the components at $k \in [n]$ of the map $\eta$ are the constant maps $\Gamma \to \Gamma^{\Delta^{n-k}}$. The pullback along these ``constant diagram maps'' then defines $\arrtofun_n^\Gamma(p)$. Alternatively, $\arrtofun^\Gamma_n(p)^k$ is the pushforward of $p_{\geq k} \colon E_{\geq k} \to \Delta^{n-k} \times \Gamma$ along the projection $\pi \colon \Delta^{n-k} \times \Gamma \to \Gamma$.
\end{cons}

\begin{cons}[functions to arrows]\label{cons:fun-arr}
Consider a sequence of $n$ composable morphisms between left fibrations over $\Gamma$:
  \begin{equation}\label{eq:Fun-n-in-context-reprise} \begin{tikzcd} {A^0} \arrow[drr, "p^0"', two heads] \arrow[r, "f^1"] & {A^{1}} \arrow[dr, "p^{1}" description, two heads] \arrow[r, "f^{2}"] & \cdots \arrow[r, "f^{n-1}"] & {A^{n-1}} \arrow[r, "f^n"] \arrow[dl, "p^{n-1}" description] & {A^n} \arrow[dll, "p^n", two heads] \\ & &  \Gamma \rlap{.} \end{tikzcd} \end{equation}
This defines an object $p^\bullet \colon A^\bullet \twoheadrightarrow \Gamma$ in $(\slice{\smE}{\Gamma})^{[n]}$. By \cref{lem:fun-to-arr-fibrations}, when we apply $\funtoarr_n^\Gamma \colon   (\slice{\smE}{\Gamma})^{[n]} \to \slice{\smE}{\Delta^n \times \Gamma}$ we obtain a left fibration over $\Delta^n \times \Gamma$. Recall from \cref{defn:relative-fun-to-arr} that the map $\funtoarr_n^\Gamma(p^\bullet)$ is just obtained by applying the functor $\funtoarr_n \colon \smE^{[n]} \to \slice{\smE}{\Delta^n}$ of \cref{defn:fun-arr-n-maps} to the map $p^\bullet \colon A^\bullet \twoheadrightarrow \const(\Gamma)$ in $\smE^{[n]}$.

By unfolding $\funtoarr_n$, we may describe this construction more explicitly. Applying $(\Delta^n)^*$ pulls \eqref{cons:fun-arr} back to a diagram:
\begin{equation}\label{eq:fun-to-arr-diagram}
  \begin{tikzcd}
    \Delta^n \times A^0 \arrow[d, "p^0"', two heads] \arrow[r, "f^1"] & \Delta^n \times A^1 \arrow[d, "p^{1}"', two heads] \arrow[r, "f^{2}"] & \cdots \arrow[r, "f^{n-1}"] & \Delta^n \times A^{n-1} \arrow[r, "f^n"] \arrow[d, "p^{n-1}", two heads] & \Delta^n \times A^n \arrow[d, "p^n", two heads] \\ \Delta^n \times \Gamma \arrow[r, equals] & \Delta^n \times \Gamma \arrow[r, equals] & \cdots \arrow[r, equals]& \Delta^n \times \Gamma \arrow[r, equals]& \Delta^n \times \Gamma \rlap{.}
  \end{tikzcd}
\end{equation}
Now we may apply the weighted limit functor $\{\Ws,-\}^{{[n]}} \colon (\slice{\smE}{\Delta^n})^{{[n]}} \to \slice{\smE}{\Delta^n}$ to the map between diagrams \eqref{eq:fun-to-arr-diagram} to obtain a morphism in $\smE_{/\Delta^n}$ that we now describe explicitly.

The objects in our diagram have the form $\pi \colon \Delta^n \times A \to \Delta^n$ while the objects in the weight $\Ws$ have the form $\iota_{\leq k} \colon \Delta^k \to \Delta^n$. The cotensor $\iota_{\leq k} \pitchfork \pi$ of the former by the latter is given by the exponential in $\smE_{/\Delta^n}$ with the discretely embedded object $\iota_{\leq k} \colon \Delta^k \to \Delta^n$. Since we are exponentiating into a free object, this is equivalently the pushforward $\prod_{\iota_{\leq k}} (\Delta^k \times A) \to \Delta^n$ of $\pi \colon \Delta^k \times A \to \Delta^k$ along $\iota_{\leq k} \colon \Delta^k \to \Delta^n$.
Thus, the weighted limit of the domain diagram is formed by the limit
\[
    \begin{tikzcd}
     M \arrow[ddd,dashed] \arrow[rrr,dashed] \arrow[ddrr, phantom, "\lrcorner" very near start] &[25pt] { } &[10pt] { } &[40pt] \Delta^n \times A^n \arrow[d,"\iota_{\leq n-1}^*"] \rlap{.} \\
     & & \prod_{i_{\le n}} (\Delta^{n-1} \times A^{n-1}) \ar[d,"\iota_{\leq n-2}^*"] \arrow[r, "\prod_{i_{\le n-1}} (\Delta^{n-1} \times f^n)" below] & \prod_{i_{\le n-1}} (\Delta^{n-1} \times A^n) \\
     & \prod_{i_{\le 1}} (\Delta^1 \times A^1) \ar[d, "\iota_{\leq 0}^*"] \arrow[r, "\prod_{i_{\le 1}} (\Delta^1 \times f^2)" below] & \cdots \\
     \prod_{i_{\le 0}} (\Delta^0 \times A^0) \arrow[r, "\prod_{i_{\le 0}} (\Delta^0 \times f^1)" below] & \prod_{i_{\le 0}} (\Delta^0 \times A^1)
    \end{tikzcd}
\]

  For the codomain diagram we have an isomorphism $\funtoarr_n(\const(\Gamma)) \cong \Delta^n \times \Gamma$ arising from the conjugate under the $\mE$-enriched adjunctions of the natural isomorphism described in \cref{defn:unit-iso}. We can also calculate this directly: observe that when we take the weighted limit of the codomain diagram, the horizontal maps are all identities, so the weighted limit is $\Delta^n \times \Gamma$. Thus, this construction gives rise to a left fibration  $\funtoarr(p^\bullet) \colon M \twoheadrightarrow \Delta^n \times \Gamma$.
\end{cons}

\subsection{Directed univalence}\label{ssec:dua-types}

We now instantiate \cref{cons:arr-fun,cons:fun-arr} in the universal cases to define maps   \[ \begin{tikzcd} \Fun_\bullet \arrow[r, "\funtoarr_\bullet", shift left=.25em] & \Ucov^{\Delta^\bullet} \arrow[l, "\arrtofun_\bullet", shift left=.25em] \end{tikzcd}\]
for each $n \in \NN$.

\begin{cons}\label{cons:arr-to-fun-n}
    For any $n \in \NN$, there is a map $\arrtofun_n \colon \Ucov^{\Delta^n} \to \Fun_n$, built by instantiating \cref{cons:arr-fun} in a universal context $\Gamma = \Ucov^{\Delta^n}$.
  We take a pullback
  \[ \begin{tikzcd} \widetilde{\Arr}_n \arrow[d, two heads,"\ev^*\pi"'] \arrow[r] \arrow[dr, phantom, "\lrcorner" very near start] & \EUcov \arrow[d, two heads, "\pi"] \\ \Delta^n \times \Ucov^{\Delta^n} \arrow[r, "\ev"'] & \Ucov \end{tikzcd} \]
  to define a left fibration over $\Delta^n \times \Ucov^{\Delta^n}$.
  By applying \cref{cons:arr-fun}, we get a pointwise left fibration in $(\slice{\smE}{\Ucov^{\Delta^n}})^{[n]}$ defined by the pullback:
\[ \begin{tikzcd} \arrtofun_n^{\Ucov^{\Delta^n}}(\widetilde{\Arr}_n) \arrow[d,  "\arrtofun^{\Ucov^{\Delta^n}}_n(\ev^*\pi)"', two heads] \arrow[r] \arrow[dr, phantom, "\lrcorner" very near start] & \arrtofun_n(\widetilde{\Arr}_n) \arrow[d, "\arrtofun_n(\ev^*\pi)", two heads] \\ \const(\Ucov^{\Delta^n}) \arrow[r, "\eta"'] & \arrtofun_n(\Delta^n \times\Ucov^{\Delta^n}) \rlap{.} \end{tikzcd}
\]
  This is exactly a family of $n + 1$ left fibrations over $\Ucov^{\Delta^n}$ together with a sequence of $n$ composable maps between them.
  The operations involved in the definition of $\arrtofun^\Gamma_n$ preserve relative $\kappa$-presentability, so the classifying property of $\Ucov$ (\cref{prop:left-universe}) implies that each of the left fibrations is the decoding of a map into $\Ucov$. Thus, by \cref{rmk:Fun-n-changing-context},
  we obtain a classifying map $\arrtofun_n \colon \Ucov^{\Delta^n} \to \Fun_n$ as displayed:
  \begin{equation}\label{eq:arr-fun-pullback} \begin{tikzcd}[column sep=large] \arrtofun_n^{\Ucov^{\Delta^n}}(\widetilde{\Arr}_n) \arrow[d, two heads, "{\arrtofun_n^{\Ucov^{\Delta^n}}(\ev^*\pi)}"'] \arrow[r] \arrow[dr, phantom, "\lrcorner" very near start] & \widetilde{\Fun}_n^\bullet \arrow[d, two heads, "\psi^\bullet"] \\ \const(\Ucov^{\Delta^n}) \arrow[r, "{\const(\arrtofun_n)}"'] & \const(\Fun_n) \rlap{.} \end{tikzcd} \end{equation}
\end{cons}

\begin{cons}\label{cons:fun-to-arr-n} For any $n \in \NN$ there is a map $\funtoarr_n \colon \Fun_n \to \Ucov^{\Delta^n}$, built by instantiating \cref{cons:fun-arr} in a universal context $\Gamma = \Fun_n$. By \cref{defn:fun-n-object}, the identity map at $\Fun_n$ gives rise to a universal diagram $\psi_n^\bullet \colon \widetilde{\Fun}_n^\bullet \twoheadrightarrow \const(\Fun_n)$ introduced in \cref{rmk:Fun-n-changing-context}. By applying \cref{cons:fun-arr}, we obtain a left fibration with codomain $\Delta^n \times \Fun_n$.
  Each of the operations involved in the definition of $\funtoarr^\Gamma_n$ preserves relative $\kappa$-presentability, so this left fibration is classified by a functor into the covariant universe
  \[ \begin{tikzcd}[column sep=large]
    \funtoarr_n^{\Fun_n}(\widetilde{\Fun}_n^\bullet) \arrow[d,two heads, "{\funtoarr_n^{\Fun_n}(\psi_n^\bullet)}"'] \arrow[r] \arrow[dr, phantom, "\lrcorner" very near start] & \EUcov \arrow[d,two heads, "\pi"] \\  \Delta^n \times \Fun_n \arrow[r, "\funtoarr_n^\dagger"'] & \Ucov \end{tikzcd} \]
that transposes to define our map $\funtoarr_n \colon \Fun_n \to \Ucov^{\Delta^n}$.
\end{cons}

We now show that the maps $\funtoarr_\bullet \colon \Fun_\bullet \to \Ucov^{\Delta^\bullet}$ assemble into the components of a homotopy commutative natural transformation whose homotopies are coherent with respect to pairwise composition in the simplex category.

\begin{lem}\label{lem:funtoarr-htpy-naturality}
  The maps $\funtoarr_\bullet \colon \Fun_\bullet \to \Ucov^{\Delta^\bullet}$ assemble into a homotopy commutative natural transformation of simplicial objects. Moreover, the homotopies $\phi_\alpha$ for $\alpha \in \Delta^\to$ can be chosen such that for any composable pair of maps
  \[ \begin{tikzcd} {[k]} \arrow[r, "\alpha"]  & {[m]} \arrow[r, "\beta"] & {[n]} \end{tikzcd} \]
  in $\Delta$, the pasted composite
  \[ \begin{tikzcd}  \Fun_n \arrow[r, "\beta^*"] \arrow[d, "\funtoarr_n"']\arrow[dr, phantom, "{\phi_\beta}"]& \Fun_m \arrow[d, "\funtoarr_m" description] \arrow[r, "\alpha^*"] \arrow[dr, phantom, "{\phi_\alpha}"] & \Fun_k \arrow[d, "\funtoarr_k"]  \\ \Ucov^{\Delta^n} \arrow[r, "\beta^*"'] & \Ucov^{\Delta^m} \arrow[r, "\alpha^*"'] & \Ucov^{\Delta^k}
  \end{tikzcd}\]
  is homotopic to $\phi_{\beta\alpha}$.
\end{lem}
\begin{proof}
  To define $\phi_\alpha$, we argue that the transposed square
    \[ \begin{tikzcd}  \Delta^k \times \Fun_m \arrow[d, "\id \times \alpha^*"'] \arrow[r, "\alpha \times \id"] & \Delta^m \times \Ucov^{\Delta^m} \arrow[d, "\funtoarr_m^\dagger"] \\ \Delta^k \times \Fun_k \arrow[r, "\funtoarr_k^\dagger"'] & \Ucov
  \end{tikzcd}\]
  commutes up to homotopy by showing that the classifying left fibrations are equivalent over $\Delta^k \times \Fun_m$ and appealing to univalence of $\pi \colon \EUcov \twoheadrightarrow \Ucov$. By \cref{cons:arr-to-fun-n}, the upper-right composite classifies the left fibration below-left:
    \[ \begin{tikzcd}[column sep=large]
    ((\alpha \times \id)^* \circ \funtoarr_m^{\Fun_m})(\widetilde{\Fun}_m^\bullet)  \arrow[d, two heads, "{((\alpha \times \id)^* \circ \funtoarr_m^{\Fun_m})(\psi_m^\bullet)}"'] \arrow[r] \arrow[dr, phantom, "\lrcorner" very near start] &  \funtoarr_m^{\Fun_m}(\widetilde{\Fun}_m^\bullet) \arrow[d,two heads, "{\funtoarr_m^{\Fun_m}(\psi_m^\bullet)}"'] \arrow[r] \arrow[dr, phantom, "\lrcorner" very near start] & \EUcov \arrow[d,two heads, "\pi"] \\  \Delta^k \times \Fun_m \arrow[r, "\alpha \times \id"'] & \Delta^m \times \Fun_m \arrow[r, "\funtoarr_m^\dagger"'] & \Ucov \rlap{.} \end{tikzcd} \]
    By Theorem \ref{thm:dua-model}\ref{itm:dua-model-naturality}, this is equivalent to the left fibration over $\Delta^k \times \Fun_m$ obtained by first restricting the universal diagram $\psi^\bullet_m \in (\slice{\smE}{\Fun_m})^{[m]}$ along the functor $\alpha \colon [k] \to [m]$ to obtain a diagram $\psi^{\alpha(\bullet)}_m \in (\slice{\smE}{\Fun_m})^{[k]}$ and then applying $\funtoarr_k^{\Fun_m}$.

    Similarly, the lower-left composite classifies the left fibration below-left:
    \[ \begin{tikzcd}[column sep=large]
    ((\id\times \alpha^*)^* \circ \funtoarr_k^{\Fun_k})(\widetilde{\Fun}_k^\bullet)  \arrow[d, two heads, "{((\id \times \alpha^*)^* \circ \funtoarr_k^{\Fun_k})(\psi_k^\bullet)}"'] \arrow[r] \arrow[dr, phantom, "\lrcorner" very near start] &  \funtoarr_k^{\Fun_k}(\widetilde{\Fun}_k^\bullet) \arrow[d,two heads, "{\funtoarr_k^{\Fun_k}(\psi_k^\bullet)}"'] \arrow[r] \arrow[dr, phantom, "\lrcorner" very near start] & \EUcov \arrow[d,two heads, "\pi"] \\  \Delta^k \times \Fun_m \arrow[r, "\id \times \alpha^*"'] & \Delta^k \times \Fun_k \arrow[r, "\funtoarr_k^\dagger"'] & \Ucov \rlap{.} \end{tikzcd} \]
    By Theorem \ref{thm:dua-model}.(\ref{itm:dua-model-naturality}), this is equivalent to the left fibration over $\Delta^k \times \Fun_m$ obtained by first pulling back the universal diagram $\psi^\bullet_k \in (\slice{\smE}{\Fun_k})^{[k]}$ along $\alpha^* \colon \Fun_m \to \Fun_k$ to obtain a diagram $\alpha^*(\psi^\bullet_k) \in (\slice{\smE}{\Fun_m})^{[k]}$ and then applying $\funtoarr_k^{\Fun_m}$.
   Now the result follows from the fact that the pullback of $\psi^\bullet_k$ along $\alpha^* \colon \Fun_m \to \Fun_k$ coincides with the restriction of $\psi^\bullet_m$ along $\alpha \colon [k] \to [m]$ by the construction of the simplicial object $\Fun_\bullet$ in \cref{lem:map-internal-category}:
      \[ \begin{tikzcd}[column sep=large] \widetilde{\Fun}_m^{\alpha(\bullet)} \arrow[d, two heads, "\psi_m^{\alpha(\bullet)}"'] \arrow[r] \arrow[dr, phantom, "\lrcorner" very near start] & \widetilde{\Fun}_k^\bullet \arrow[d, two heads, "\psi_k^\bullet"] \\ \const(\Fun_m) \arrow[r, "{\const(\alpha^*)}"'] & \const(\Fun_k) \rlap{.} \end{tikzcd} \]

For the coherence statement, note that there are commuting isomorphisms in $(\slice{\smE}{\Fun_n})^{[k]}$ between $\psi_n^{\beta\alpha(\bullet)}$, the pullback of $\psi^{\alpha(\bullet)}_m$ along $\const(\beta^*)$, and the pullback of $\psi^\bullet_k$ along $\const((\beta\alpha)^*)$. After applying $\funtoarr_k^{\Fun_n}$, these become commuting isomorphisms between left fibrations over $\Delta^k \times \Fun_n$, which under univalence of $\pi \colon \EUcov \to \Ucov$ define the claimed homotopy between the $\phi_{\beta\alpha}$ and the pasted composite of $\phi_\beta$ and $\phi_\alpha$.
\end{proof}

We can now state our main result at the $\infty$-topos level.

\begin{thm}[directed univalence]\label{thm:dua-types}
Let $\iE$ be an $\infty$-topos. The maps
\[ \begin{tikzcd} \Fun_\bullet \arrow[r, "\funtoarr_\bullet", shift left=.25em] & \Ucov^{\Delta^\bullet} \arrow[l, "\arrtofun_\bullet", shift left=.25em] \end{tikzcd}\]
define a pointwise equivalence between the simplicial objects $\Fun_\bullet$ and $\Ucov^{\Delta^\bullet}$ in $\sE$ that is natural up to homotopy.
\end{thm}

By \cref{thm:dua-types} we know that the model-topos-level avatars of these maps, the functors $\arrtofun_n$ and $\funtoarr_n$, are equivalence inverses in a suitable homotopical sense.

\begin{proof}[Proof of \cref{thm:dua-types}]
In \cref{cons:arr-to-fun-n}, the map $\arrtofun_n \colon \Ucov^{\Delta^n} \to \Fun_n$ was defined to be the map classifying the left fibration $\arrtofun_n^{\Ucov^{\Delta^n}}(\ev^*\pi)$ in $\smE^{[n]}$, as expressed by the pullback \eqref{eq:arr-fun-pullback}
in $\smE^{[n]}$. In \cref{cons:fun-to-arr-n}, the map $\funtoarr_n \colon \Fun_n \to \Ucov^{\Delta^n}$ was defined as a transpose of the map classifying the left fibration $\funtoarr_n^{\Fun_n}(\psi^\bullet)$ in $\slice{\smE}{\Delta^n}$, defining a classifying pullback square as below-left, which factors as below-right:
  \begin{equation}\label{eq:fun-arr-pullback} \begin{tikzcd}
    \funtoarr_n^{\Fun_n}(\widetilde{\Fun}_n^\bullet) \arrow[d,two heads, "{\funtoarr_n^{\Fun_n}(\psi^\bullet)}"'] \arrow[r] \arrow[dr, phantom, "\lrcorner" very near start] & \EUcov \arrow[d,two heads, "\pi"] & [1em]     \funtoarr_n^{\Fun_n}(\widetilde{\Fun}_n^\bullet) \arrow[d,two heads, "{\funtoarr_n^{\Fun_n}(\psi^\bullet)}"'] \arrow[r] \arrow[dr, phantom, "\lrcorner" very near start] &[1em] \widetilde{\Arr}_n \arrow[r] \arrow[d, two heads, "\ev^*\pi"'] \arrow[dr, phantom, "\lrcorner" very near start] & \EUcov \arrow[d, two heads, "\pi"] \\  \Delta^n \times \Fun_n \arrow[r, "\funtoarr_n^\dagger"'] & \Ucov \arrow[ur, phantom, "="{xshift=-2em}] & \Delta^n\times \Fun_n \ar[r, "\Delta^n \times \funtoarr_n"'] & \Delta^n \times \Ucov^{\Delta^n} \arrow[r, "\ev"'] & \Ucov \rlap{.} \end{tikzcd}
\end{equation}

Applying \cref{lem:fun-to-arr-pullback} to \eqref{eq:arr-fun-pullback} and composing with the pullback \eqref{eq:fun-arr-pullback}, we obtain the pullback in $\slice{\smE}{\Delta^n}$ below:
\[ \begin{tikzcd} \funtoarr_n^{\Ucov^{\Delta^n}}(\arrtofun_n^{\Ucov^{\Delta^n}}(\widetilde{\Arr}_n)) \arrow[d, two heads, "{\funtoarr_n^{\Ucov^{\Delta^n}}(\arrtofun_n^{\Ucov^{\Delta^n}}(\ev^*\pi))}"'] \arrow[r] \arrow[dr, phantom, "\lrcorner" very near start] & \funtoarr_n^{\Fun_n}(\widetilde{\Fun}_n^\bullet) \arrow[d,two heads, "{\funtoarr_n^{\Fun_n}(\psi^\bullet)}"'] \arrow[r] \arrow[dr, phantom, "\lrcorner" very near start] & \widetilde{\Arr}_n \arrow[r] \arrow[d, "\ev^*\pi"', two heads] \arrow[dr, phantom, "\lrcorner" very near start] & \EUcov \arrow[d, two heads, "\pi"] \\ \Delta^n \times \Ucov^{\Delta^n} \arrow[r, "{\Delta^n \times \arrtofun_n}"'] & \Delta^n\times \Fun_n \ar[r, "\Delta^n \times \funtoarr_n"'] & \Delta^n \times \Ucov^{\Delta^n} \arrow[r, "\ev"'] & \Ucov \rlap{.}
 \end{tikzcd} \]
By \cref{lem:arr-fun-arr-equiv}, the fibration $\ev^*\pi \colon \widetilde{\Arr}_n \twoheadrightarrow \Delta^n \times \Ucov^{\Delta^n}$ is equivalent over $\Delta^n \times \Ucov^{\Delta^n}$ to the left fibration defined by this composite pullback. By \cref{rmk:ucov-univalence-application}, it follows that the classifying maps into $\Ucov$ are homotopic. This homotopy between the lower-horizontal composite and $\ev$ transposes to define a homotopy $\funtoarr_n\circ\arrtofun_n \sim \id$ proving one direction of the equivalence inverse.

For the converse, apply \cref{lem:arr-to-fun-pullback} to the middle pullback square in \eqref{eq:fun-arr-pullback} and compose with the pullback \eqref{eq:arr-fun-pullback} to obtain the following diagram in $\smE^{[n]}$:
\[
  \begin{tikzcd}[column sep=large] \arrtofun_n^{\Fun_n}(\funtoarr_n^{\Fun_n}(\widetilde{\Fun}_n^\bullet)) \arrow[r] \arrow[d, two heads, "{\arrtofun_n^{\Fun_n}(\funtoarr_n^{\Fun_n}(\psi^\bullet))}"'] \arrow[dr, phantom, "\lrcorner" very near start] & \arrtofun_n^{\Ucov^{\Delta^n}}(\widetilde{\Arr}_n) \arrow[d, two heads, "{\arrtofun_n^{\Ucov^{\Delta^n}}(\ev^*\pi)}"'] \arrow[r] \arrow[dr, phantom, "\lrcorner" very near start] & \widetilde{\Fun}_n^\bullet \arrow[d, two heads, "\psi^\bullet"] \\  \const(\Fun_n) \arrow[r, "\const(\funtoarr_n)"'] & \const(\Ucov^{\Delta^n}) \arrow[r, "{\const(\arrtofun_n)}"'] & \const(\Fun_n) \rlap{.} \end{tikzcd}\]
By \cref{lem:fun-arr-fun-equiv}, the left fibration $\psi^\bullet \colon \widetilde{\Fun}_n^\bullet \twoheadrightarrow \const(\Fun_n)$ is equivalent over $\const(\Fun_n)$ to the left fibration defined by this composite pullback. By \cref{cor:fun-univalence}, it follows that the classifying maps into $\Fun_n$ are homotopic, defining a homotopy $\arrtofun_n\circ\funtoarr_n \sim \id$.

By \cref{lem:funtoarr-htpy-naturality}, $\funtoarr_\bullet \colon \Fun_\bullet \to \Ucov^{\Delta^\bullet}$ is natural up to homotopy and in fact these homotopies satisfy a higher coherence indexed by composable pair of morphisms in the simplex category. The same is then true of the inverse equivalence $\arrtofun_\bullet \colon \Ucov^{\Delta^\bullet} \to \Fun_\bullet$.
\end{proof}

\subsection{Applications}\label{ssec:applications}

Various applications of directed univalence have been noted in other settings where results of this nature have been established \cite{KV, gratzer-weinberger-buchholtz-univalence, weaver-licata}. We collect several corollaries of \cref{thm:dua-types} to observe that these results hold at the present level of generality and exhibit the general utility of directed univalence theorems.

\cref{thm:dua-types} identifies arrows in the covariant universe with functors between the types it classifies. This allows us to identify the ``covariant transport'' operation in the universal left fibration with functor application as we now observe.

\begin{rmk}\label{rmk:transport-as-evaluation}
 For a left map $f \colon Y \to X$, $Y^{\Delta^1}$ is equivalent to the pullback $X^{\Delta^1}\times_X Y$, which is to say that arrows in the total space are uniquely determined by arrows in the base plus a lift of their domain. Accordingly, we refer to the map $\ev_1 \colon Y^{\Delta^1} \to Y$ as ``covariant transport'' as it sends an arrow in $X$ together with a lift of its domain to $Y$ to a lift of its codomain to $Y$. Directed univalence provides an equivalent description of the covariant transport operation for the universal left fibration:
\[
\begin{tikzcd}
  \EUcov \arrow[d, "\pi"']&  \EUcov^{\Delta^1} \arrow[d, "\pi^{\Delta^1}"]
  \arrow[l, "\ev_0"'] \arrow[dl, phantom, "\llcorner" very near start] \arrow[r, "\ev_1"] & \EUcov \arrow[d, "\pi"] \\ \Ucov & \Ucov^{\Delta^1} \arrow[l, "\ev_0"] \arrow[r, "\ev_1"'] & \Ucov \rlap{.} \end{tikzcd}
  \]

By \cref{thm:dua-types}, the bottom span is equivalently $(s,t) \colon \Fun_1 \to \Fun_0$, which we unpack further using the construction of \cref{lem:map-internal-category}. After doing so, we see that the diagram above coincides with the diagram on the right of \eqref{eq:internal-cat-evaluation} reproduced below:
\[
\begin{tikzcd}
  \EUcov \arrow[d, "\pi"']&  \EUcov \underset{\Ucov}{\times} {\hom_{\Ucov \times \Ucov}(\EUcov \times \Ucov, \Ucov \times \EUcov)}  \arrow[d, "\pi^{\Delta^1}"]
  \arrow[l, "\pi_{\mathup{left}}"'] \arrow[dl, phantom, "\llcorner" very near start] \arrow[r, "\ev"] & \EUcov \arrow[d, "\pi"] \\ \Ucov & {\hom_{\Ucov \times \Ucov}(\EUcov \times \Ucov, \Ucov \times \EUcov)} \arrow[l, "s"] \arrow[r, "t"'] & \Ucov \rlap{.} \end{tikzcd}
\]
Thus the covariant transport map $\ev_1\colon \EUcov^{\Delta^1} \to \EUcov$ is identified with the evaluation map \[\ev \colon \EUcov \underset{\Ucov}{\times} \Fun_1 \to \Ucov\] that can be described very explicitly.

  We denote elements of $\EUcov$ by $(A,a)$, where $A$ is a code for a left fibration classified by $\Ucov$ and $a$ is an section of the decoded map
  \[ \begin{tikzcd} A \arrow[d, two heads] \arrow[r] \arrow[dr, phantom, "\lrcorner" very near start] & \EUcov \arrow[d, two heads, "\pi"] \\ \Gamma \arrow[r, "A"'] \arrow[ur, dashed, "{(A,a)}" description] \arrow[u, bend left, dashed, "a"] & \Ucov \rlap{.} \end{tikzcd}\] The evaluation map takes an element $(A,a)$ and a function $f \colon A \to B$ over $\Gamma$ to the element $(B, fa)$.
\end{rmk}

Directed univalence can also be used to construct examples of what Martini--Wolf call \emph{internal $\infty$-categories} in the $\infty$-topos $\mE$ \cite{martini}. These are simplicial objects satisfying the Segal and Rezk conditions of the simplicial type theory \cite{RS}. To state these definitions, we require a model for the homotopy coherent isomorphism in the category of simplicial sets, such as the following:

\begin{defn}
Consider the simplicial sets defined by the left and center colimit diagrams:
\[
\begin{tikzcd}
  &
  {\Delta^1 + \Delta^1}
  \arrow[d, "{\delta^1 + \delta^1}"']
  \arrow[r, "{\sigma^0 + \sigma^0}"] & {\Delta^0 + \Delta^0} \arrow[dd, dashed] \\
{\Delta^1 + \Delta^1}
\arrow[d, "{\langle\id,\id\rangle}"']
\arrow[r, "{\delta^0 + \delta^2}"] & {\Delta^2 + \Delta^2} \\
\Delta^1 \arrow[rr, dashed] & & I' \rlap{,}
\end{tikzcd} \qquad
 \begin{tikzcd}
  {\Delta^1 + \Delta^1} \arrow[dr, phantom, "\ulcorner" very near end]
  \arrow[d, "{\langle \delta^{0,2}, \delta^{1,3}\rangle}"']
  \arrow[r, "{\sigma^0 + \sigma^0}"] & {\Delta^0 + \Delta^0} \arrow[d, dashed] \\ \Delta^3 \arrow[r, dashed, "\alpha"'] & I \rlap{,}
\end{tikzcd} \qquad
\begin{tikzcd}
  {\Delta^2 \cup_{\Delta^1}\Delta^2} \arrow[d, tail, "{\langle \delta^3,\delta^0\rangle}"'] \arrow[r] \arrow[dr, phantom, "\ulcorner" very near end] & I' \arrow[d, tail] \\ \Delta^3 \arrow[r, dashed] & I \rlap{.}
\end{tikzcd}
\]
The simplicial set $I'$ might be thought of as the ``free-living bi-invertible arrow,'' the 1-simplex in the image of the lower dashed map $\Delta^1 \to I'$ that is equipped with left and right inverses witnessed by the pair of 2-simplices. The simplicial set $I$ is the ``free-living 2-of-6 diagram,'' extending the free-living bi-invertible arrow by an associativity coherence. There is a natural inclusion from the former to the latter that is an inner anodyne extension.
\end{defn}

\begin{defn}\label{defn:segal-rezk}
A fibrant object $C \in \smE$ is an \textbf{internal $\infty$-category} if it satisfies the following conditions:
\begin{itemize}
  \item \textbf{Segal condition}: the map $C^{\Delta^2} \to C^{\Lambda^2_1}$ defined by restricting along the inclusion of discretely embedded simplicial sets is an equivalence.
  \item \textbf{Rezk condition}: the map $C \to C^{I}$ defined by restricting along the projection $! \colon I \to \Delta^0$ of discretely embedded simplicial sets is an equivalence.\footnote{We typically only consider the Rezk condition for objects also satisfying the Segal condition, from which point of view the models $I$ and $I'$ for the coherent isomorphism are equivalent.}
\end{itemize}
A fibrant object $C \in \smE$ is an \textbf{internal $\infty$-groupoid} if the map $C^{\Delta^0} \to C^{\Delta^1}$ defined by restricting along the projection of discretely embedded simplicial sets is an equivalence.
\end{defn}

\begin{rmk}
The conditions of \cref{defn:segal-rezk} are all internal orthogonality conditions between the fibration $! \colon C \to 1$ and various maps of simplicial sets. Similarly, we say that a fibration $f \colon Y \to X$ in $\smE$ is \textbf{relatively Segal} or \textbf{relatively Rezk} if it is internally orthogonal to the maps $\iota\colon \Lambda^2_1 \to \Delta^2$ or $! \colon I \to \Delta^0$ respectively. By pullback cancellation, if the codomain of a relatively Segal or Rezk map satisfies the Segal or Rezk conditions, so does the domain.
\end{rmk}

\begin{lem}\label{lem:left-implies-segal-rezk}
Left fibrations are relatively Segal and relatively Rezk. Thus, if the base of a left fibration is an internal $\infty$-category, then so is the total space.
\end{lem}
\begin{proof}
The argument is given as the ``categorical proof'' of \cite[8.8]{RS} or in \cite[4.1.4]{martini}. Recall that left maps are internally orthogonal to the map $\iota_0 \colon \Delta^0 \to \Delta^1$. We show that the maps $\iota\colon \Lambda^2_1 \to \Delta^2$ and $!\colon I \to \Delta^0$ belong to a suitably defined ``saturation'' of the map $\iota_0 \colon \Delta^0 \to \Delta^1$ defined relative to the notion of internal orthogonality. If a map $f \colon Y \to X$ in a cartesian closed $\infty$-category is internally orthogonal to some set of maps then it is also internally orthogonal to maps obtained in a larger class by adding the equivalences and closing up under composition, pushout, colimits in the arrow category, right cancellation, and Leibniz product with any map. It follows that the class containing $\iota_0 \colon \Delta^0 \to \Delta^1$ also contains
\begin{itemize}
  \item $! \colon \Delta^1 \to \Delta^0$, from right cancellation involving an isomorphism,
  \item $\alpha \colon \Delta^3 \to I$, from closure under pushouts,
  \item $\iota_0 \colon \Delta^0 \to \Delta^{n+1}$ for any $n$, as a retract of $\Delta^n \times \iota_0 \colon \Delta^n \to \Delta^n \times \Delta^1$,
  \item $\iota_0 \colon \Delta^0 \to I$ by composition,
  \item $! \colon I \to \Delta^0$ by right cancellation involving an isomorphism,
  \item $\delta^{0,1} \colon \Delta^1 \to \Lambda^2_1$, as a pushout of $\iota_0 \colon \Delta^0 \to \Delta^1$,
  \item $\iota_0 \colon \Delta^0 \to \Lambda^2_1$, by composition,
  \item $\iota \colon \Lambda^2_1 \to \Delta^2$, by right cancellation. \qedhere
\end{itemize}
\end{proof}

\begin{rmk}
A fibrant object $C \in \smE$ defines an internal $\infty$-groupoid if and only if it is internally orthogonal to either and hence both of the sections $\iota_0, \iota_1 \colon \Delta^0 \to \Delta^1$ of the projection $! \colon \Delta^1 \to \Delta^0$. In particular, the map $! \colon C \twoheadrightarrow 1$ to the terminal object is a left fibration if and only if it is a right fibration if and only if $C$ is an internal $\infty$-groupoid. Thus the points of $\Ucov$ are internal $\infty$-groupoids.

Note that internal $\infty$-groupoids are internal $\infty$-categories, satisfying both the Segal and Rezk conditions \cite[7.3, 10.10]{RS}.
\end{rmk}

As a first example, following \cite{KV} and others, we use higher directed univalence to prove that $\Ucov$ is an internal $\infty$-category, which we refer to as the internal $\infty$-category of internal $\infty$-groupoids.

\begin{cor}\label{cor:Ucov-category}
  The base $\Ucov$ of the universal left fibration is an internal $\infty$-category.
\end{cor}
\begin{proof}
We must verify the Segal and Rezk conditions. For the Segal condition, we must show that the following square is a pullback in the $\infty$-categorical sense:
  \[ \begin{tikzcd} \Ucov^{\Delta^2} \arrow[r, "-\circ \delta^2"] \arrow[d, "-\circ \delta^0"'] & \Ucov^{\Delta^1} \arrow[d, "-\circ \delta^0"] \\ \Ucov^{\Delta^1} \arrow[r, "-\circ \delta^1"'] & \Ucov^{\Delta^0} \rlap{.} \end{tikzcd} \]
By \cref{thm:dua-types,lem:funtoarr-htpy-naturality}, this is equivalent to the square
  \[ \begin{tikzcd} \Fun_2 \arrow[r, "{(\delta^2)^*}"] \arrow[dr, phantom, "\lrcorner" very near start] \arrow[d, "{(\delta^0)^*}"'] & \Fun_1 \arrow[d, "{(\delta^0)^*}"] \\ \Fun_1 \arrow[r, "{(\delta^1)^*}"'] & \Fun_0 \rlap{,} \end{tikzcd} \] which is a strict pullback as $\Fun_\bullet$ is constructed as an internal category object.

For the Rezk condition, we must show that the following are pullback diagrams in the $\infty$-categorical sense, where \cref{thm:dua-types,lem:funtoarr-htpy-naturality} give an equivalence between the left-hand and right-hand diagrams:
\[ \begin{tikzcd}[column sep=large] \Ucov \arrow[d, "{\langle \id,\id\rangle}"', dashed] \arrow[r, dashed]  \arrow[dr, phantom, "\lrcorner" very near start] & \Ucov^{\Delta^3}\arrow[d, "{\langle -\circ \delta^{0,2}, -\circ \delta^{1,3}\rangle}"] & \Fun_0\arrow[dr, phantom, "\lrcorner" very near start] \arrow[d, "{\langle \id,\id\rangle}"', dashed] \arrow[r, dashed] & \Fun_3 \arrow[d, "{\langle (\delta^{0,2})^*, (\delta^{1,3})^*\rangle}"] \\ \Ucov \times \Ucov \arrow[r, "{-\circ \sigma^0\times -\circ\sigma^0}"'] & \Ucov^{\Delta^1} \times \Ucov^{\Delta^1} \rlap{,} & \Fun_0 \times \Fun_0 \arrow[r, "{(\sigma^0)^* \times (\sigma^0)^*}"'] & \Fun_1 \times \Fun_1 \rlap{.} \end{tikzcd}\]
The limit in the right-hand diagram defines the universal object $\mathsf{Binv}(\EUcov)$ of bi-invertible maps between left fibrations classified by $\Ucov$, which is equivalent to the universal equivalence $\Eq(\EUcov)$, while the map from $\Ucov$ picks out the identity biequivalence. By univalence of $\Ucov$, this is an equivalence, as desired.
\end{proof}

A useful consequence of univalence is the structure identity principle, characterizing path
spaces of various classifying spaces built from the universe $\Ufib$. Similarly, directed univalence implies a directed
analog \cite[\S7.3.2]{weaver:24}~\cite[\S7.3]{gratzer-weinberger-buchholtz-univalence}, the \textbf{structure homomorphism principle}, identifying the arrows in various classifying categories built from the universe $\Ucov$ as homomorphisms of the appropriate structure.

\begin{cor}\label{cor:EUcov-category}
The object $\EUcov$ is an internal $\infty$-category of pointed internal $\infty$-groupoids whose arrows are pointed functions.
\end{cor}
\begin{proof}
By \cref{lem:left-implies-segal-rezk,cor:Ucov-category}, $\EUcov$ is an internal $\infty$-category.

  Recall that we denote elements of $\EUcov$ by $(A,a)$, where $A$ is a code for a left fibration classified by $\Ucov$ and $a$ is an section of the decoded map
  \[ \begin{tikzcd} A \arrow[d, two heads] \arrow[r] \arrow[dr, phantom, "\lrcorner" very near start] & \EUcov \arrow[d, two heads, "\pi"] \\ \Gamma \arrow[r, "A"'] \arrow[ur, dashed, "{(A,a)}" description] \arrow[u, bend left, dashed, "a"] & \Ucov \rlap{.} \end{tikzcd}\]
  In the case of global elements $A$ is an internal $\infty$-groupoid and $a$ is a global element of $A$.

We next characterize the arrows in $\EUcov$. The left fibration provides a pullback displayed below-left:
\[
\begin{tikzcd}[column sep=large]
  \EUcov^{\Delta^1} \arrow[dd, bend right=60, "\ev_0"'] \arrow[d, "{(\ev_0, \pi \ev_1)}" description]\arrow[r, "\pi^{\Delta^1}"] \arrow[dr, phantom, "\lrcorner" very near start] &
  \Ucov^{\Delta^1} \arrow[d, "{(\ev_0,\ev_1)}" description]\arrow[dd, bend left=60, "\ev_0"] & & \Gamma \arrow[r, dotted] \arrow[dr, bend right=10, dashed, "{\langle(A,a), (B,b)\rangle}"' pos=.9] \arrow[ddr, bend right=40, dashed, "{\langle(A,a),B\rangle}"'] \arrow[rr, bend left=20, dashed, "f"]  & \EUcov^{\Delta^1} \arrow[d, "{(\ev_0, \ev_1)}"' pos=.4] \arrow[r] \arrow[dr, phantom, "\lrcorner" very near start] & \Fun_1 \arrow[dd, "{(\ev_0,\ev_1)}"] \\
\EUcov \times \Ucov \arrow[r] \arrow[d, "\pi_{\mathup{left}}"'] \arrow[dr, phantom, "\lrcorner" very near start] & \Ucov \times \Ucov \arrow[d, "\pi_{\mathup{left}}"] & \rightsquigarrow & & \EUcov \times \EUcov  \arrow[d, "1 \times \pi"'] & ~ \\ \EUcov \arrow[r, "\pi"'] & \Ucov &~ & &  \EUcov \times \Ucov \arrow[r, "\pi \times 1"'] & \Ucov \times \Ucov \rlap{.} \end{tikzcd}
\]
By directed univalence, arrows $\Gamma \to \EUcov^{\Delta^1}$ from $(A,a)$ to $(B,b)$ are given by the data of a function $f \colon A \to B$ between left fibrations over $\Gamma$ and a section $a \colon \Gamma \to A$. By \cref{rmk:transport-as-evaluation}, commutativity of the triangle involving the map $(\ev_0,\ev_1) \colon \EUcov^{\Delta^1} \to \EUcov \times \EUcov$ implies that $fa = b$.
\end{proof}

Another instance of the structure homomorphism principle is illustrated by the internal $\infty$-category of \textbf{magmas}: internal $\infty$-groupoids equipped with a specified binary operation. This result is illustrative of a family of similar examples whose complete description we leave for future work.

\begin{cor}\label{cor:}
  There is an internal $\infty$-category of magmas whose arrows are magma homomorphisms.
\end{cor}
\begin{proof}
The left fibration $\pi \times \pi \colon \EUcov \times \EUcov \twoheadrightarrow \Ucov \times \Ucov$ is classified by a map $\times \colon \Ucov \times \Ucov \to \Ucov$. The universe of magmas $\Magma \coloneqq \Sigma_{A : \Ucov} A \times A \to A$ is defined by the pullback
\[ \begin{tikzcd} \Magma \arrow[d, two heads] \arrow[r] \arrow[dr, phantom, "\lrcorner" very near start] & \Fun_1 \arrow[d, "{\langle \ev_0,\ev_1\rangle}", two heads] \arrow[r, phantom, "\simeq"] &  [-1em] \Ucov^{\Delta^1} \arrow[d, "{\langle \ev_0,\ev_1\rangle}", two heads] \\ \Ucov \arrow[r, "{\langle\times \circ \Delta, \id\rangle}"'] & \Ucov \times \Ucov \arrow[r, equals] & \Ucov \times \Ucov\end{tikzcd}\]
with the equivalent pullback by directed univalence. The right-hand vertical map is a Leibniz exponential of $\partial\Delta^1 \to \Delta^1$ with a map $! \colon \Ucov \to 1$ that is relatively Segal and relatively Rezk, and hence this map is relatively Segal and relatively Rezk. Thus the pullback $\Magma \to \Ucov$ is relatively Segal and Rezk. Since $\Ucov$ is an internal $\infty$-category, so is $\Magma$.

We denote an element $(A,\alpha) \colon \Gamma \to \Magma$ as a pair, where $A$ is a code for a left fibration $A \twoheadrightarrow \Gamma$ and $\alpha$ is a binary function $\alpha \colon A \times A \to A$ over $\Gamma$.

We characterize arrows in $\Magma$ by exponentiating the defining pullback by $\Delta^1$ to yield:
\[ \begin{tikzcd} \Magma^{\Delta^1} \arrow[d, two heads] \arrow[r] \arrow[dr, phantom, "\lrcorner" very near start] & \Ucov^{\Delta^1 \times \Delta^1}\arrow[d, "{\langle \ev_{0-},\ev_{1-}\rangle}", two heads] \arrow[r, phantom, "\simeq"] & \Fun_2 \times_{\Fun_1} \Fun_2 \arrow[d, "{\langle \ev_{0-},\ev_{1-}\rangle}", two heads] \\ \Ucov^{\Delta^1} \arrow[r, "{\langle(\times \circ \Delta)^{\Delta^1}, \id\rangle}"'] & \Ucov^{\Delta^1} \times \Ucov^{\Delta^1} \arrow[r, phantom, "\simeq"] & \Fun_1 \times \Fun_1 \rlap{.} \end{tikzcd}\]
The bottom horizontal composite is equivalently the map $\Fun_1 \to \Fun_1 \times \Fun_1$ that sends an arrow $f \colon A \to B$ to the pair of arrows $f \times f \colon A \times A \to B \times B$ and $f \colon A \to B$. Thus, the top horizontal composite sends an arrow in $\Magma$ from $(A,\alpha)$ to $(B,\beta)$ to the commutative square of functions
\[ \begin{tikzcd} A \times A \arrow[d, "\alpha"'] \arrow[r, "f \times f"] & B \times B \arrow[d, "\beta"] \\ A \arrow[r, "f"'] & B \rlap{.} \end{tikzcd} \]
Thus, the pullback characterizes arrows in $\Magma$ as magma homomorphisms.
\end{proof}

\printbibliography

\end{document}